\documentclass[11pt,reqno]{amsart}
\usepackage{fullpage}

\usepackage{graphicx}
\usepackage{hyperref}
\usepackage{amsmath,amsopn,amssymb,amsfonts,stmaryrd}
\usepackage{verbatim}
\usepackage{amsthm}
\usepackage{mathtools}
\usepackage{color}
\usepackage{enumitem}
\usepackage[framemethod=TikZ]{mdframed}
\usepackage{bbm}
\usepackage{mathrsfs}
\usepackage{booktabs}
\usepackage{caption}
\usepackage{bm}
\usepackage{tensor}
\usepackage{cancel}

\usepackage{mathrsfs}
\usepackage[scr=boondox]{mathalfa}

\newcommand{\IR}{{\mathbb R}}%Reals
\newcommand{\IC}{{\mathbb C}}%Complex
\newcommand{\C}{\mathbb{C}}
\newcommand{\Z}{{\mathbb Z}}%Integers
\newcommand{\N}{{\mathbb N}}%Natural numbers
\newcommand{\IQ}{{\mathbb Q}}%Rationals
\newcommand{\IH}{{\mathbb H}}%quaternions
\renewcommand{\H}{\mathbb{H}}
\newcommand{\CA}{\mathscr{A}}

\newcommand{\CD}{\mathscr{D}}

\newcommand{\CF}{\mathcal{F}}

\newcommand{\cd}{\mathscr{d}}

\newcommand{\sgn}{\mbox{sgn}}

\newcommand{\SU}{\mathrm{SU}}

\newcommand{\SL}{\mathrm{SL}}

\newcommand{\zz}{\mathfrak{z}}

\theoremstyle{plain}
\newtheorem{thm}{Theorem}[section]

\newtheorem{lem}[thm]{Lemma}
\newtheorem{lemma}[thm]{Lemma}
\newtheorem{prop}[thm]{Proposition}

\theoremstyle{definition}

\newtheorem{rem}[thm]{Remark}

\numberwithin{equation}{section}

\newcommand{\pmat}[1]{\left( \smallmatrix #1 \endsmallmatrix \right)}
\newcommand{\mat}[1]{\left( \begin{matrix} #1 \end{matrix} \right)}

\renewcommand{\sgn}{\textnormal{sgn}}

\def\lp{\left(}
\def\rp{\right)}
\def\lb{\left[}
\def\rb{\right]}

\def\a{\alpha}
\def\b{\beta}
\def\d{\delta}

\def\l{\lambda}

\def\w{\omega}

\def\vth{\vartheta}

\def\e{\varepsilon}

\def\s{\sigma}
\def\g{\gamma}
\def\t{\tau}

\def\DD{\Delta}
\def\LL{\Lambda}

\def\del{  \partial}

\newcommand{\ch}{\operatorname{ch}}

\newcommand{\y}{{\rm \mathbbm{y}}}
\newcommand{\zzb}{{\rm \mathbbm{z}}}
\newcommand{\mmb}{{\rm \mathbbm{m}}}

\renewcommand{\sgn}{{\rm sgn}}

\def\wt{\widetilde}
\def\wh{\widehat}
\def\bar{\overline}
\newcommand{\andd}{\quad \mbox{ and } \quad}
\newcommand{\where}{\quad \mbox{ where }}

\setlist[itemize]{noitemsep, topsep=0pt}

\makeatletter
\newcommand{\vast}{\bBigg@{2}}
\newcommand{\Vast}{\bBigg@{5}}
\makeatother

\renewcommand{\pmod}[1]{\    \left(  \mathrm{mod} \,  #1 \right)}
\newcommand{\erf}{{\mathrm{erf}}}

\newcommand{\rint}[2]{\tensor*[_{\ *\!\!\!}]{\int}{_{#1}^{#2}}}
\newcommand{\IZ}{{\mathbb{Z}}}
\newcommand{\MM}{\mathfrak{M}}
\newcommand{\FF}{\mathfrak{F}}

\makeatletter
\newcommand{\nolisttopbreak}{\par\nobreak\@afterheading} 
\makeatother

\DeclareRobustCommand{\SkipTocEntry}[5]{}

\title{Higher Depth False Modular Forms}
\author{Kathrin Bringmann}
\address{University of Cologne, Department of Mathematics and Computer Science, Weyertal 86-90, 50931 Cologne, Germany}
\email{kbringma@math.uni-koeln.de}

\author{Jonas Kaszian}
\address{Max Planck Institute for Mathematics, Vivatsgasse 7, 53111 Bonn, Germany}
\email{jonask@mpim-bonn.mpg.de}

\author{Antun Milas}
\address{Department of Mathematics, SUNY-Albany, Albany NY 12222}
\email{amilas@albany.edu}

\author{Caner Nazaroglu}
\address{University of Cologne, Department of Mathematics and Computer Science, Weyertal 86-90, 50931 Cologne, Germany}
\email{cnazarog@math.uni-koeln.de}

\thanks{The first and the fourth author are supported by the Deutsche Forschungsgemeinschaft (DFG) Grant No. BR 4082/5-1. The third author was partially supported by the NSF grant 2101844 and the Simons Collaboration Grant for Mathematicians 709563. The research of the fourth author is  supported by the SFB/TRR 191 ``Symplectic Structures in Geometry, Algebra and Dynamics", funded by the DFG (Projektnummer 281071066 TRR 191).}

\begin{document}

\begin{abstract}
False theta functions are functions that are closely related to classical theta functions and mock theta functions. In this paper, we study their modular properties at all ranks by forming modular completions analogous to modular completions of indefinite theta functions of any signature and thereby develop a structure parallel to the recently developed theory of higher depth mock modular forms.  
 We then demonstrate this theoretical base on a number of examples up to depth three coming from characters of modules for the vertex algebra $W^0(p)_{A_n}$, $1 \leq n \leq 3$, 
 and from $\hat{Z}$-invariants of $3$-manifolds associated with gauge group $\SU(3)$.
\end{abstract}

\maketitle

\tableofcontents

\section{Introduction and Statement of Results}\label{sec:intro}
False theta functions are similar to theta functions, which are classical examples of modular forms and Jacobi forms; they differ, however, by the inclusion of certain sign factors that break their modular symmetry. False theta functions have appeared in several contexts including representation theory, quantum topology, and theoretical physics. 
In representation theory, they occur essentially as characters of modules for the $(1,p)$-singlet algebras and of higher rank $W$-algebras $W^0(p)_L$ \cite{BKM1,BM2, Creutzig0,CM2}. They have been also identified within character formulas of parafermionic vertex algebras coming from Fourier coefficients of negative index Jacobi forms \cite{BKMN, BKMZ}. In quantum topology, false theta functions capture unified WRT invariants of Seifert $3$-manifolds \cite{GM, Hi2} and also occur in the computation of tails of colored Jones polynomials of alternating knots \cite{BO, Garo, KO, Yuasa}. Very recently, they prominently featured in physics through the work of Gukov, Pei, Putrov, and Vafa on homological blocks of $3$-manifolds $\hat{Z}_{M, \bm a}(q)$ in Chern--Simons theory \cite{GPPV}. These remarkable $q$-series, called \textit{$\hat{Z}$-invariants}, can be expressed as false theta functions (at least if $M$ is a plumbed $3$-manifold). The $\hat{Z}$-invariants,  their refinements \cite{GM}, and higher rank generalizations \cite{Park}, are conjecturally (higher depth) quantum modular forms \cite{BMM,CCFGH}.

In the simplest case of false theta functions with a single sign function, the modular transformation properties can be simply deduced by realizing these functions as holomorphic Eichler integrals of unary theta functions. For example, the unary false theta function
\begin{equation*}
\psi (\t) := i \sum_{n \in \IZ} \sgn \lp n + \frac{1}{2} \rp (-1)^n q^{\frac{1}{2} \lp n + \frac{1}{2} \rp^2},
\end{equation*}
with $\sgn(x):=\frac{x}{\lvert x\rvert}$ for $x\neq 0$ and $\sgn(0):=0$,
can be realized as the holomorphic Eichler integral
\begin{equation*}
\psi (\t) = i \int_\t^{\t + i \infty} \frac{\eta (w)^3}{\sqrt{i (w-\t)}} d w,
\end{equation*}
where $\eta (\t) := q^\frac{1}{24} \prod_{n \geq 1} (1-q^n)$ is the \textit{Dedekind eta-function} and the integration path is appropriately chosen in view of the double-valuedness of the square-root (see e.g.~\cite{BN} for more details). This realization immediately tells us that $\psi (\t)$ transforms like a modular form up to suitably defined cycle integrals of the cusp form $\eta (w)^3$, which can be realized as a classical unary theta function. Such classical considerations are sufficient for most purposes involving unary false theta functions. For example, the modular properties used to obtain Rademacher-type exact formulae for Fourier coefficients of $\frac{\psi (\t)}{\eta (\t)^2}$ derived in \cite{BN} can be understood from this picture. Similarly, these considerations imply that the holomorphic Eichler integral representation of $\psi (\t)$ makes sense not only on $\IH$ but also on $\IQ$ (since cusp forms are exponentially decaying towards each rational point). This shows that we can define a function $f:\mathbb{Q}\to \IC$ by taking the limit
\begin{equation*}
f(x) := \lim_{t \to 0^+} \psi (x+it) 
\end{equation*}
and it also transforms like a modular form up to the aforementioned cycle integrals of $\eta^3$. This is a remarkable fact since $f$ does not extend to any continuous function on the reals, whereas the cycle integral that gives its {obstruction to modularity} yields a real-analytic function on $\IR$ except for one point. Such functions on rational numbers with more ``regular" obstructions to modularity (including vector-valued generalizations with suitable multiplier systems) are called \emph{quantum modular forms} \cite{Zagier}.\footnote{
One can also obtain the function $f: \IQ \to \IC$ through rational limits of the non-holomorphic Eichler integral
\begin{equation*}
\int_{-\bar{\t}}^{i \infty} \frac{\eta (w)^3}{\sqrt{-i (w+\t)}} d w .
\end{equation*}
As reviewed below, mock modular forms are holomorphic functions on $\IH$ which form real-analytic modular forms upon addition of such integrals. For example, the particular non-holomorphic Eichler integral above for $\eta^3$ is related to the mock modular form with Fourier coefficients 
\begin{equation*}
2 q^{-\frac{1}{8}} \lp -1 + 45 q + 231 q^2 + 770 q^3 + 2277 q^4 
+ 5796 q^5 + 13915 q^6 + \ldots \rp
\end{equation*}
that led to the observation of Mathieu moonshine \cite{EOT}.
This point of view provides an alternative approach to the quantum modular properties that can be pursued by studying mock modular forms replicating the asymptotic behavior of false theta functions near rational numbers 
(this is an observation due to Zwegers - see also \cite{LZ}).
}

The results of \cite{BN} offer an alternative repackaging of false theta functions where their modular properties can be studied in a way similar to indefinite theta functions by constructing modular completions (these are modular covariant functions which reproduce these false or indefinite theta functions at their boundary - see below for more details). 
This point of view is particularly suitable for generalizations to false theta functions with more factors of sign functions. That is precisely the strategy we follow in this paper to formulate a general framework, where such functions can be studied through modular completions constructed using generalized error functions defined by \cite{ABMP}. Our results also point out how the holomorphic Eichler integral representations above generalize to this more general setting (by using Lemma \ref{sign.lemma}, which is in turn naturally derived by studying properties of generalized error functions). It turns out that a false theta function with $r$ factors of sign functions can be expressed as a linear combination of iterated integrals of modular forms.
For example, it was found in Proposition 3.2 of \cite{BKMN} that the general $r=2$ false theta function\footnote{Throughout this article, we denote vectors with bold letters.}
\begin{equation*}
\Omega (\t) := 
\sum_{\bm{n} \in \IZ^2 + \bm{\a}}
\sgn (n_1) \sgn (n_2) q^{\frac{1}{2} (an_1^2+2bn_1n_2+cn_2^2)},
\end{equation*}
where $\bm{\a} \in \IQ$ and $a,b,c$ are integers giving a positive definite quadratic form in the exponent, can be expressed as (for $\bm{\a} \not \in \IZ^2$ and $\DD := ac-b^2$)
\begin{equation*}
\Omega (\t) =
\sqrt{\DD} \int_\t^{\t + i \infty}  \int_\t^{\w_1}
\frac{\Theta_1 (\bm{w}) + \Theta_2 (\bm{w})}{\sqrt{i(w_1-\t)} \sqrt{i (w_2 - \t)}} dw_2 dw_1
-\frac{2}{\pi} \arctan \lp \frac{b}{\sqrt{\DD}} \rp \Theta (\t),
\end{equation*}
where we define the ordinary theta function
\begin{align*}
\Theta (\t)
&:= 
\sum_{\bm{n} \in \IZ^2 + \bm{\a}}  q^{\frac{1}{2} \left(an_1^2+2bn_1n_2+cn_2^2\right)}
\end{align*}
and
\begin{align*}
\Theta_1 (\bm{w})
&:= 
\sum_{\bm{n} \in \IZ^2 + \bm{\a}} n_1 \lp n_2 + \frac{b}{c} n_1 \rp 
e^{\pi i \frac{\DD}{c} n_1^2 w_1 + \pi i c \lp n_2 + \frac{b}{c} n_1 \rp^2 w_2},
\\
\Theta_2 (\bm{w})
&:= 
\sum_{\bm{n} \in \IZ^2 + \bm{\a}} n_2 \lp n_1 + \frac{b}{a} n_2 \rp 
e^{\pi i \frac{\DD}{a} n_2^2 w_1 + \pi i a \lp n_1 + \frac{b}{a} n_2 \rp^2 w_2},
\end{align*}
which can be written in terms of weight $\frac{3}{2}$ unary theta functions.
However for general $r$, the number of such Eichler-type iterated integrals is at least $r!$ and hence quickly proliferates with increasing $r$. So although this simple approach to unary false theta functions for $r=1$ leads to a very managable generalization for $r=2$ as in \cite{BKMN}, it quickly becomes cumbersome for larger values of $r$. 
The approach presented here in terms of modular completions, on the other hand, provides a compact and straightforward way to package and study these functions. Furthermore, to deal with additional polynomial insertions, we derive a set of modular covariant differential operators acting on the modular completion. These operators automatically produce all the modular covariant terms that have to be introduced in an ad-hoc fashion in an approach based on  Lemma \ref{sign.lemma}. All in all, these points suggest that the approach we assume here based on modular completions (following \cite{BN}) is the natural framework for false theta functions at any rank and with arbitrary insertions of sign functions and polynomials.

To understand these completions, we first recall mock modular forms, which are themselves functions that have well-behaved obstructions to modularity. They were first introduced by Ramanujan in the form of 17 examples of the so-called mock theta functions. Their modular properties however remained largely mysterious until the seminal work of Zwegers \cite{Zwegers}, which described mock modular forms as the holomorphic parts of certain real-analytic functions (called their {\it modular completion}) that transform like modular forms. More specifically, a {\it (mixed) mock modular form} $h$ of weight $k$ is a holomorphic function on $\H$, which is the holomorphic part of a real-analytic function $\wh h(\t,\bar\t)$ that transforms like a weight $k$ modular form and that satisfies (see \cite{DMZ})\footnote{The ``hats" used to denote modular completions should not be confused with the hats appearing in $\hat Z$-invariants; especially because the latter can itself have modular completions. We use smaller hats for homological blocks to make the distinction.}
\begin{equation*}
	\frac{\del   }{\del \bar{\tau}}\wh{h} \in \bigoplus_j  \t_2^{r_j} \MM_{k + r_j} \otimes \bar{\MM_{2+r_j}},
\end{equation*}
where $\tau_2:={\rm Im}(\tau)$, and $\MM_{r}$ denotes the space of weight $r$ (weakly holomorphic) modular forms  (possibly on a subgroup of $\mathrm{SL}_2 (\IZ)$ and/or with a vector-valued multiplier system).
An important class of functions that fit into this framework are indefinite theta functions on signature $(n-1,1)$ lattices \cite{Zwegers}. Just like false theta functions, these indefinite theta functions include various sign factors, which break the modular properties expected from ordinary theta functions, but are also essential for the convergence of the associated indefinite theta functions.

These results on indefinite theta functions of signature $(n-1,1)$ were then generalized to construct modular completions for indefinite theta functions of arbitrary signature \cite{ABMP, FK, Kudla, Na}, showing that they yield higher depth mock modular forms as defined in the unpublished work of Zagier and Zwegers.\footnote{Note that the definition of Zagier--Zwegers first appeared in literature in \cite{Raum}, where ``product type'' indefinite theta functions were investigated.} The space $\MM_k^d$ of mock modular forms with weight $k$ and depth $d$, to which these indefinite theta functions fit, as well as the space of their modular completions $\wh\MM_k^d$ can be defined recursively starting with those modular forms at depth zero, which are nothing but ordinary modular forms (so that $\MM_k^0=\wh\MM_k^0=\MM_k$). More precisely, a {\it mock modular form} $h:\H\to\C$ of {\it weight} $k$ and {\it depth} $d\ge1$ is a holomorphic function, which is the holomorphic part of a real-analytic function $\wh h(\t,\bar\t)$ transforming like a weight $k$ modular form and satisfying 
\begin{equation*}
	\frac{\del   }{\del \bar{\tau}}\wh{h} \in \bigoplus_j  \t_2^{r_j} \wh{\MM}^{d-1}_{k + r_j}  \otimes \bar{\MM_{2+r_j}}  .
\end{equation*}

The example of mock modular forms above combined with the approach of \cite{BN} suggests similarly building vector spaces of \textit{false modular forms} of weight $k$ and depth $d$, which we denote by $\FF^d_k$, and their modular completions, which we denote by $\wh{\FF}^d_k$ (so that $\FF_k^0 = \wh{\FF}_k^0 = \MM_k$). An element of $\FF^d_k$ is a holomorphic function $f$ on $\IH$ with an associated modular completion $\wh{f}\in \wh{\FF}^d_k$, where  $(i (w-\t) )^{\frac{\ell}{2}} \wh{f} (\tau,w)$ is a holomorphic function for $(\tau,w) \in \IH \times \IH \setminus \{ \tau =w \}$ and for some $\ell \in \{0,1\}$.\footnote{Thereby, for $\ell = 1$, the completion $\wh{f} (\tau,w)$ is holomorphic on a double cover of $\IH \times \IH$ with a branch point at $w = \tau$ in the $w$-plane.} 
The completion $\wh{f}$ reduces to $f$ in the $w-\t \to i \infty$ limit\footnote{Throughout this article, we assume with the notation $\lim_{w \to \t + i \infty}$ that the limit is taken along the vertical line in $\IH$ passing from $\t$.} and it satisfies a modular transformation of the form 
\begin{equation} \label{2-modular}
\wh{f} \lp \frac{a \t + b}{c \t + d},  \frac{a w + b}{c w + d} \rp 
= \chi_{\t,w} (\g)^\ell (c \t + d)^k \rho(\g) \wh{f} (\t,w).
\end{equation}
Here $\g = \pmat{ a & b \\ c & d}$ is an element of (a subgroup of) $\mathrm{SL}_2 (\IZ)$, $\rho(\g)$ is a possibly vector-valued multiplier system, and $\chi_{\t,w} (\g) \in \{ \pm 1 \}$ is a sign factor that keeps track of branch cuts one may use on the factor $(i(w-\t))^{\frac\ell2}$ to make $\wh{f} (\t,w)$ single-valued (see Theorem \ref{co:modularity} for more details). With this transformation law, the modular completions studied here are examples of bimodular forms defined in the unpublished work of Stienstra and Zagier.\footnote{
A {\it bimodular form} of {\it biweight} $(k,m)$ is a holomorphic function $g$ that satisfies the transformation property 
\begin{align*}
g \lp \frac{a \t + b}{c \t + d},  \frac{a w + b}{c w + d} \rp 
=  (c \t + d)^k(c w + d)^m \rho(\gamma) g (\t,w) .
\end{align*}
The notion can be flexibly used by assuming the domain of $g$ is $\IH \times \IH$, $\IH \times \IH \setminus \{ \t = w\}$, a double cover of $\IH \times \IH \setminus \{ \t = w\}$ (in which case one also needs a sign factor $\chi_{\t,w} (\gamma)$ to keep track of the branch on the double cover), or $\IH \times \IH^-$ where $\IH^-$ is the lower half-plane (which are relevant for mock modular forms by analytically continuing their real-analytic modular completion).
}
More specifically, the modular completions here are subclasses of bimodular forms, for which we require the derivative of $\wh{f}(\t,w)$ in $w$ to satisfy
\begin{equation}\label{eq:false_completion_derivative}
\frac{\del  }{\del w} \wh{f}(\t, w)
= \sum_j \lp i (w-\t) \rp^{r_j} \wh{g}_j (\t,w) h_j (w),
\end{equation}
where $\wh{g}_j \in \wh{\FF}^{d-1}_{k + r_j}$ and $h_j \in \MM_{2+r_j}$ in a way similar to mock modular forms (of higher depth). The modular completions of false theta functions with a single sign function, which were studied in \cite{BN}, fit into this framework and yield examples of false modular forms at depth one (or more simply false modular forms, where we leave the phrase ``depth one" implicit).  Another simple example is the Eisenstein series $E_2 (\tau)$, which can be recognized as a false modular form of weight two with a modular completion 
\begin{equation}\label{E+w}
\widehat{E}_2(\tau,w):=E_2(\tau)-\frac{6 i}{\pi (\tau-w)}.
\end{equation}
 Of course, this is a rather degenerate example since $E_2 (\tau)$ can also be viewed as a mock modular form of weight two and depth one.

In this paper, we show that false theta functions with arbitrary rank and an arbitrary number of sign function insertions fit into these spaces by providing modular completions as described here (and including elliptic variables that lead to Jacobi form-like generalizations) and by showing how one can obtain further examples by constructing differential operators that are covariant under the modular action. To give a more concrete statement, let $\LL$ be a lattice of rank $n$ with an integral, positive definite bilinear form $B$ (where $B(\bm{x},\bm{y}):= \bm{x}^T A \bm{y}$ with $A\in\IZ^{n\times n}$ a symmetric, positive definite matrix) and an associated quadratic form $Q$ (so that $Q(\bm{n}) = \frac{1}{2} B(\bm{n},\bm{n})$ for all $\bm{n} \in \IR^n$). Moreover, let $\bm{\ell} \in \LL$ be a \textit{characteristic vector} satisfying $Q(\bm{n}) + \frac{1}{2} B(\bm{\ell},\bm{n}) \in \Z$ for every $\bm{n} \in \LL$. For $\bm{\mu} \in \LL^* / \LL$, $\t\in\mathbb{H}$, and a set of $r$ linearly independent vectors $\bm{m_1}, \ldots, \bm{m_r} \in \IR^n$ gathered in a rank $r$ matrix $M := \lp \bm{m_1}, \ldots, \bm{m_r} \rp\in \IR^{n\times r}$, we define the \textit{false theta function}
\begin{align} \label{false.gen}
F_{Q,\bm{\mu}, \bm \ell,M} (\tau) := 
\sum_{\bm{n} \in \LL + \bm{\mu} + \frac{\bm{\ell}}{2}}
 \sgn ( B ( M,  \bm{n} ))	e^{\pi i B \lp \bm{n}, \bm\ell \rp} q^{Q(\bm{n})} ,
\end{align}
where $ \sgn ( B ( M,  \bm{n} )) := \prod_{j=1}^r \sgn (B (\bm{m_j}, \bm{n}))$. The main mathematical object of interest in this paper is the following modular completion (that also includes elliptic variables), which we define for $\bm{z},\bm{\mathfrak{z}} \in \IC^n$ and $w, \t \in \mathbb{H}$ with $w\neq \t $ as
\begin{equation}\label{eq:hat_psi_definition}
\wh{\Psi}_{Q,\bm{\mu},\bm{\ell},M} ( \bm{z}, \bm{\zz}; \tau, w)   
:= \sum_{\bm{n} \in \LL + \bm{\mu} + \frac{\bm{\ell}}{2}}
E_{Q,M} \lp - i \sqrt{i (w-\t)}  \lp  \bm{n} + \frac{\bm{\zz} -\bm{z}}{w-\t} \rp \rp
q^{Q(\bm{n})} e^{2 \pi i B \lp \bm{n},\bm{z} + \frac{\bm{\ell}}{2} \rp},
\end{equation}
where we define the square-root with its principal value and where following \cite{ABMP} the \emph{generalized error function}, $E_{Q,M}$, is given by
\begin{equation}\label{eq:generalized_error_definition}
E_{Q,M} ( \bm{x} ) 
:= 
\det(A)^\frac12
\int_{\IR^n} \sgn  (B (M,\bm{y} ) )
e^{-2 \pi Q\lp \bm{y}-\bm{x}\rp  } d^n \bm{y}.
\end{equation}
From now on, we omit $Q$, $\bm{\ell}$, and $M$ from the notation $\wh{\Psi}_{Q,\bm{\mu}, \bm{\ell}, M}$ if they are clear from the context. We also use the notation $\wh{\Psi}_{Q,\bm{\mu},\bm{\ell},M} (\tau, w) := \wh{\Psi}_{Q,\bm{\mu},\bm{\ell},M} ( \bm{0}, \bm{0}; \tau, w)$.

Given these definitions, we can now state our main result that shows that $\wh{\Psi}$ is invariant under modular and elliptic transformations (which one could call ``bi-Jacobi transformation properties" analogous to bimodular forms). For this, set
\begin{equation*}
	\chi_{\t,w} := \sqrt{\frac{i (w-\t)}{w \t} } 
	\frac{\sqrt{w} \sqrt{\t} }{\sqrt{i (w-\t) }} \in \{ \pm 1 \}.
\end{equation*}

\begin{thm}\label{co:modularity}
\ \begin{enumerate}[leftmargin=*,label={\rm (\arabic*)}]
\item For $\bm{m},\bm{r} \in \LL$, we have the elliptic transformations
\begin{equation*}
\wh{\Psi}_{\bm{\mu}} ( \bm{z} + \bm{m}\t + \bm{r}, \bm{\zz}+ \bm{m}w + \bm{r}; \tau, w)  = 
(-1)^{2Q(\bm{m}+\bm{r})} q^{-Q(\bm{m})} e^{-2\pi i B(\bm{m},\bm{z})}
\wh{\Psi}_{\bm{\mu}} ( \bm{z}, \bm{\zz}; \tau, w).
\end{equation*}
\item We have the modular transformations
\begin{align*}
\wh{\Psi}_{\bm{\mu}} ( \bm{z}, \bm{\zz}; \tau+1, w+1) &=
e^{2\pi iQ \lp \bm{\mu} + \frac{\bm{\ell}}{2} \rp}
\wh{\Psi}_{\bm{\mu}} ( \bm{z}, \bm{\zz}; \tau, w), \\
\wh{\Psi}_{\bm{\mu}} \lp \frac{\bm{z}}{\t}, \frac{\bm{\zz}}{w}; 
-\frac{1}{\tau}, -\frac{1}{w} \rp
&= \chi_{\t,w}^r
\frac{(-i \t)^{\frac{n}{2}} e^{-\pi i Q(\bm{\ell})}}{\sqrt{|\LL^* / \LL |}} 
e^{2\pi i \frac{Q(\bm{z})}{\t}}
\sum_{\bm{\nu} \in \LL^*/\LL} e^{-2\pi i B(\bm{\mu},\bm{\nu})}
\wh{\Psi}_{\bm{\nu}} ( \bm{z}, \bm{\zz}; \tau, w) .
\end{align*}

\item The limit $\lim_{w \to \t + i \infty} \wh{\Psi}_{Q,\bm{\mu},\bm{\ell},M} (\tau, w)$ yields the false theta function $F_{Q,\mu, \bm \ell,M} (\tau)$ up to possible corrections including false theta functions with fewer sign function insertions. If the vectors in $M$ are elements of $\LL$, the resulting limit is a weight $\frac{n}{2}$, depth $r$ false modular form with modular completion $\wh{\Psi}_{Q,\bm{\mu},\bm{\ell},M} (\tau, w)$. 

\end{enumerate}
\end{thm}
A more precise version of part (3) is obtained using Lemma \ref{la:generalized_error_asymptotic_form} and Lemma \ref{lem:EQMComplexRecursion} (also see the proof given in Section \ref{sec:higher_depth_false_theta} for more details). With this result at hand, we then show how one can obtain further examples of higher depth false modular forms using differential operators such as
\begin{equation}\label{eq:modular_covariant_diff_op}
\bm{\CD} := \bm{\del_z} + \frac{2 \pi i}{\tau-w} A ( \bm{z}-\bm{\zz} ),
\end{equation}
which are covariant under modular and elliptic transformations (see Proposition \ref{prop:ModularlyCovariant}). Such operators are also important in giving a representation theoretic handle to the space of false modular forms. This opens up the possibility to explore how general this structure is beyond the examples constructed from false theta functions in this paper and the trivial example of $E_2$. For further comments along this directions, see the discussion and the examples at the end of Section \ref{sec:higher_depth_false_theta}.

\begin{rem}
Thanks to Theorem \ref{co:modularity} and Proposition \ref{prop:ModularlyCovariant}, the higher rank false theta functions studied in this paper automatically lead to higher depth quantum modular forms as defined and studied in \cite{BKM1} and \cite{BMM} if the finiteness of the rational limits is ensured (depending on the polynomial insertion, lattice shift, etc.).
\end{rem}

The paper is organized as follows. In Section \ref{sec:generalized_error}, we study various properties of generalized error functions following \cite{ABMP,Na}. One particular result we prove here is Lemma \ref{sign.lemma}, which generalizes the results of \cite{BKMN} used to write rank two false theta functions in terms of iterated Eichler-type integrals of modular forms to arbitrary ranks. In Section \ref{sec:higher_depth_false_theta}, we show our main result, Theorem \ref{co:modularity}, and study covariant differential operators such as the one in equation \eqref{eq:modular_covariant_diff_op}. In Section \ref{sec:w_algebra_characters}, we apply the formalism developed here to a number of false modular forms up to depth three appearing in characters of $W^0(p)_L$-modules. These $q$-series were first investigated in \cite{BM2,CM2} and later in \cite{BKM1}, where three of the authors proved that for $L=A_2$  their asymptotic behavior at rationals gives rise to a family of depth two quantum modular forms of weight one and two. The results in this paper are stated quite differently compared to \cite{BKM1}, but Lemma \ref{sign.lemma} can in fact be employed to reprove the depth two quantum modularity of $W^0(p)_{A_2}$-characters. Finally, in Section \ref{sec:Zhat_invariant} we show how our results can be used to study $\hat{Z}$-invariants of plumbed $3$-manifolds \'a la Gukov--Pei--Putrov--Vafa \cite{GPPV}.

\addtocontents{toc}{\SkipTocEntry}
\section*{Acknowledgements}
We thank the referee for the useful comments and helpful suggestions.

\section{Properties of Generalized Error Functions}\label{sec:generalized_error}
We start by laying out the notations and conventions we use in this paper and setting up properties of the generalized error functions defined in \eqref{eq:generalized_error_definition}.
Note that the $\bm{x}$-dependence of $E_{Q,M}( \bm{x})$ is only through the orthogonal projection of $\bm{x}$ to the subspace spanned by the columns of $M$ and we have $E_{Q,\emptyset}(\bm{x})=1$.
In the following, we employ the notations $[j,r] := \{ j,j+1, \ldots, r \}$ and $[r]:=[1,r]$, and for any matrix $C=(\bm{c_1},\dots, \bm{c_r})$, we denote by $C_S:=(\bm{c_{k_1}},\bm{c_{k_2}},\dots, \bm{c_{k_{\lvert S\rvert}}})$ the submatrix consisting of the columns indexed by $S=\{k_1,k_2,\dots, k_{\lvert S\rvert} \}\subseteq[r]$ (with $k_1<k_2<\dots< k_{\lvert S\rvert}$).

The first property we give is the limiting behavior of $E_{Q,M}$, which follows directly from the definition of $E_{Q,M}$ given in \eqref{eq:generalized_error_definition}. This is used below to determine the limiting behavior of our modular completions. Throughout we let $ \d_\mathcal{S}:=1$ if a statement $\mathcal{S}$ holds and $\d_\mathcal{S}:=0$ otherwise.

\begin{lemma}\label{la:generalized_error_asymptotic_form}
For $\bm{x} \in \IR^n$, we have 
\begin{equation*}
\lim_{t \to \infty} E_{Q,M}(t \bm{x}) = \sum_{S \subseteq [r]}
E_{Q, M_{S} } ( \bm{0})
\d_{\scriptscriptstyle B(M_{S}, \bm{x})=\bm{0}} 
\sgn (B(M_{[r] \setminus S},  \bm{x} ) ) ,
\end{equation*}
In particular, if all components of $B( M, \bm{x} )$ are non-zero, then 
\begin{equation*}
\lim_{t \to \infty} E_{Q,M}(t \bm{x}) = \sgn (B( M, \bm{x} )).
\end{equation*}
\end{lemma}

\begin{rem}\label{rem:E2_zero_value}
Note that $E_{Q,M}( \bm{0})  = 0$ for odd $r$ since the integrand in \eqref{eq:generalized_error_definition} is odd in $\bm{y}$. Moreover, passing to an orthonormal basis, $E_{Q,M}( \bm{0})$ can be recognized as the average of $\sgn (B(M, \bm{x} ))$ over the unit sphere. For instance, if $r=2$ with $\wh{\bm{m}}_{\bm{j}}:=\frac{\bm{m_j} }{\sqrt{2 Q(\bm{m_j})}}$, we more explicitly have
\begin{align*}
E_{Q,M} (\bm{0})
=\frac{2}{\pi} \arcsin \lp B(\widehat{\bm{m}}_{\bm{1}}, \wh{\bm{m}}_{\bm{2}}) \rp.
\end{align*}
 \end{rem}

We also need derivatives of the generalized error functions. We first introduce some notation for this purpose. Let $\bm{w_1}, \ldots, \bm{w_r}$ be the basis dual to $\bm{m_1}, \ldots, \bm{m_r}$ for the subspace spanned by them, which in particular means that $B(\bm{m_j}, \bm{w_k}) = \d_{j,k}$, where $\d_{j,k}:=1$ if $j=k$ and 0 otherwise. We write $\bm{m_{k\perp S}}$ for the orthogonal projection of $\bm{m_k}$ to the orthogonal complement of the span of $M_S$ for some $S\subset[r]$. Finally, we let $\bm{\partial_x}:=(\frac{\partial}{\partial x_1},\frac{\partial}{\partial x_2},\dots, \frac{\partial}{\partial x_n})^T$ and $M_{\{k_1,\dots, k_\ell\}\perp S}:=(\bm{m_{k_1\perp S}},\dots, \bm{m_{k_\ell\perp S}})$.
Then the required derivatives were computed in Lemma 3.5 of \cite{Na}. We restate this result here and give an elementary proof for the reader's convenience.
\begin{lem}\label{lem:generalized_error_derivative}
We have ($\bm{z} \in \IC^n$)
\begin{equation*}
\bm{w_k}^T \bm{\partial_z} \lp E_{Q,M} (\bm{z})  \rp  = 
\frac{2}{\sqrt{2Q(\bm{m_k})} }
e^{-  \frac{\pi B\left(\bm{m_k}, \bm{z}\right)^2}{2Q\left(\bm{m_k}\right) } }
E_{  Q, M_{[r] \setminus \{k\}\perp \{k\}}} \left(  \bm{z} \right) .
\end{equation*}
\end{lem}
\begin{proof}
Without loss of generality we can restrict to $\bm{z} \in \IR^n$ and the result generalizes to complex values with analytic continuation. First we make a linear coordinate transformation to switch to Euclidean coordinates for which 
$Q(\bm{z}) = \frac{1}{2} z_0^2 + \frac{1}{2} \bm{\zzb}^T \bm{\zzb}$ and where we write any vector $\bm{z} \in \IR^n$ as $\bm{z} = (z_0, \bm{\zzb} )$ for $z_0 \in \IR$ and $\bm{\zzb} \in \IR^{n-1}$. Moreover we can pick these Euclidean coordinates so that $\mathbbm{w}_k = \bm{0}$ and further rescale $\bm{m}_k$ to pick
\begin{equation*}
\bm{w_k} = (1, \bm{0}), 
\quad 
\bm{m_k} = (1, \bm{\mmb_k}), 
\quad \mbox{and} \ \ 
\bm{m_j} = (0,  \bm{\mmb_j}) \mbox{ for } j \neq k .
\end{equation*}
The left-hand side of the lemma is then equal to, with $\bm{\y}^2 := \bm{\y}^T \bm{\y}$,
\begin{equation*}
\int_{\IR^{n-1}}  \prod_{\substack{j=1 \\ j \neq k}}^r \sgn \left( \bm{\mmb_j}^T \bm{\y} \right)
e^{-\pi \left(\bm{\y}-\bm{\zzb}\right)^2} 
\int_\IR \sgn \left(y_0 + \bm{\mmb_k}^T \bm{\y} \right)
 \del_{z_0} e^{-\pi \left(y_0-z_0\right)^2}
d y_0 d^{n-1} \bm{\y}.
\end{equation*}
Writing $\del_{z_0} e^{-\pi (y_0-z_0)^2} = - \del_{y_0} e^{-\pi (y_0-z_0)^2}$ and integrating over $y_0$  the left-hand side of the lemma equals
\begin{equation}\label{eq:lem_derivative_LHS}
2 \int_{\IR^{n-1}}  \prod_{\substack{j=1 \\ j \neq k}}^r \sgn \left( \bm{\mmb_j}^T \bm{\y} \right)
e^{-\pi \left( \bm{\y}-\bm{\zzb} \right)^2 -\pi \left(z_0+ \bm{\mmb_k}^T \bm{\y}\right)^2} 
 d^{n-1} \bm{\y}.
\end{equation}

Next we focus on the right-hand side of the lemma. We start by noting that for $j \neq k$ we have
\begin{equation*}
\bm{m_{j \perp k}} = \bm{m_j} - \frac{B\left(\bm{m_j}, \bm{m_k} \right)}{B\left(\bm{m_k}, \bm{m_k} \right)} \bm{m_k}
= \frac{1}{1+\bm{\mmb_k}^2} \lp - \bm{\mmb_k}^T \bm{\mmb_j}, \left(1+\bm{\mmb_k}^2\right) \bm{\mmb_j} - \left( \bm{\mmb_k}^T \bm{\mmb_j} \right) \bm{\mmb_k} \rp .
\end{equation*}
Then we write $E_{  Q, M_{[r] \setminus \{k\}\perp \{k\}}} \left(  \bm{z} \right) $ in these coordinates as
\begin{align*}
\int_{\IR^n} \prod_{\substack{j=1 \\ j \neq k}}^r
\sgn \lp - \left(\bm{\mmb_k}^T \bm{\mmb_j}\right) y_0 + \left(1+\bm{\mmb_k}^2\right) \left(\bm{\mmb_j}^T \bm{\y} \right) - \left( \bm{\mmb_k}^T \bm{\mmb_j} \right) \left(\bm{\mmb_k} ^T \bm{\y}\right)   \rp 
e^{-\pi \left(y_0-z_0\right)^2 -\pi \left(\bm{\y}-\bm{\zzb}\right)^2} d y_0 d^{n-1} \bm{\y} .
\end{align*}
Making the change of variables $(y_0, \bm{\y}) \mapsto (y_0 - \bm{\mmb_k}^T \bm{\y}, \bm{\y} + y_0 \bm{\mmb_k}) $, the integral turns into
\begin{align*}
\left(1+\bm{\mmb_k}^2\right)
\int_{\IR^n} \prod_{\substack{j=1 \\ j \neq k}}^r
\sgn \lp \bm{\mmb_j}^T \bm{\y}   \rp 
e^{-\pi \left(y_0 - \bm{\mmb_k}^T \bm{\y}-z_0\right)^2 -\pi \left(\bm{\y}+ y_0 \bm{\mmb_k}-\bm{\zzb}\right)^2} d y_0 d^{n-1} \bm{\y} .
\end{align*}
Together with the prefactor $\frac{2}{\sqrt{2Q(\bm{m_k})} }
e^{-  \frac{\pi B(\bm{m_k}, \bm{z})^2}{2Q(\bm{m_k}) } }$, the right-hand side then overall becomes
\begin{equation*}
2 \sqrt{1+\bm{\mmb_k}^2}  
\int_{\IR^n} \prod_{\substack{j=1 \\ j \neq k}}^r
\sgn \lp \bm{\mmb_j}^T \bm{\y}   \rp 
e^{-\pi \left(y_0 - \bm{\mmb_k}^T \bm{\y}-z_0\right)^2 -\pi \left(\bm{\y}+ y_0 \bm{\mmb_k}-\bm{\zzb}\right)^2
-\frac{\pi}{1+\bm{\mmb_k}^2} \left(z_0 + \bm{\mmb_k}^T \bm{\zzb} \right)^2 }
 d y_0 d^{n-1} \bm{\y}.
\end{equation*}
Noting that 
\begin{align*}
&\left(y_0 - \bm{\mmb_k}^T \bm{\y}-z_0\right)^2 + \left(\bm{\y}+ y_0 \bm{\mmb_k}-\bm{\zzb} \right)^2+ \frac{\left(z_0 + \bm{\mmb_k}^T \bm{\zzb} \right)^2}{1+\bm{\mmb_k}^2} 
\\ & \qquad \qquad
=
\left(\bm{\y}-\bm{\zzb} \right)^2 + \left( z_0+\bm{\mmb_k}^T \bm{\y} \right)^2 + \left(1+\bm{\mmb_k}^2\right) \lp y_0 - \frac{z_0 + \bm{\mmb_k}^T \bm{\zzb} }{1+\bm{\mmb_k}^2} \rp^2
\end{align*}
and taking the Gaussian integral over $y_0$ yields an expression equal to equation \eqref{eq:lem_derivative_LHS}, and hence proving the lemma statement.
\end{proof}

\begin{rem}
As in Proposition 3.13 of \cite{Na}, Lemma \ref{lem:generalized_error_derivative} implies
\begin{equation*}
\bm{x}^T \bm{\partial_x} \lp E_{Q,M} (\bm{x}) \rp  = 
2 \sum_{k=1}^r \frac{B(\bm{m_k}, \bm{x})}{\sqrt{2Q(\bm{m_k})} }
e^{-  \frac{\pi B\left(\bm{m_k}, \bm{x}\right)^2}{2Q\left(\bm{m_k}\right) } }
E_{  Q,M_{[r] \setminus \{k\}\perp \{k\}}} \left(  \bm{x} \right)
\end{equation*}
and we use this expression to get the following recursive relation
\begin{equation}\label{eq:generalized_error_recursion}
E_{Q,M}( t \bm{x}) 
= E_{Q,M}( \bm{0})  +
2 \sum_{k=1}^r \int_0^t 
\frac{B(\bm{m_k}, \bm{x})}{\sqrt{2Q(\bm{m_k})} }
e^{-  \frac{\pi B\left(\bm{m_k}, \bm{x}\right)^2}{2Q\left(\bm{m_k}\right) }t_1^2 }
E_{  Q,M_{[r] \setminus \{k\}\perp \{k\}}}  \left(   t_1  \bm{x} \right) d t_1 .
\end{equation}
\end{rem}
A particularly convenient rewriting of this result that is useful in this paper is as follows.
\begin{lemma}\label{lem:EQMComplexRecursion}
For $\t,w \in \mathbb{H}$ with $\t \neq w$ and $\bm{x} \in \IC^n$, we have
\begin{multline*}
E_{Q,M} \left(-i\sqrt{i(w-\t)}\bm{x}\right)
\\\qquad 
= E_{Q,M}( \bm{0})+	
\sum_{k=1}^r \int_{\t}^{w} 
B(\widehat{\bm{m}}_{\bm{k}}, \bm{x})
\frac{e^{ \pi i B(\widehat{\bm{m}}_{\bm{k}}, \bm{x})^2 (w_1-\t) }}{\sqrt{i(w_1-\t)}}
E_{  Q,M_{[r] \setminus \{k\}\perp \{k\}}}  \left( -i  \sqrt{i(w_1-\t)}  \bm{x} \right) d w_1 .
\end{multline*}
\end{lemma}

\begin{rem}
Since we define the square-root via its principal value, the function $\sqrt{i (w_1- \t)}$ is implemented with a branch-cut in the $w_1$-plane that is a vertical ray from $w_1 = \t$ to $w_1 = i \infty$. Therefore, in the statement of Lemma \ref{lem:EQMComplexRecursion}, deforming the integration path from $\t$ to $w$ without crossing the branch-cut does not change the value of the corresponding integral thanks to the holomorphicity of the integrand. Whenever we need to be more concrete, we assume that the integral is taken along the hyperbolic geodesic from $\t$ to $w$ following \cite{BKMN}. This is a particularly convenient choice for studying the modular transformations of such objects. We make this choice also when we encounter analogous (iterated) integrals below.
\end{rem}

The main upshot of the recursive relation in equation \eqref{eq:generalized_error_recursion} is that it allows us to express the generalized error functions as iterated integrals of Gaussian functions. To state this result, we introduce some notation. 
For a permutation $\s \in S_r$, we define the basis $(\bm{v_{1,\s}},\dots, \bm{v_{r,\s}})$ obtained by applying the orthonormal Gram--Schmidt process (with respect to $B$) to $(\bm{m_{\s(1)}},\dots, \bm{m_{\s(r)}})$. 
By \eqref{eq:generalized_error_recursion}, using induction on $r$, the number of vectors in $M$, then gives the following result.\footnote{
In the following lemma, instead of writing $\s \left([a,b]\right)$ for permutations $\s\in S_r, a,b\in[r]$, we simply write $\s[a,b]$ for improved readability.
}

\begin{lemma}\label{la:ErrorIntegral}
For $\bm{x} \in \IC^n$, we have\footnote{We leave the linear factors $ B(\bm{v_{j,\s}}, \bm{x}  )$ inside the integrals to emphasize that summing over $\bm{x}$ would generate unary theta functions of weight $\frac{3}{2}$ as integrands.} 
\begin{align*}
&E_{Q,M}( t \bm{x}) = 
\sum_{\s \in S_r} \sum_{\ell=0}^r \frac{2^\ell}{(r-\ell)!}
E_{Q,M_{\s[\ell+1,r]\perp \s[1,\ell]}} ( \bm{0})  \int_0^t    B(\bm{v_{1,\s}}, \bm{x} )
e^{- \pi B\left( \bm{v_{1,\s}}, \bm{x}\right)^2 t_1^2 }
\\ &\hspace{2cm}\quad
\times
\int_0^{t_1}  B(\bm{v_{2,\s}}, \bm{x}  )
e^{- \pi B\left( \bm{v_{2,\s}}, \bm{x}\right)^2 t_2^2 } \ldots
\int_0^{t_{\ell-1}} B(\bm{v_{\ell,\s}}, \bm{x} )
e^{- \pi B( \bm{v_{\ell,\s}}, \bm{x})^2 t_\ell^2 } d t_\ell\dots d t_2  d t_1.
\end{align*}
\end{lemma}

\begin{rem}
Lemma \ref{la:ErrorIntegral} can be conveniently restated as (for $\t,w \in \mathbb{H}$ with $\t \neq w$)
\begin{multline*}
E_{Q,M} \left(-i\sqrt{i(w-\t)}\bm{x}\right) = 
\sum_{\s \in S_r} \sum_{\ell=0}^r \frac{E_{Q,M_{\s[\ell+1,r]\perp \s[1,\ell]}} ( \bm{0}) }{(r-\ell)!}
\int_\t^w
\frac{B(\bm{v_{1,\s}}, \bm{x} ) e^{\pi i B\left( \bm{v_{1,\s}}, \bm{x}\right)^2 (w_1-\t)}}{\sqrt{i(w_1-\t)}}
\\
\times
\int_\t^{w_1}  
\frac{B(\bm{v_{2,\s}}, \bm{x} ) e^{\pi i B\left( \bm{v_{2,\s}}, \bm{x}\right)^2 (w_2-\t)}}{\sqrt{i(w_2-\t)}}
\ldots
\int_\t^{w_{\ell-1}} 
\frac{B(\bm{v_{\ell,\s}}, \bm{x} ) e^{\pi i B\left( \bm{v_{\ell,\s}}, \bm{x}\right)^2 (w_\ell-\t)}}{\sqrt{i(w_\ell-\t)}}
d w_\ell\dots d w_2  d w_1.
\end{multline*}
\end{rem}

\begin{rem}
The iterated integrals in Lemma \ref{la:ErrorIntegral} can be used to give a convergent series for generalized error functions around $\bm{x}=\bm{0}$ as in Proposition 3.7 of \cite{ABMP}, which is in turn useful for its numerical evaluation.
\end{rem}

Finally, we note that combining the results of Lemma \ref{la:generalized_error_asymptotic_form} and Lemma \ref{la:ErrorIntegral} leads to an identity for products of sign functions, which generalizes Lemma 3.1 of \cite{BKMN} that was used there for studying the modular properties of rank two false theta functions.
\begin{lemma} \label{sign.lemma}
For $\bm{x} \in \IR^n$ and if all components of $B( M, \bm{x} )$ are non-zero, then we have
\begin{align*}
&\sgn ( B( M, \bm{x} ) )  =
\sum_{\s \in S_r} \sum_{\ell=0}^r \frac{E_{Q,M_{\s[\ell+1,r]\perp \s[1,\ell]}} ( \bm{0}) }{(r-\ell)!}
\int_\t^{\t+i \infty}
\frac{B(\bm{v_{1,\s}}, \bm{x} ) e^{\pi i B\left( \bm{v_{1,\s}}, \bm{x}\right)^2 (w_1-\t)}}{\sqrt{i(w_1-\t)}}
\\ & \qquad
\times
\int_\t^{w_1}  
\frac{B(\bm{v_{2,\s}}, \bm{x} ) e^{\pi i B\left( \bm{v_{2,\s}}, \bm{x}\right)^2 (w_2-\t)}}{\sqrt{i(w_2-\t)}}
\ldots
\int_\t^{w_{\ell-1}} 
\frac{B(\bm{v_{\ell,\s}}, \bm{x} ) e^{\pi i B\left( \bm{v_{\ell,\s}}, \bm{x}\right)^2 (w_\ell-\t)}}{\sqrt{i(w_\ell-\t)}}
d w_\ell\dots d w_2  d w_1.
\end{align*}
In case some coordinates of $B( M, \bm{x} )$ are zero, the left-hand side more generally becomes
\begin{equation*}
\sum_{S \subseteq [r]}
E_{Q, M_{S} } ( \bm{0}) 
\d_{\scriptscriptstyle B(M_{S}, \bm{x})=\bm{0}}  
\sgn (B(M_{[r] \setminus S},  \bm{x} ) ) .
\end{equation*}
\end{lemma}

\section{Higher Depth False Theta Functions}\label{sec:higher_depth_false_theta}
Now that the technical groundwork for generalized error functions is laid, we move to studying the completed false theta functions $\wh{\Psi}_{\bm{\mu}} ( \bm{z}, \bm{\zz}; \tau, w)$ and their modular properties by following the arguments of \cite{BN}, which studies the $r=1$ case.
It is convenient for this purpose to work with an auxiliary, rescaled function 
\begin{equation}\label{eq:tildePsi_definition}
\wt{\Psi}_{\bm{\mu}} ( \bm{z}, \bm{\zz}; \tau, w) := 
\lp i (w-\t) \rp^{\frac{r}{2}} 
\wh{\Psi}_{\bm{\mu}} ( \bm{z}, \bm{\zz}; \tau, w) ,
\end{equation}
which, as we see below, eliminates the branch cut due to $\sqrt{i(w-\t)}$ and the corresponding sign-factor $\chi_{\tau, w}$ appearing in the modular transformations (see Theorem \ref{co:modularity}), while still retaining a covariant behavior under modular transformations.\footnote{However, it is still $\wh{\Psi}$ that asymptotes to false theta functions and this is why we designate it as the main object to study.} The next thing we note is the decomposition 
\begin{equation}\label{zDecompose}
\bm{z} =  \frac{\bm{\zz} -\bm{z}}{w-\t}  \t - \frac{\bm{\zz} \t - \bm{z} w }{w-\t},
\end{equation}
which is analogous to the familiar decomposition $\bm{z} = \bm{a} \t + \bm{b}$, where $\bm{a},\bm{b} \in \IR^n$ are given by $\bm{a}:= \frac{\mathrm{Im}(\bm{z})}{\mathrm{Im}(\t)}$ and $\bm{b} := - \frac{\mathrm{Im}(\bm{z} \bar{\t})}{\mathrm{Im}(\t)}$. In fact, $\frac{\bm{\zz} -\bm{z}}{w-\t}$ and $-\frac{\bm{\zz} \t - \bm{z} w }{w-\t}$ transform under the modular and elliptic transformations listed in Theorem \ref{co:modularity} in  the same way as $\bm{a}$ and $\bm{b}$ would for the modular and elliptic transformations of $(\bm{z},\t)$ under the Jacobi group. This provides another motivation for the complexified modular and elliptic transformations given in Theorem \ref{co:modularity} and the factor of $ \frac{\bm{\zz} -\bm{z}}{w-\t}$ inside the generalized error function appearing in the definition of $\wh{\Psi}$ in \eqref{eq:hat_psi_definition}. 
With this background at hand, we are now ready to state the properties of the auxiliary function $\wt{\Psi}$.
\begin{thm}\label{thm:wtPsi_modularity}
The function $\wt{\Psi}: \IC^n \times \IC^n \times (\IH \times \IH \setminus \{ \t = w \}) \to \IC$ is a holomorphic function that satisfies the following elliptic and modular transformations.
\begin{enumerate}[leftmargin=0.7cm, label={ \rm (\arabic*)}]
\item For $\bm{m},\bm{r} \in \LL$, we have
\begin{equation*}
\wt{\Psi}_{\bm{\mu}} ( \bm{z} + \bm{m}\t + \bm{r}, \bm{\zz}+ \bm{m}w + \bm{r}; \tau, w)  = 
(-1)^{2Q(\bm{m}+\bm{r})} q^{-Q(\bm{m})} e^{-2\pi i B(\bm{m},\bm{z})}
\wt{\Psi}_{\bm{\mu}} ( \bm{z}, \bm{\zz}; \tau, w).
\end{equation*}
\item Under modular transformation, we have
\begin{align*}
\wt{\Psi}_{\bm{\mu}} ( \bm{z}, \bm{\zz}; \tau+1, w+1) &=
e^{2\pi iQ \lp \bm{\mu} + \frac{\bm{\ell}}{2} \rp}
\wt{\Psi}_{\bm{\mu}} ( \bm{z}, \bm{\zz}; \tau, w), \\
\wt{\Psi}_{\bm{\mu}} \lp \frac{\bm{z}}{\t}, \frac{\bm{\zz}}{w}; 
-\frac{1}{\tau}, -\frac{1}{w} \rp
&=\frac{(-i)^{\frac{n}{2}} e^{-\pi i Q(\bm{\ell})}}{\sqrt{|\LL^* / \LL |}} 
\t^{\frac{n-r}{2}} w^{-\frac{r}{2}} e^{2\pi i \frac{Q(\bm{z})}{\t}}
\sum_{\bm{\nu} \in \LL^*/\LL} e^{-2\pi i B(\bm{\mu},\bm{\nu})}
\wt{\Psi}_{\bm{\nu}} ( \bm{z}, \bm{\zz}; \tau, w) .
\end{align*}
\end{enumerate}
\end{thm}
\begin{proof}
To investigate the properties of $\wt{\Psi}_{\bm{\mu}}$, we use \eqref{eq:hat_psi_definition}, \eqref{eq:tildePsi_definition}, and \eqref{zDecompose} to write
\begin{equation}\label{PsiAlternative}
\wt{\Psi}_{\bm{\mu}} ( \bm{z}, \bm{\zz}; \tau, w)  = 
e^{- 2 \pi i  Q \lp \frac{\bm{\zz} -\bm{z}}{w-\t}  \rp \t}
\sum_{\bm{n} \in \LL + \bm{\mu} + \frac{\bm{\ell}}{2}}
f_{\t,w} \lp \bm{n} + \frac{\bm{\zz} -\bm{z}}{w-\t}   \rp
e^{2\pi i B\lp \bm{n}, \frac{\bm{\ell}}{2}  - \frac{\bm{\zz} \t - \bm{z} w }{w-\t}  \rp} ,
\end{equation}
where we introduce
\begin{equation*}
f_{\t,w} (\bm{x}) := (i (w-\t))^{\frac{r}{2}} e^{2 \pi i Q(\bm{x}) \t}
E_{Q,M} \lp -i \sqrt{i \lp w-\t\rp }\bm{x}\rp .
\end{equation*}
Thanks to Lemma \ref{la:ErrorIntegral}, this function can be rewritten as 
\begin{align}
&f_{\t,w} (\bm{x})  
=  e^{2 \pi i Q\lp\bm{x}\rp  \t}
\sum_{\s \in S_r} \sum_{\ell=0}^r 
\frac{2^\ell i^{\frac{r-\ell}{2}}  \lp w-\t\rp ^{\frac{r+\ell}{2}}  }{\lp r-\ell\rp !} 
E_{Q,M_{\s[\ell+1,r]\perp \s[1,\ell]}} ( \bm{0}) 
\int_0^1    B( \bm{v_{1,\s}}, \bm{x} )
e^{\pi i B\lp \bm{v_{1,\s}}, \bm{x} \rp ^2 \lp w-\t\rp t_1^2 }
\notag
\\ &\qquad
\times
\int_0^{t_1}  B( \bm{v_{2,\s}}, \bm{x} )
e^{\pi i B\lp \bm{v_{2,\s}}, \bm{x} \rp ^2 \lp w-\t\rp t_2^2 } \ldots
\int_0^{t_{\ell-1}} B( \bm{v_{\ell,\s}}, \bm{x} )
e^{\pi i B\lp \bm{v_{\ell,\s}}, \bm{x} \rp ^2 \lp w-\t\rp t_\ell^2 } 
d t_\ell\dots  d t_1.
\label{eq:Ftw_definition}
\end{align}
Note that $E_{Q, M_{\s[\ell+1,r]\perp \s[1,\ell]}} ( \bm{0})$ is zero unless $r - \ell$ is even. Therefore, only the summands in which the term $\lp w-\t\rp ^{\frac{r+\ell}{2}}$ has an integer power contribute nontrivially to the sum. This ensures that $(\bm{x}, \t,w) \mapsto f_{\t,w} (\bm{x})$ is holomorphic on $\mathbb{C}^n \times \mathbb{H} \times \mathbb{H}$. Moreover, if $(\bm{z}, \bm{\zz}; \tau, w)$ belongs to a compact subset of $\IC^n \times \IC^n \times (\IH \times \IH \setminus \{ \t = w \})$, we have a bound of the form
\begin{equation}
\left| f_{\t,w} \lp \bm{n} + \frac{\bm{\zz} -\bm{z}}{w-\t}   \rp
e^{2\pi i B\lp \bm{n}, \frac{\bm{\ell}}{2}  - \frac{\bm{\zz} \t - \bm{z} w }{w-\t}  \rp}  \right| 
\leq C_1 e^{-C_2 Q(\bm{n})},
\label{eq:Ftw_bound}
\end{equation}
where the constants $C_1,C_2 > 0$ depend on the chosen compact subset. The most important contribution to obtain such a bound is the Gaussian part of the integrand in equation \eqref{eq:Ftw_definition}, which contributes as
\begin{equation*}
\left|  e^{2 \pi i Q\lp\bm{n}\rp  \t
+ \pi i \sum_{j=1}^\ell  B\lp \bm{v_{j,\s}}, \bm{n}  \rp ^2 (w-\t) t_j^2 }   \right|
\leq 
e^{-2 \pi  Q\lp\bm{n}\rp  \mathrm{Im}(\t) }
\hspace{-0.05cm}
\begin{cases}
1  & \mbox{if } \mathrm{Im} (w) \geq \mathrm{Im} (\t) \\
e^{- \pi \sum_{j=1}^\ell  B\lp \bm{v_{j,\s}}, \bm{n}  \rp ^2 \mathrm{Im}\lp w-\t\rp }
& \mbox{if } \mathrm{Im} (w) < \mathrm{Im} (\t) .
\end{cases}
\end{equation*}
Upon noting $\sum_{j=1}^{r} B(\bm{v_{j,\s}},\bm{x})^2\leq 2Q(\bm{x})$, we find that this Gaussian part is bounded by  
\begin{equation*}
e^{-2 \pi  Q\lp\bm{n}\rp  \min ( \mathrm{Im}(\t), \mathrm{Im}(w) ) } \leq e^{-2 C_2 Q(\bm{n})}
\end{equation*}
for some constant $C_2 >0$ depending on the chosen compact subset. Using part of the Gaussian $e^{-C_2 Q(\bm{n})}$ to uniformly bound the remaining components then yields the result.

The bound in \eqref{eq:Ftw_bound} implies that the series defining $\wt{\Psi}$ in equation \eqref{PsiAlternative} is normally convergent on compact subsets of $\IC^n \times \IC^n \times (\IH \times \IH \setminus \{ \t = w \})$ and hence that $\wt{\Psi}$ is a holomorphic function of $(\bm{z}, \bm{\zz}; \tau, w)$ on $\IC^n \times \IC^n \times (\IH \times \IH \setminus \{ \t = w \})$. 
It is also important to note here that the derivatives of $f_{\t,w} (\bm{n} + \frac{\bm{\zz} -\bm{z}}{w-\t})
e^{2\pi i B (\bm{n}, \frac{\bm{\ell}}{2}  - \frac{\bm{\zz} \t - \bm{z} w }{w-\t})}$ with respect to $\bm{n}$ can also be bounded to be $O(e^{-C_2 Q(\bm{n})})$ on compact subsets.\footnote{Here we can use the same coefficient $C_2 >0$ as the one deduced in \eqref{eq:Ftw_bound}. This is possible because taking derivatives with respect to $\bm{n}$ does not change the Gaussian part of $f_{\t,w} (\bm{n} + \frac{\bm{\zz} -\bm{z}}{w-\t}) e^{2\pi i B(\bm{n}, \frac{\bm{\ell}}{2}  - \frac{\bm{\zz} \t - \bm{z} w }{w-\t})}$ and we can use the same arguments as above (with the implied multiplicative constant in front depending on the applied derivative operator).} This shows that the summand in equation \eqref{PsiAlternative} is a Schwartz function of $\bm{n}$ and justifies the use of Fourier techniques we employ next in studying modular transformation properties of $\wt{\Psi}$.

The elliptic and the modular translation simply follow from the definition of $\wt{\Psi}$ through equations \eqref{eq:tildePsi_definition} and \eqref{eq:hat_psi_definition}. To prove the modular inversion, we simply apply Poisson summation. Here the details can be followed as in Theorem 2.2 of \cite{BN}. We work with the  representation of $\wt{\Psi}$ in \eqref{PsiAlternative} and note that the appropriate generalization of Lemma 2.1 of \cite{BN} is given by 
\begin{equation*}
\CF (f_{\t,w}) (\bm{x}) = (-i)^{-\frac{n}{2}} \t^{-\frac{n-r}{2}}  w^{\frac{r}{2}} f_{-\frac{1}{\t}, -\frac{1}{w}} (\bm{x}) ,
\end{equation*}
defining the Fourier transform of a function $f: \IR^n \to \IC$ with respect to $Q$ as
\begin{equation*}
\CF ( f) (\bm{x}) := \det (A)^{\frac{1}{2}}
\int_{\IR^n}  f(\bm{y}) e^{-2\pi i B(\bm{x},\bm{y})} d^n \bm{y} . \qedhere
\end{equation*}
\end{proof}

We are now ready to prove Theorem \ref{co:modularity}.
\begin{proof}[Proof of Theorem \ref{co:modularity}]
Dividing the statements of Theorem \ref{thm:wtPsi_modularity} by $ (i (w-\t)) ^{\frac{r}{2}} $ we obtain the claimed elliptic and modular transformations of $\wh{\Psi}_{Q,\bm{\mu}} ( \bm{z}, \bm{\zz}; \tau, w)$. Next we use Lemma \ref{la:generalized_error_asymptotic_form} to find
\begin{multline}
\lim_{w \to \t + i \infty} \wh{\Psi}_{Q,\bm{\mu},\bm{\ell},M} (\tau, w)
= 
\sum_{\bm{n} \in \LL + \bm{\mu} + \frac{\bm{\ell}}{2}}
\sgn ( B ( M,  \bm{n} ))	e^{\pi i B \lp \bm{n}, \bm{\ell} \rp} q^{Q(\bm{n})} 
 \\
+ \sum_{\substack{S \subseteq [r] \\ S \neq [r] }} E_{Q, M_{S} } ( \bm{0})
\sum_{\substack{  \bm{n} \in \LL + \bm{\mu} + \frac{\bm{\ell}}{2} \\   B(M_{S}, \bm{n})=\bm{0} }}
\sgn \left(B\left(M_{[r] \setminus S},  \bm{n} \right) \right)  e^{\pi i B \lp \bm{n}, \bm{\ell} \rp} q^{Q(\bm{n})} .
\label{eq:widehat_psi_limit}
\end{multline}
The first term on the right-hand side reproduces the false theta function $F_{Q,\bm{\mu}, \bm \ell,M} (\tau)$ defined in equation \eqref{false.gen}, whereas the second term consists of corrections built from false (or ordinary) theta functions with lower-dimensional lattices and fewer sign function factors.

Finally, using Lemma \ref{lem:EQMComplexRecursion}  we find
\begin{multline*}
\wh{\Psi}_{\bm{\mu}} (\tau, w)
=
E_{Q,M}( \bm{0})  \sum_{\bm{n} \in \LL + \bm{\mu} + \frac{\bm{\ell}}{2}} e^{\pi i B \lp \bm{n}, \bm{\ell} \rp} q^{Q(\bm{n})} 
+
\sum_{k=1}^r \int_{\t}^{w} 
\sum_{\bm{n} \in \LL + \bm{\mu} + \frac{\bm{\ell}}{2}}  
B(\widehat{\bm{m}}_{\bm{k}}, \bm{n}) e^{ \pi i B(\widehat{\bm{m}}_{\bm{k}}, \bm{n})^2 w_1 }
\\ \times
e^{\pi i B \lp \bm{n}, \bm{\ell} \rp}
E_{  Q,M_{[r] \setminus \{k\}\perp \{k\}}}  \left( -i  \sqrt{i(w_1-\t)}  \bm{n} \right) 
q^{Q(\bm{n})-\frac{1}{2} B(\widehat{\bm{m}}_{\bm{k}}, \bm{n})^2}
\frac{d w_1}{\sqrt{i(w_1-\t)}}  .
\end{multline*}
When all the vectors in $M$ are elements of $\LL$, the lattice sum in the integrand can be decomposed into unary theta functions along the direction of the lattice vectors $\bm{m_k}$ and completed false theta functions on $(n-1)$-dimensional lattices orthogonal to $\bm{m_k}$ and that include $r-1$ vectors in $M_{[r] \setminus \{k\}\perp \{k\}}$ (which up to rescaling also lie in this orthogonal lattice) appearing in the generalized error functions. 
Therefore, the derivative of $w\mapsto\wh{\Psi}_{\bm{\mu}} (\tau, w)$ is of the form of equation \eqref{eq:false_completion_derivative} with $r_j = -\frac{1}{2}$ for each $j$ and we find that (as can be seen inductively) $\wh{\Psi}_{\bm{\mu}} (\tau, w)$ is the completion of a depth $r$, weight $\frac{n}{2}$ false modular form.
\end{proof}

\begin{rem}
It is useful to emphasize here that the right-hand side of equation \eqref{eq:widehat_psi_limit} is what we call a false modular form (if the vectors in $M$ are elements of $\LL$). This includes the false theta function $F_{Q,\mu, \bm \ell,M} (\tau)$ as its leading contribution, but it also contains false (or ordinary) theta functions on lower-dimensional lattices (and with fewer sign function factors). Conversely, unless these corrections vanish, the false theta function $F_{Q,\mu, \bm \ell,M} (\tau)$ should be recognized as a linear combination of false modular forms including one term of depth $r$ and weight $\frac{n}{2}$ and with the rest having lower weights and depths. 

As an illustration, see part \ref{prop_part:A3_psi_boundary} of Proposition \ref{prop:A_3_false_theta} and part \ref{prop_part:A3_phi_boundary} of Proposition \ref{prop:PhiA3}, where rank three false theta functions are supplemented with rank one false theta functions in order to form a false modular form with a well-defined weight and depth. In other words, if we omit these lower rank corrections and only restrict our attention to the highest rank false theta function, then the resulting object is a combination of objects with different weights and depths.
\end{rem}

An important upshot here is that by starting with the modular completion \eqref{eq:hat_psi_definition}, one can automatically deduce the correct linear combinations appearing as limits of objects with modular properties. The modular completions $\wh{\Psi}_{\bm{\mu}} ( \bm{z}, \bm{\zz}; \tau, w)$ are also useful for developing differential operators that are covariant under modular transformations. In turn, these operators lead us to a much larger class of false modular forms and this is the topic that we turn our attention to next. We start by defining the operators
\begin{align}\label{eq:diff_operators}
R_k  &:= 
\del_\t + \frac{(\bm{z}-\bm{\zz})^T}{\t-w} \bm{\del_z} + 2 \pi i Q\lp \frac{\bm{z}-\bm{\zz}}{\t-w} \rp + \frac{k}{\t-w},\quad L := (\t-w) \lp (\t-w) \del_w + (\bm{z} - \bm{\zz})^T \bm{\del_\zz} \rp ,
\notag
\\
\bm{\CD} &:= \bm{\del_z} + \frac{2 \pi i}{\tau-w} A ( \bm{z}-\bm{\zz} ),
\ \  \andd \ \ 
\bm{\cd} := (\t-w) \bm{\del_\zz} .
\end{align}
These operators implicitly depend on the quadratic form $Q$ (through the associated symmetric matrix $A$) and they are the natural generalizations of covariant differential operators acting on real-analytic Jacobi forms as weight raising and lowering operators.

\begin{prop}\label{prop:ModularlyCovariant}
Let $r \in \{0,1\}$ and let $\bm{g} : \mathbb{C}^n \times \mathbb{C}^n \times (\mathbb{H} \times \mathbb{H} \setminus \{ \t = w \} ) \to \mathbb{C}^m$ be a vector-valued function such that $(i (w-\t))^{\frac{r}{2}}\bm{g} (\bm{z}, \bm{\zz}; \tau,w)$ is holomorphic on $ \mathbb{C}^n \times \mathbb{C}^n \times (\mathbb{H} \times \mathbb{H} \setminus \{ \t = w \} )$. 
Consider the following action of the Jacobi group $\SL_2 (\IZ) \ltimes \LL^{2}$ on such functions 
\begin{align}
	\bm{g} |_{k,Q} (\g, \bm{m}, \bm{r}) (\bm{z}, \bm{\zz}; \tau,w) 
	&:=
	\chi_{\t,w} (\g)^r \, (c \t + d)^{-k}
	\, e^{2 \pi i \lp - \frac{c Q(\bm{z}+\bm{m}\tau+\bm{r})}{c \t + d} 
		+ Q(\bm{m}) \t + B(\bm{m},\bm{z}) + Q(\bm{m}+\bm{r}) \rp}
	\notag
	\\ & \qquad \quad \times
	\rho(\g)^{-1}
	\bm{g} \lp \frac{\bm{z}+\bm{m}\tau+\bm{r}}{c \t +d}, \frac{\bm{\zz}+\bm{m}w+\bm{r}}{c w +d}; 
	\frac{a \tau + b}{c \tau + d}, \frac{a w+b}{cw+d} \rp,
	\label{eq:Jacobi_action}
\end{align}
where $\g = \pmat{a & b \\ c & d} \in \SL_2 (\IZ)$, $\bm{m},\bm{r} \in \LL$, 
\begin{equation*}
	\chi_{\t,w} (\g) := \sqrt{\frac{i(w-\t)}{(c \t+d)(c w +d)} } \frac{\sqrt{c \t+d} \sqrt{c w+d} }{ \sqrt{i(w-\t)} } \in \{ \pm 1\},
\end{equation*}
and the multiplier system $\rho(\g)$ is an $m \times m$ unitary matrix forming a projective representation of the modular group in a way that makes the action well-defined.
The differential operators defined in equation \eqref{eq:diff_operators} are covariant with respect to \eqref{eq:Jacobi_action} as follows
\begin{align*}
\CD_\a (\bm{g} |_{k,Q} (\g, \bm{m}, \bm{r})) &= \CD_\a (\bm{g})|_{k+1,Q} (\g, \bm{m}, \bm{r}),
\qquad
\cd_\a (\bm{g} |_{k,Q} (\g, \bm{m}, \bm{r})) = \cd_\a (\bm{g})|_{k-1,Q} (\g, \bm{m}, \bm{r}), 
\\
R_k (\bm{g} |_{k,Q} (\g, \bm{m}, \bm{r})) &= R_k (\bm{g})|_{k+2,Q} (\g, \bm{m}, \bm{r}),
\qquad
\hspace{.3cm}L (\bm{g} |_{k,Q} (\g, \bm{m}, \bm{r})) = L (\bm{g})|_{k-2,Q} (\g, \bm{m}, \bm{r}), 
\end{align*}
where $\a \in \{1,2,\ldots,n \}$.
Furthermore, these differential operators satisfy the algebraic relations
\begin{align*}
[\CD_\a, \CD_\b] = [\cd_\a, \cd_\b] &= 0, 
  \qquad \qquad
[\CD_\a, \cd_\b] = 2 \pi i A_{\a\b},
  \qquad 
L\circ R_k-R_{k-2}\circ L = k,
 \\ 
\cd_\a \circ R_k - R_{k-1} \circ \cd_\a &= - \CD_\a ,
\qquad \quad \,
[\cd_\a, L]  = 0 ,
\\  
\CD_\a \circ R_k - R_{k+1} \circ \CD_\a &= 0 ,
\qquad \qquad \ \ \,
[\CD_\a, L]  = \cd_\a .
\end{align*}
where $\a,\b \in \{1,2,\ldots,n \}$.
\end{prop} 
\begin{proof}
The proposition statement follows from a direct computation using \eqref{eq:diff_operators} and 
\eqref{eq:Jacobi_action}.
\end{proof}
Proposition \ref{prop:ModularlyCovariant} implies that the differential operators in \eqref{eq:diff_operators} behave like raising and lowering operators if they act on the modular completions of false theta functions defined in  \eqref{eq:hat_psi_definition}. The resulting objects transform under elliptic and modular transformations exactly as described in Theorem \ref{co:modularity} with only the weight changing by $1$ for $\CD_\a$ and $-1$ for $\cd_\a$. The operator $\CD_\a$ is specifically important for our applications below since it produces false modular forms that feature false theta functions with extra polynomial insertions. We clarify this statement in the next lemma.

\begin{lemma}\label{la:derivativeLimit}
Let $P(\bm{x})\in \IC[x_1,\dots, x_n]$ be a homogeneous polynomial of degree $d$ and the completed false theta function $\wh{\Psi}_{\bm{\mu}} (\bm{z}, \bm{\zz}; \tau, w)$ be defined as in equation \eqref{eq:hat_psi_definition}.
Then
\begin{equation*}
\wh{\varphi}_{\bm{\mu}} (\t,w) := \left[P\lp \frac{A^{-1} \bm{\CD} }{2 \pi i} \rp \lp\wh{\Psi}_{\bm{\mu}} (\bm{z}, \bm{\zz}; \tau, w)\rp\right]_{\bm{z} = \bm{\zz} = \bm{0}}
\end{equation*}
is a vector-valued bimodular form transforming as 
\begin{align*}
\wh{\varphi}_{\bm{\mu}} (\tau+1, w+1) &=
e^{2\pi iQ \lp \bm{\mu} + \frac{\bm{\ell}}{2} \rp}
\wh{\varphi}_{\bm{\mu}} (\tau, w), \\
\wh{\varphi}_{\bm{\mu}} \lp -\frac{1}{\tau}, -\frac{1}{w} \rp
&= \chi_{\t,w}^r  \t^{d+\frac{n}{2} } 
\frac{e^{-\frac{\pi i n}{4}} e^{-\pi i Q(\bm{\ell})}}{\sqrt{|\LL^* / \LL |}} 
\sum_{\bm{\nu} \in \LL^*/\LL} e^{-2\pi i B(\bm{\mu},\bm{\nu})}
\wh{\varphi}_{\bm{\nu}} ( \tau, w) .
\end{align*}
Its boundary value $\varphi_{\bm{\mu}} (\t) := \lim_{w\to \t+i\infty} \wh{\varphi}_{\bm{\mu}} (\t,w)$ is given by
\begin{multline*}
\varphi_{\bm{\mu}} (\t) 
= 
\sum_{\bm{n} \in \LL + \bm{\mu} + \frac{\bm{\ell}}{2}}
\sgn ( B ( M,  \bm{n} ))  P(\bm{n})  e^{\pi i B \lp \bm{n}, \bm{\ell} \rp} q^{Q(\bm{n})} 
 \\
+ \sum_{\substack{S \subseteq [r] \\ S \neq [r] }} E_{Q, M_{S} } ( \bm{0})
\sum_{\substack{  \bm{n} \in \LL + \bm{\mu} + \frac{\bm{\ell}}{2} \\   B(M_{S}, \bm{n})=\bm{0} }}
\sgn \left(B\left(M_{[r] \setminus S},  \bm{n} \right) \right) P(\bm{n})  e^{\pi i B \lp \bm{n}, \bm{\ell} \rp} q^{Q(\bm{n})} .
\end{multline*}
\end{lemma}
\begin{proof}
The modular properties are easy consequences of Theorem \ref{co:modularity} and Proposition \ref{prop:ModularlyCovariant}. To compute $\varphi_{\bm{\mu}} (\t)$ we start with the expression
\begin{align*}
\wh{\varphi}_{\bm{\mu}} (\t,w)
&=
\left[P\lp \frac{A^{-1}}{2 \pi i} \bm{\del_z} + \frac{\bm{z}}{\tau-w} \rp
\lp\wh{\Psi}_{\bm{\mu}} (\bm{z}, \bm{0}; \tau, w)\rp \right]_{\bm{z} =  \bm{0}}
\\
&=
\sum_{\bm{n} \in \LL + \bm{\mu} + \frac{\bm{\ell}}{2}}
E_{Q,M} \lp - i \sqrt{i (w-\t)} \bm{n} \rp  
P(\bm{n})  e^{\pi i B \lp \bm{n}, \bm{\ell} \rp} q^{Q(\bm{n})}  +
O \lp |w-\t|^{- \frac{1}{2} } \rp 
\end{align*}
as $w \to \t + i \infty$. 
The first term on the second line comes from $[P(\frac{A^{-1}}{2 \pi i} \bm{\del_z}) (e^{2 \pi i B \lp \bm{n},\bm{z} \rp})]_{\bm{z} =  \bm{0}} = P(\bm{n})$. The remaining terms come from the action of $\bm{\del_z}$ on the factors $\frac{\bm{z}}{\tau-w}$ appearing in the differential operators and on the factor $E_{Q,M}(- i \sqrt{i (w-\t)}(\bm{n} + \frac{\bm{z}}{\t-w}))$. To prove the bound $O(|w-\t|^{- \frac{1}{2} } )$, note that these terms all appear with a positive power of $(i (w-\t))^{-\frac{1}{2}}$ multiplying a term bounded by ordinary theta functions (possibly with polynomial insertions) in $\t$ or $w$. 
This is thanks to Lemma \ref{lem:generalized_error_derivative} and the fact that generalized error functions are bounded by one if they have real parameters. Finally, taking the limit $w \to \t + i \infty$ using Lemma \ref{la:generalized_error_asymptotic_form} proves the lemma statement.
\end{proof}

We finish our discussion of the raising operator $\CD_\a$ with a small lemma that simplifies the objects in Lemma \ref{la:derivativeLimit} when the relevant polynomial $P$ is harmonic.

\begin{lemma}\label{lem:HarmonicDerivative}
Let $P(\bm{x})\in \IC[x_1,\dots, x_n]$ be a polynomial that is harmonic with respect to the metric induced by the quadratic form $Q$, i.e., $\bm{\del_x}^T A^{-1} \bm{\del_x} (P (\bm{x})) = 0$. Then we have
\begin{equation*}\label{eq:diff_op_harmonic}
\left[P\lp \frac{A^{-1} \bm{\CD} }{2 \pi i} \rp \lp f(\bm{z}, \bm{\zz})\rp\right]_{\bm{z} = \bm{\zz} = \bm{0}}
=
\left[P\lp \frac{A^{-1} \bm{\del_z} }{2 \pi i} \rp \lp f(\bm{z}, \bm{\zz})\rp\right]_{\bm{z} = \bm{\zz} = \bm{0}} .
\end{equation*}
\end{lemma}

\begin{proof} 
Without loss of generality we can assume the polynomial $P$ is homogenous and of the form $P(\bm{x}) = \sum_{\a_1, \ldots, \a_\ell \in \{1,\ldots,n\}} P_{\a_1, \dots, \a_\ell} x_{\a_1} \cdots x_{\a_\ell}$, where $P_{\a_1, \dots, \a_\ell}$ is symmetric in its indices.
Since $\lp A^{-1} \bm{\CD} \rp_\a =\sum_{\g=1}^n (A^{-1})_{\a \g} \del_{z_\g} + 2 \pi i  \frac{z_\a - \zz_\a}{\t - w}$, the difference between the two sides of the equation in the lemma statement consists of terms in which $\sum_{\g=1}^n (A^{-1})_{\a_j \g} \del_{z_\g}$ from a $\lp A^{-1} \bm{\CD} \rp_{\a_j}$ acts on a $(z_{\a_k} - \zz_{\a_k})$ from another $\lp A^{-1} \bm{\CD} \rp_{\a_k}$. All such terms feature a factor of the form
$\sum_{\a_j, \a_k\in \{1,\ldots,n\}} (A^{-1})_{\a_j, \a_k} P_{\a_1, \dots, \a_\ell}$ and these are all zero by the harmonicity of $P$.
\end{proof}

The relations in Proposition \ref{prop:ModularlyCovariant} are useful in understanding completions of false modular forms from a representation theoretic perspective, which we briefly elaborate upon in the remainder of this section. As usual we denote the Lie algebra of $2 \times 2$ traceless matrices by $\mathfrak{sl}_2(\mathbb{C}):=\mathrm{span}\{e,f,h\}$, where $e$, $f$, and $h$ form the standard Chevalley basis. Recall that \textit{irreducible} {\em weight} $\mathfrak{sl}_2(\mathbb{C})$-modules are highest weight modules, lowest weight modules, or Harish--Chandra modules from the so-called intermediate series. These modules have one-dimensional weight subspaces and can be constructed as follows. Consider $E_{r,\mu}:=\oplus_{j \in \mathbb{Z}} v_j$, where $r, \mu \in \mathbb{C}$. Then we have an action of $\mathfrak{sl}_2(\mathbb{C})$ defined via its generators
\begin{equation*}
e.v_j=-(\mu+j)v_{j-1}, \ \ h.v_j=(-2\mu-2j+r)v_j, \  f.v_j=(\mu+j-r)v_{j+1}.
\end{equation*}
This module is irreducible if and only if $\mu \notin \mathbb{Z}$ and $r - \mu \notin \mathbb{Z}$. Any irreducible  $\mathfrak{sl}_2(\mathbb{C})$-module is a subquotient of  $E_{r,\mu}$. 

Letting $\bm{z}=\bm{\zz}=\bm{0}$, it is implicit in Proposition \ref{prop:ModularlyCovariant} that the differential operators
\begin{equation*}
\partial_\tau+\frac{k}{\tau-w}, \ \  \ (\tau-w)^2 \partial_w,
\end{equation*}
map the space of bimodular forms of weight $(k,0)$ to the space of bimodular forms of weight $(k+2,0)$ and $(k-2,0)$, respectively. 
Let $\widehat{F}(w,\tau)$ be the completion of a false modular form of weight $k$ and consider the vector space $\mathcal{L}_{\widehat{F}}$ spanned by $\widehat{F}$ and 
applications of $L$ and $R$ on $\widehat{F}$. It is not difficult to see that 
\begin{equation*}
e \mapsto i L ,  \ \ \ f \mapsto i R_k, \ \ \ h \mapsto -(L \circ R_k-R_{k-2} \circ L)
\end{equation*}
defines an $\mathfrak{sl}_2(\mathbb{C})$-module structure on $\mathcal{L}_{\widehat{F}}$.  Next we present an interpretation of several simple examples using these modules. 

\noindent
{\bf Examples.}
\begin{enumerate}[leftmargin=*]
	\item An ordinary modular form $\widehat{F}(w,\tau)=F(\tau)$ of weight $k \in \frac12 \mathbb{N}$, depending only on $\tau$, is annihilated by $L$. The raising operator $R$ acts freely thus it generates an irreducible Verma module of highest weight $-k$ with $\widehat{F}$ as its highest weight vector. We can visualize $\widehat{F}$ and the corresponding module $\mathcal{L}_{\widehat{F}}$ as
	\begin{equation*}
		\underset{\widehat{F}}{\bullet} \rightleftarrows  \circ \rightleftarrows \circ \rightleftarrows \cdots
	\end{equation*}
	where right (resp. left) arrows indicates an action of the raising (resp. lowering) operators. This module is a submodule of $E_{-k,0}$.

	\item Let  $\widehat{F}(w,\tau)$ be the completion of a false modular form of weight $\frac12$ coming from a unary false theta function $F(\t)$ (so $\widehat{F}(\t,w)$ is an Eichler-type integral). We have $L(\widehat{F}(w,\tau))=(\tau-w)^{\frac32} \theta'(w)$, where $\theta'(w)$ is a weight $\frac32$ theta function. On the other hand,  $R_{-\frac{3}{2}}(L(\widehat{F}(w,\tau)))=0$. Therefore we obtain an indecomposable  $\mathfrak{sl}_2(\mathbb{C})$-module
	\begin{equation*}
	\cdots \rightleftarrows \circ \overset{L}{\longleftarrow} \underset{\widehat{F}}{\odot}  \rightleftarrows \circ  \rightleftarrows \cdots
	\end{equation*}
	It is easy to see that this module is isomorphic to $E_{-\frac32,-\frac12}$.

	\item We have (with $\widehat{E}_2(\tau,w)$ defined in \eqref{E+w}) that $L(\widehat{E}_2(\tau,w))$ is a constant annihilated by $R_0$ and $L$. This leads to the following diagram for $\mathcal{L}_{\widehat{E}_2}$
	\begin{equation*}
	\bullet \leftarrow \underset{\widehat{E}_2}{\odot} \rightleftarrows  \circ \rightleftarrows \circ \rightleftarrows \cdots 
	\end{equation*}
	This module is isomorphic to the submodule of $E_{0,1}$ spanned by $v_j$, $j \geq -1$.
\end{enumerate}

More complicated examples will be studied in great detail elsewhere.

\section{Characters of the $W^0(p)_{A_n}$-vertex algebras for $n=1$, $2$, and $3$}\label{sec:w_algebra_characters}

With the necessary groundwork at hand, our aim in the rest of the paper is to work out the details of the formalism above on a number of examples. For that purpose, we first turn to affine $W$-algebras. These are important examples of vertex algebras defined via quantum Hamiltonian reduction and are widely studied both in mathematics and physics. 
For certain special values of the central charge, any principal affine $W$-algebra of ADE type admits a remarkable vertex algebra extension denoted by $W^0(p)_{R}$,
where $R$ is a root lattice and $p\in\N_{\ge2}$ is a scaling parameter.
It was previously observed in \cite{BM2,CM2} that the character
\begin{equation*}
{\rm ch}\left[W^0(p)_{A_n}\right](\tau) :={\rm tr}_{W^0(p)_{A_n}} q^{L(0)-\frac{c}{24}},
\end{equation*}
where $L(0)$ is the degree operator and $c$ is the central charge, can be written as an alternating sum of derivatives of higher rank partial theta functions.
Creutzig and the third author introduced and studied a regularized version of the character, ${\rm ch}[W^0(p)_{A_n}]^{\bm{\varepsilon}}(\tau)$, depending on charge variables $\bm \varepsilon \in \mathbb{C}^n$. Although regularized characters nicely connect with quantum group invariants (see \cite{CM2} for preliminary results), they do not transform as modular (or Jacobi) forms.

Our goal in this section is to investigate the functions $\eta(\tau)^n {\rm ch}[W^0(p)_{A_n}](\tau)$ for $1 \leq n \leq 3$ using the results from Section \ref{sec:higher_depth_false_theta}. We first discuss how they can be expressed in terms  
of false theta functions with polynomial insertions. Then we give modular completions that organize into vector-valued bimodular forms transforming as in (\ref{2-modular}) and whose boundary values combine into the characters we study. The price we have to pay for this modularity is that the completed characters are no longer partition functions on a two-torus, since they also depend on a second modular variable $w \in \mathbb{H}$. It would be very interesting to provide a physical interpretation for these bimodular completions either through $2d$ conformal field theories, half-indices for $3d$ theories, or chiral sectors of $4d$ superconformal field theories, where false theta functions make their appearance.

Another interesting direction would be to generalize the discussion here to more general graded traces called one-point functions (on the torus). For a vertex operator algebra $V$ and a graded $V$-module $M$, these are defined through the ``insertion'' of a homogeneous element  $v \in V$ as 
\begin{equation*}
T_M(v,\tau):={\rm tr}_{M} o(v) q^{L(0)-\frac{c}{24}},
\end{equation*}
where $o(v):=v_{{\rm deg}(v)-1}$ is the degree-preserving operator acting on $M$. The special case $v={\bf 1}$, the ``vacuum" vector, yields the usual character.
Since degree zero operators $o(v)$ act as generalized differential operators on a $W^0(p)_{A_n}$-module, the modular properties of completed graded traces $\hat{T}_M(v,\tau,w)$ can be unravelled using Proposition \ref{prop:ModularlyCovariant} and Lemma \ref{la:derivativeLimit}. We will pursue this direction in our future work.

\subsection{Generalities on $W$-algebra characters of type $A_n$}\label{sec:generalities_W_algebra}

We start by introducing the notation needed to describe the details of the characters that we study in this section. We denote the simple roots of the root lattice $A_n$ by $\bm{\a_1},\ldots,\bm{\a_n}$ and work with coordinates with respect to this basis whenever we use explicit coordinates to represent vectors.
The standard inner product on $A_n$ is denoted by $B_n(\cdot,\cdot)$ (and the associated quadratic form by $Q_n(\cdot)$), which satisfies
\begin{equation*}
B_n(\bm{\a_j},\bm{\a_{k}}) = 2 \d_{j,k} - \d_{j,k+1} - \d_{j+1,k}
\where j,k \in \{1,\ldots, n\} .
\end{equation*}
We then let $\bm{\b_1}, \ldots, \bm{\b_n}$ denote the dual basis for $A_n^*$, for which ${B_n(\bm{\a_j}, \bm{\b_k} ) = \d_{j,k}}$, denote the Weyl vector by $\bm{\rho}$, and the associated Weyl group by $W$ (where $n$ is implicit in the last two cases).

The vacuum character ${\rm ch}[W^0(p)_{A_n}](\tau)$ can be expressed as a lattice sum over the rescaled lattice $\LL := \sqrt{p} A_n$, $p \geq n+1$. To specify it, we fix representatives for conjugacy classes in $\LL^* / \LL$ as
\begin{equation*}
\bm{\l} := \sqrt{p} \bm{\wh{\l}} + \bm{\bar{\l}},
\end{equation*}
where $\bm{\wh{\l}} \in \{ \bm{0}, \bm{\b_1}, \ldots, \bm{\b_n} \}$ is a representative from a conjugacy class in $A_n^*/A_n$ and 
\begin{equation*}
\bm{\bar{\l}} := \sum_{j=1}^n (1-s_j) \frac{\bm{\b_j}}{\sqrt{p}},
\quad \mbox{for }
s_j \in \{ 1,2,\ldots, p \} .
\end{equation*}
Then characters of atypical irreducible modules $W^0 (p, \bm{\b} )_{A_n}$ of $W^0 (p)_{A_n}$ parametrized by $\bm{\b} := \bm{\l} + \sqrt{p} \bm{\g}$ for some $\bm{\g} \in A_n$, are proposed to be (see \cite[formula (5.2)]{BM2}; see also \cite{Sugimoto,Sugimoto2} for a rigorous derivation of character formulas for $W(p)_{A_n}$)
\begin{equation*}
\ch \lb W^0 (p, \bm{\b} )_{A_n} \rb (\t) := \mathrm{CT}_{[\bm{\zeta}]} \left(  
\bm{\zeta}^{-\bm{\wh{\l}} - \bm{\g}}
\ch \lb W (p, \bm{\l} )_{A_n} \rb (\t, \bm{z} )
\right),
\end{equation*}
where 
\begin{equation*}
\ch \lb W (p, \bm{\l} )_{A_n} \rb (\t, \bm{z} ) := 
\frac{1}{\eta(\tau)^{n}} \sum_{w\in W} \sum_{\bm{\alpha}\in A_n} (-1)^{\ell(w)}
q^{p  \, Q_n \left( \bm{\a} +\bm{\rho} +\bm{\wh{\l}} + \frac{\bm{\bar{\l}}}{\sqrt{p}} -  \frac{\bm{\rho}}{p}  \right) }
\frac{\bm{\zeta}^{w\left( \bm{\a} + \bm{\rho} + \bm{\wh{\l}}  \right) -  \bm{\rho}  }}{\DD (\bm{\zeta})}.
\end{equation*}
Moreover the Weyl denominator is given by
\begin{equation*}
\Delta({\bm  \zeta}):=\prod_{\bm{\a} \in\DD_-}\lp 1-\bm{\zeta}^{\bm{\a}} \rp
\end{equation*}
with $\DD_-:=\{ - \sum_{j=k}^\ell \bm{\a_j} : 1\leq k\leq \ell \leq n\}$ denoting the set of negative roots. 
By expanding the Weyl denominator as
\begin{equation*}
\frac{1}{\Delta(\bm{\zeta})}=\sum_{\bm{\a}\in A_n} K_n(\bm{\a})  {\bm{\zeta}}^{-\bm{\a}},
\end{equation*}
where $K_n$ denotes the $A_n$-Kostant partition function,\footnote{Given a vector $\bm{\a} \in A_n$, the {\it Kostant partition function} $K_n(\bm{\a})$ is the number of ways to express $\bm{\a}$ as a linear combination of positive roots with nonnegative integer coefficients.} we obtain
\begin{multline*}
\ch \lb W^0 (p, \bm{\b} )_{A_n} \rb (\t) 
\\
= 
\frac{1}{\eta\left(\tau\right)^{n}} \sum_{w\in W} \sum_{\bm{\alpha}\in A_n} \left(-1\right)^{\ell\left(w\right)}
K_n\left(w\left( \bm{\a} + \bm{\rho} + \bm{\wh{\l}}  \right) -  \bm{\rho} -\bm{\wh{\l}} - \bm{\g}\right)
q^{p  \, Q_n \left( \bm{\a} +\bm{\rho} +\bm{\wh{\l}} + \frac{\bm{\bar{\l}}}{\sqrt{p}} -  \frac{\bm{\rho}}{p}  \right) }  .
\end{multline*}
The special case $\bm{ \b } = \bm{0}$ corresponds to the vacuum character of $W^0(p)_{A_n}$
\begin{equation}\label{eq:w_algebra_An_vacuum_character}
{\rm ch}\left[W^0(p)_{A_n}\right](\tau) = 
\frac{1}{\eta\left(\tau\right)^{n}} \sum_{w\in W} \sum_{\bm{\alpha}\in A_n} \left(-1\right)^{\ell\left(w\right)}
K_n\left(w\left( \bm{\a} + \bm{\rho}  \right) -  \bm{\rho} \right)
q^{p \, Q_n \left( \bm{\a} +\frac{p-1}{p} \bm{\rho}  \right) }  .
\end{equation}

\begin{rem}
\mbox{ } \nolisttopbreak
\begin{enumerate}[leftmargin=*, label=\rm(\arabic*)]
\item[(1)] The structure of irreducible $W^0(p)_{A_n}$-modules  $W^0(p, \bm{\beta} )_{A_n}$ was analyzed in \cite{CM2,Sugimoto,Sugimoto2}. 
\item[(2)] The characters of $W^0(p, \bm{\beta} )_{A_n}$ also appear as Fourier coefficients of Schur indices of the generalized Argyres--Douglas theories of type $(A_{n},A_{(n+1)(p-1)-1})$.
\item[(3)]  In a recent work of Park \cite{Park} (see also \cite{Chung}), a higher rank generalization of the $\hat{Z}$-invariants of $3$-manifolds  \cite{GPPV} was introduced for any root system $L$. These $q$-series $\hat{Z}_M^L(q)$, where $L$ is a rank $n$ root system of ADE type and $M$ is the $3$-manifold obtained from $0$-surgery on twist knots $K_p$, agree with the functions $\eta(\tau)^n {\rm ch}[W^0(p)_L](\t)$.
\end{enumerate}
\end{rem}

\subsection{Vacuum character for the $W$-algebras of type $A_1$ and $A_2$}\label{sec:vac_character_A1_A2}
The simplest case is that of the $A_1$ root system, for which the vacuum character of the corresponding $W$-algebra can be expressed in terms of unary false theta functions \cite{BKM1}
\begin{equation*}
{\rm ch}\left[W^0(p)_{A_1}\right](\tau) = \frac{1}{\eta (\t)} \sum_{n \in \IZ + \frac{p-1}{2p} } \sgn (n) q^{p n^2} .
\end{equation*}
The modular properties of these characters can be studied by realizing the unary false theta functions here in terms of holomorphic Eichler integrals of weight $\frac{3}{2}$ unary theta functions. These modular properties then can be used to find Rademacher-type exact formulae for Fourier coefficients of these characters following the road map of \cite{BN}.\footnote{This question is studied by G.~Cesana as part of her upcoming thesis.}

Moving on to the case of $A_2$, first note that the Kostant partition function here is given by $K_2 (n_1 \bm{\a_1} + n_2 \bm{\a_2}) = \min (n_1,n_2) + 1$ for  $n_1,n_2 \in \N_0$ (with $K_2$ zero otherwise). It is then convenient to note that for any $\bm{\a} \in A_2 + \bm{\rho}$ we have
\begin{equation*}
\sum_{w \in W} (-1)^{\ell (w)} K_2 \lp w ( \bm{\a} ) -  \bm{\rho} \rp
=
\frac{1}{4} \sum_{w \in W} (-1)^{\ell (w)} \mathbb{P}_2 ( w (\bm{\a}) ),
\end{equation*}
where
\begin{equation*}
\mathbb{P}_2 (n_1 \bm{\a_1} + n_2 \bm{\a_2} ) := \sgn ( n_1 ) \, \sgn ( n_2 )(n_1 + n_2) .
\end{equation*}
Plugging this expression in equation \eqref{eq:w_algebra_An_vacuum_character} yields
\begin{align*}
\mathrm{ch}\left[W^0(p)_{A_2} \right](\tau) 
&=\frac{1}{2\eta(\tau)^2} \hspace{-.05cm} \sum_{\bm n  \in \Z^2 +\frac{p-1}{p} (1,1)}  \hspace{-.05cm}
 \sgn\lp n_1 + \frac 1p \rp \sgn \lp n_2 + \frac 1p \rp \lp n_1 +n_2 + \frac 2p \rp q^{ p \lp n_1^2 +n_2^2 -n_1n_2\rp }
\\
&\qquad -
\frac{1}{\eta(\tau)^2} \sum_{\bm n  \in \Z^2 + \frac{p-1}{p} (1,0)}  \sgn \lp n_1 + \frac 1p \rp \sgn \lp n_2 \rp \lp n_1 +n_2 + \frac 1p \rp 
q^{ p \lp n_1^2 +n_2^2 -n_1n_2\rp } .
\end{align*}
Up to false theta functions with fewer factors of sign functions, we can drop the $+ \frac{1}{p}$ factors inside the terms $\sgn ( n_j + \frac{1}{p} )$. In fact, the contribution of these one-dimensional false theta functions turns out to be zero and 
$\mathrm{ch}\left[W^0(p)_{A_2} \right] (\t)$ can be written as a linear combination of the false modular forms $\psi_{p,\bm{s}}^{[A_n]} (\t)$ and $\varphi_{p,\bm{s}}^{[A_n]} (\t)$ defined in Section \ref{sec:An_false_modular_forms} (see Propositions \ref{prop:psiA2_properties} and \ref{prop:PhiA2} for explicit expressions giving these false modular forms).
Three of the authors analyzed the asymptotics of $\mathrm{ch}\left[W^0(p)_{A_2} \right] (\t)$ towards rational numbers and showed that it has depth two quantum modular properties \cite{BKM1}. With the methods developed here, one can determine the modularity properties on the whole upper half-plane, which we describe in detail for $\psi_{p,\bm{s}}^{[A_n]} (\t)$ and $\varphi_{p,\bm{s}}^{[A_n]} (\t)$ in Section \ref{sec:An_false_modular_forms}.

\subsection{Vacuum character for the $W$-algebra of type $A_3$}
Our first task in order to study the vacuum character ${\rm ch}\left[W^0(p)_{A_3}\right](\tau)$ is to rearrange the Kostant partition function into a more useful form for our purposes.

\begin{lem}\label{la:KostkaRewrite}
For any $\bm{\a} \in A_3 + \bm{\rho}$ we have
\begin{equation*}
\sum_{w \in W} (-1)^{\ell (w)} K_3 \lp w ( \bm{\a} ) -  \bm{\rho} \rp
=
\frac{1}{16} \sum_{w \in W} (-1)^{\ell (w)} \mathbb{P}_3 ( w (\bm{\a}) ),
\end{equation*}
where $\mathbb{P}_3 (n_1 \bm{\a_1} + n_2 \bm{\a_2} + n_3 \bm{\a_3} ) :=  
 \sgn ( n_1 ) \, \sgn ( n_2 ) \, \sgn ( n_3 )(4 n_1 n_2 n_3  - n_2 )$.
\end{lem}
We give a proof in Appendix \ref{sec:Kostant_A3_lemma} based on the work of \cite{KP}, which yields a piecewise polynomial expression formula for the $\mathfrak{sl}_4$ Kostka numbers $K_{\bm{\lambda},\bm{0}}:=\mathrm{dim}(V(\bm{\lambda})_0)$, i.e., the dimension of the zero weight subspace of the representation $V(\bm{\lambda})$ based on the dominant weight $\bm{\lambda} \in A_3$.

With Lemma \ref{la:KostkaRewrite}, we can rewrite the vacuum character as
\begin{equation*}
\mathrm{ch} \left[W^0(p)_{A_3}\right](\tau) = 
\frac{1}{16 \eta\left(\tau\right)^{3}} 
\sum_{w\in W} (-1)^{\ell (w)} \sum_{\bm{\alpha}\in A_3 + w(\bm{\rho}) } 
\mathbb{P}_3  \lp \bm{\a} \rp
q^{p \, Q_3 \left( \bm{\a} -\frac{\bm{w(\rho)}}{p}   \right) }   
\end{equation*}
or more explicitly as
\begin{align*}
\mathrm{ch} \left[W^0(p)_{A_3}\right](\tau) = 
\frac{1}{8 \eta(\t)^{3}} 
&\sum_{\bm{r} \in \mathcal{S}} \varepsilon (\bm{r})
\sum_{\bm{n} \in \IZ^3 + \lp \frac{1}{2}, 0, \frac{1}{2} \rp } 
 \sgn ( n_1 ) \, \sgn ( n_2 ) \, \sgn ( n_3 )  n_2 (4n_1 n_3 - 1) 
\\ & \quad \times
q^{p \, \lp \left( n_1 - \frac{r_1}{p} \right)^2 + \left( n_2 - \frac{r_2}{p} \right)^2 + \left( n_3 - \frac{r_3}{p} \right)^2
- \left( n_1 - \frac{r_1}{p} \right) \left( n_2 - \frac{r_2}{p} \right) -\left( n_2 - \frac{r_2}{p} \right) \left( n_3 - \frac{r_3}{p} \right)  \rp }   ,
\end{align*}
where $\mathcal{S} := \mathcal{S}_1 \cup \mathcal{S}_2 \cup \mathcal{S}_{-1} \cup \mathcal{S}_{-2}$,
\begin{align*}
\mathcal{S}_1 &:= \left\{ \lp \frac{3}{2}, 2, \frac{3}{2} \rp, \lp \frac{1}{2}, 2, \frac{1}{2} \rp \right\},
\qquad  \ \ \ \,
\mathcal{S}_2 := \left\{ \lp \frac{3}{2}, 0, \frac{1}{2} \rp, \lp \frac{3}{2}, 1, -\frac{1}{2} \rp \right\},
\\
\mathcal{S}_{-1} &:= \left\{ \lp \frac{3}{2}, 1, \frac{3}{2} \rp, \lp \frac{1}{2}, -1, \frac{1}{2} \rp \right\},
\qquad
\mathcal{S}_{-2} := \left\{ \lp \frac{3}{2}, 0, -\frac{1}{2} \rp, \lp \frac{3}{2}, 2, \frac{1}{2} \rp \right\} ,
\end{align*}
and $\varepsilon (\bm{r}) := s$ if $\bm{r} \in \mathcal{S}_s$.

To go further, first note that we can change the factors $\sgn ( n_j )$ in $\mathrm{ch} \left[W^0(p)_{A_3}\right](\tau)$ for $j \in \{1,2,3\}$ to $\sgn ( n_j - \frac{r_j}{p} )$ up to corrections by false theta functions with fewer sign factors. In fact, for $p \geq 4$ it is easy to see that all such possible corrections vanish:
For $j \in \{1,3\}$ we have $\sgn ( n_j ) = \sgn ( n_j - \frac{r_j}{p} )$ if $n_j \in \IZ + \frac{1}{2}$ because $| \frac{r_j}{p} | < \frac{1}{2}$ for all $\bm{r} \in \mathcal{S}$, whereas for $j=2$ we have $  \sgn ( n_2 )n_2 = \sgn ( n_2 - \frac{r_2}{p} ) n_2$ if $n_2 \in \IZ$ because $| \frac{r_2}{p}| < 1$ for all $\bm{r} \in \mathcal{S}$ and the vanishing at $n_2 = 0$ is ensured by the explicit $n_2$ factor in front.
Therefore, for $p \geq 4$ we have
\begin{multline*}
\mathrm{ch} \left[W^0(p)_{A_3}\right](\tau) = 
\frac{1}{8  \eta(\t)^{3}} 
\sum_{\bm{r} \in \mathcal{S}} \varepsilon (\bm{r})
\sum_{\bm{n} \in \IZ^3 - \frac{\bm{r}}{p} + \lp \frac{1}{2}, 0, \frac{1}{2} \rp }
\sgn ( n_1 ) \, \sgn ( n_2 ) \, \sgn ( n_3 )    
\\  \times
\lp n_2 + \frac{r_2}{p} \rp
\lp 4 \lp n_1 + \frac{r_1}{p} \rp \lp n_3 + \frac{r_3}{p} \rp - 1 \rp 
q^{p \, \lp n_1^2 + n_2^2 + n_3^2 - n_1 n_2-n_2 n_3 \rp }   .
\end{multline*}

The modular properties of these functions can now be studied using the technology we are developing.
Recalling equations \eqref{eq:hat_psi_definition}, \eqref{eq:diff_operators}, and Lemma \ref{la:derivativeLimit}, we can define the modular completions of these vacuum characters by
\begin{align*}
&\wh{\mathrm{ch}} \left[W^0(p)_{A_3}\right](\tau, w) 
:= 
\frac{1}{8  \eta(\t)^{3}} 
\sum_{\bm{r} \in \mathcal{S}} \varepsilon (\bm{r})
\lb D\lp \frac{\bm{r}}{p} \rp \lp
\wh{\Psi}_{pQ_3, \lp \frac{1}{2}, 0, \frac{1}{2} \rp - \frac{\bm{r}}{p}, \bm{0}, (\bm{\b_1}, \bm{\b_2}, \bm{\b_3})}\rp ( \bm{z}, \bm{\zz}; \tau, w) 
\rb_{\bm{z} = \bm{\zz}= \bm{0}},
\end{align*}
where
\begin{equation*}
D\lp \bm{\nu} \rp := \lp \frac{\CD_1+ 2 \CD_2 + \CD_3}{4\pi i p} +\nu_2 \rp
\lp 4 \lp \frac{3\CD_1+ 2 \CD_2 + \CD_3}{8\pi i p}  + \nu_1 \rp 
\lp \frac{\CD_1+ 2 \CD_2 + 3\CD_3}{8\pi i p}   + \nu_3 \rp - 1 \rp   .
\end{equation*}
Thanks to Lemma \ref{la:derivativeLimit}, this object is a linear combination of bimodular forms with  weights $\frac{3}{2}, \frac{5}{2}, \frac{7}{2}$, and $\frac{9}{2}$ in $\tau$. Moreover, Lemma \ref{la:derivativeLimit} also implies that (for $p \geq 4$)
\begin{equation*}
\mathrm{ch} \left[W^0(p)_{A_3}\right](\tau) = 
\lim_{w \to \t + i \infty} \wh{\mathrm{ch}} \left[W^0(p)_{A_3}\right](\tau, w)  .
\end{equation*}

In the coming sections, we focus on the part $\sgn (n_1) \, \sgn (n_2) \, \sgn (n_3)(n_2 +\nu_2)$ of the insertion and give a more detailed discussion of it. In fact, looking at Lemma \ref{lem:konstant_A3_average2}, we can replace the $n_2$ factor with $\frac12(n_1+n_2+n_3)=\frac12B_3(\bm{\rho},\bm{\a})$ (up to corrections with fewer sign functions) and study false theta functions on $\sqrt{p} \, A_3$ with the insertion
\begin{equation*}
\sgn \lp  B_3 (\bm{\b_1}, \bm{\a}) \rp \sgn \lp  B_3 (\bm{\b_2}, \bm{\a}) \rp  
\sgn \lp  B_3 (\bm{\b_3}, \bm{\a}) \rp 
\end{equation*}
or 
\begin{equation*}
\sgn \lp  B_3 (\bm{\b_1}, \bm{\a}) \rp  \sgn \lp  B_3 (\bm{\b_2}, \bm{\a}) \rp  
\sgn \lp  B_3 (\bm{\b_3}, \bm{\a}) \rp B_3 (\bm{\rho}, \bm{\a}) .
\end{equation*}

\subsection{Examples of false modular forms on $\sqrt{p} A_n$}\label{sec:An_false_modular_forms}

In this section, we aim to investigate in greater detail a class of false modular forms associated with $A$-type root lattices that we encounter in our discussion above of vacuum characters for $W$-algebras of type $A_n$ for $n \in\{1,2,3\}$.
The main object under the microscope are the completed false theta functions on $A_n$ defined by
\begin{align}\label{eq:A_n_completed_false}
	\widehat{\psi}_{p,\bm{s}}^{[A_n]} (\bm{z}, \bm{\zz}; \t, w):= \ \ \ \sum_{\mathclap{\bm{n}\in A_n+\bm{\mu_n} (\bm{s})}}  \ E_{pQ_n, (\bm{\b_1},\dots,\bm{\b_n})}\left(-i\sqrt{i(w-\t)}\left(\bm{n}+\frac{\bm{z}-\bm{\zz}}{\t-w}\right)\right) q^{pQ_n(\bm{n})} e^{2\pi i p B_n(\bm{n},\bm{z})} ,
\end{align}
where the corresponding lattice conjugacy class (with respect to the quadratic form $p Q_n(\cdot)$) can be represented as
\begin{equation*}
	\bm{\mu_n} (\bm{s}) := \frac{1}{p} \sum_{m=1}^n s_m \bm{\b_m} ,\quad \mbox{where } s_j \in \mathbb{Z} \mbox{ for all } j \in [n].
\end{equation*}
The modular and elliptic transformation properties of $\widehat{\psi}_{p,\bm{s}}^{[A_n]} (\bm{z}, \bm{\zz}; \t, w)$ follow from Theorem \ref{co:modularity}.

Using this completed false theta function, we define two bimodular forms
\begin{equation*}
	\widehat{\psi}_{p,\bm{s}}^{[A_n]} (\t, w):=\widehat{\psi}_{p,\bm{s}}^{[A_n]} (\bm{0}, \bm{0}; \t, w) \andd \wh{\varphi}_{p,s}^{[A_n]}  (\t, w): = \frac{1}{2\pi i p} \bm{\rho}^T\bm{\partial_{\bm{z}}} \lp\widehat{\psi}_{p,\bm{s}}^{[A_n]}\rp (\bm{0}, \bm{0}; \t, w),
\end{equation*}
which have weights $\frac n2$ and $\frac n2+1$ in $\t$ under modular transformations in $(\t,w)$, simultaneously.\footnote{The latter result follows from Proposition \ref{prop:ModularlyCovariant} and Lemma \ref{la:derivativeLimit} while noting that $\bm\del_{\bm z}$ can be replaced with the modular covariant operator $\bm\CD=\bm{\del_z}+\frac{2\pi i}{\t-w}A(\bm z-\bm\zz)$ (as the contribution of the second term vanishes when $\bm z=\bm\zz=\bm0$).} As we demonstrate in detail by working out explicit Eichler integral expressions, their boundary values,
\begin{align*}
	\psi_{p,\bm{s}}^{[A_n]} (\t):=\lim_{w \to \t + i \infty} \wh{\psi}_{p,\bm{s}}^{[A_n]} (\t, w) \quad \mbox{and} \quad \varphi_{p,\bm{s}}^{[A_n]} (\t):=\lim_{w \to \t + i \infty} \wh{\varphi}_{p,\bm{s}}^{[A_n]} (\t, w) ,
\end{align*}
are depth $n$ false modular forms.\footnote{The Eichler integral representation immediately gives the derivative in $w$, which we can use to determine the depth by comparing to equation \eqref{eq:false_completion_derivative}. The Eichler integral representations we derive are also useful in giving an alternative understanding and derivation of the modular properties.}

Thanks to Lemma \ref{la:generalized_error_asymptotic_form}, up to corrections with fewer sign functions and lower-dimensional lattices, these false modular forms are equal to the false theta functions
\begin{equation*}
	\sum_{\bm{n}\in A_n+\bm{\mu_n} (\bm{s})} \prod_{m=1}^n \sgn (B_n (\bm{\b_m}, \bm{n})) \, q^{pQ_n(\bm{n})} \mbox{ and } \quad \ \ \sum_{\mathclap{\bm{n}\in A_n+\bm{\mu_n} (\bm{s})}} \quad\ \sgn (B_n (\bm{\rho}, \bm{n})) \prod_{m=1}^n \sgn (B_n (\bm{\b_m}, \bm{n})) \, q^{pQ_n(\bm{n})} .
\end{equation*}
Another goal of the upcoming discussion is to display these false modular forms quite explicitly including all the corrections.

\addtocontents{toc}{\SkipTocEntry}
\subsection*{$4.4.1.$ Depth one false modular forms over $\sqrt{p} A_1$}

We start with a quick review of the case $A_1$, which was already discussed in \cite{BN}, in order to set up notation and to establish results that is used in the cases $A_2$ and $A_3$.

The $A_1$-lattice can be explicitly realized as the lattice $\sqrt{2} \IZ$ with the usual Euclidean metric. With this realization we have $\a = 2 \b = \sqrt{2}$ and a complete set of conjugacy classes for $\mu_1(s) = \frac{s}{2p} \a$ is obtained by letting $s$ take integral values modulo $2p$. Then the completed false theta function in equation \eqref{eq:A_n_completed_false} more explicitly becomes
\begin{align}\label{eq:psi_A_1_explicit_form}
\widehat{\psi}_{p,s}^{[A_1]}( {z}, {\zz}; \t,w)=
\sum_{ {n}\in \Z+ {\frac{s}{2p}}} \erf\left(-i\sqrt{2\pi ip(w-\t)}\left( {n}+\frac{ {z}- {\zz}}{\t-w}\right)\right) q^{p {n}^2} e^{4\pi i p  {n} {z}}.
\end{align}
Now we use the Theorem \ref{co:modularity}, Lemma \ref{la:generalized_error_asymptotic_form}, and Lemma \ref{lem:EQMComplexRecursion} to study the properties of $\psi_{p,\bm{s}}^{[A_1]} (\t)$ and $\varphi_{p,\bm{s}}^{[A_1]} (\t)$. To state our results, it is useful to introduce the unary theta functions ($m,k \in \N$)
\begin{equation*}%\label{eq:unary_theta_functions}
\vartheta_{m, r} (z; \t) := \sum_{n \in \IZ + \frac{r}{2m}} q^{m n^2} e^{4 \pi i m n z} 
\andd
\vartheta_{m, r}^{[k]} (\t) := \lb \lp \frac{1}{4 \pi i m} \frac{\del}{\del z} \rp^k \vartheta_{m, r} (z; \t) \rb_{z=0}.
\end{equation*}

\begin{prop} \label{prop:psiA1_properties}
\mbox{ } \nolisttopbreak
\begin{enumerate}[label={\rm(\arabic*)},wide,labelindent=0pt]
\item The function $\widehat{\psi}_{p,s}^{[A_1]}$ has the modular transformation properties
\begin{align*}
\widehat{\psi}_{p,s}^{[A_1]}(\t+1,w+1)&=e^{\frac{\pi i s^2}{2p}} \widehat{\psi}_{p,s}^{[A_1]}(\t,w),
\\ 
\widehat{\psi}_{p,s}^{[A_1]}\left(-\frac{1}{\t}, -\frac{1}{w}\right) 
&=\chi_{\t,w} \frac{\sqrt{-i\t}}{\sqrt{2p}}  \sum_{r\pmod{2p}} e^{-\frac{\pi i sr}{p}} 
\widehat{\psi}_{p,r}^{[A_1]} (\t,w).
\end{align*}

\item We have
\begin{align*}
&\psi^{[A_1]}_{p, s} (\t) = \sum_{n\in \Z+\frac{s}{2p}}
\sgn(n) q^{pn^2}.
\end{align*}

\item \label{Psi1AsIntegral} We have the Eichler integral representation
\begin{align*}
\widehat{\psi}_{p,s}^{[A_1]}(\t,w)=\sqrt{2p} \int_{\t}^{w} \frac{\vartheta_{p,s}^{[1]}(w_1)}{\sqrt{i(w_1-\t)}} dw_1. 
\end{align*}
\end{enumerate}
In summary, $\psi^{[A_1]}_{p, s}$ forms a vector-valued false modular form of depth one and weight $\frac{1}{2}$.
\end{prop}

To discuss the properties of $\widehat{\varphi}_{p,s}^{[A_1]}$ next, we define the following regularized integral as in \cite{BKMN}
\begin{equation}\label{eq:regularized_integral}
\rint{\t}{w} \frac{f(w_1)}{(i(w_1-\t))^\frac32} dw_1 := \lim_{w_2 \to \t} 
\lp \int_{w_2}^w \frac{f(w_1)}{(i(w_1-\t))^\frac32}dw_1  + 2 i  \frac{f(\t)}{\sqrt{i(w_2-\t) }} \rp,
\end{equation}
where we assume both the integral and the one-sided limit are taken along the hyperbolic geodesic from $\t$ to $w$. Then a direct computation (together with Lemma \ref{la:derivativeLimit} and Lemma \ref{lem:EQMComplexRecursion}) yields the following result.

\begin{prop}\label{prop:PhiA1}
\mbox{ }  \nolisttopbreak
\begin{enumerate}[label={\rm(\arabic*)},wide,labelindent=0pt]
\item\label{PhiA1} The function $\widehat{\varphi}_{p,s}^{[A_1]}$ can be explicitly written as 
\begin{align*}
&\wh{\varphi}_{p,s}^{[A_1]} (\t, w) =\sum_{n\in \Z+\frac{s}{2p}} \erf\left(-i\sqrt{2\pi i p(w-\t)}n\right) n q^{pn^2}
+\frac{i}{\pi } \frac{1}{\sqrt{2 i p (w-\t)}} \sum_{n\in \Z+\frac{s}{2p}} e^{2\pi i p  n^2 w}. 
\end{align*}

\item The function $\widehat{\varphi}_{p,s}^{[A_1]}$ has  the modular transformation properties
\begin{align*}
\widehat{\varphi}_{p,s}^{[A_1]}(\t+1,w+1)&=e^{\frac{\pi i s^2}{2p}} \widehat{\varphi}_{p,s}^{[A_1]}(\t,w),
\\ 
\widehat{\varphi}_{p,s}^{[A_1]}\left(-\frac{1}{\t}, -\frac{1}{w}\right) 
&=\chi_{\t,w} \frac{\sqrt{-i}\t^\frac32}{\sqrt{2p}}  \sum_{r\pmod{2p}} e^{-\frac{\pi i sr}{p}} 
\widehat{\varphi}_{p,r,1}(\t,w).
\end{align*}

\item We have 
\begin{align*}
\varphi_{p,s}^{[A_1]} (\t)=\sum_{ n\in \Z+ \frac{s}{2p}} \sgn(n) n q^{p n^2}.
\end{align*}

\item We have the Eichler integral representation
\begin{align*}
\widehat{\varphi}_{p,s}^{[A_1]}(\t,w)&= \frac{1}{2\pi } \frac{1}{\sqrt{2p}}  \rint{\t}{w} \frac{\vartheta_{p,s}(w_1)}{(i(w_1-\t))^\frac32} dw_1.
\end{align*}
\end{enumerate}
In summary, $\varphi^{[A_1]}_{p, s}$ forms a vector-valued false modular form of depth one and weight $\frac{3}{2}$.
\end{prop}

%%%%%%%%%%%%%%%%%%%%%%%%%%%%%%%%%%%%%%%%%%%%%%%%%%%%%%%%
\addtocontents{toc}{\SkipTocEntry}
\subsection*{$4.4.2.$ Depth two false modular forms over $\sqrt{p} A_2$}
We next turn our attention to the case of $A_2$. First note that the fundamental weights $\{ \bm{\b_1}, \bm{\b_2} \}$ and the Weyl vector $\bm{\rho}$ look as follows in the $\bm{\a}$-basis (which we use to identify $A_2$ with $\IZ^2$)
\begin{equation*}
\bm{\b_1} = \frac{2}{3} \bm{\a_1} + \frac{1}{3} \bm{\a_2}, \quad
\bm{\b_2} = \frac{1}{3} \bm{\a_1} + \frac{2}{3} \bm{\a_2},
\andd
\bm{\rho} = \bm{\a_1} + \bm{\a_2} .
\end{equation*}
Also because $\bm{\b_1}+\bm{\b_2}, 3\bm{\b_1}, 3\bm{\b_2} \in A_2$, the conjugacy classes $\bm{\mu_2}(\bm{s}) =\frac{s_1}{p}\bm{\b_1}+\frac{s_2}{p}\bm{\b_2} = ( \frac{2s_1+s_2}{3p}, \frac{s_1+2s_2}{3p} )$ remain invariant under shifts of $\bm{s}$ as follows
\begin{equation*}
\bm{\mu_2}(\bm{s}) \equiv \bm{\mu_2}(s_1+3p,s_2), \  \bm{\mu_2}(s_1,s_2+3p), \  \bm{\mu_2}(s_1+p,s_2+p) \pmod{A_2} .
\end{equation*}
Therefore, we can represent all such conjugacy classes uniquely with a $\bm{\mu_2}(\bm{s})$ by restricting $\bm{s}$ to
\begin{equation*}
s_1\in \{0,1,\dots, 3p-1\}
\andd
s_2\in \{0,1,\dots, p-1\} .
\end{equation*}
Before we state our results, let us also recall that while applying Lemma \ref{lem:EQMComplexRecursion} we need to decompose $A_2 + \bm{\mu_2} (\bm{s})$ into components in the linear span of $\bm{\b_m}$ (for $m \in \{1,2\}$) and its orthogonal complement. The following lemma shows how this decomposition is done.
\begin{lem}
\label{lem:A_2_decomposition}
Decomposing $A_2 + \bm{\mu_2} (\bm{s})$ to parts in the linear span of $(\bm{\b_1},\bm{\a_2})$ or $(\bm{\b_2},\bm{\a_1})$ gives
\begin{align*}
A_2 + \bm{\mu_2} (\bm{s}) &=
\bigcup_{\d \in \{0,1\} } \left(
3\left(\Z+\frac{2s_1+s_2+3p\delta}{6p}\right) \bm{\b_1}+\left(\Z+\frac{s_2-p\delta}{2p}\right) \bm{\a_2}\right)
\\&=
\bigcup_{\d \in \{0,1\} } 
\left( 3\left(\Z+\frac{2s_2+s_1+3p\delta}{6p}\right) \bm{\b_2}+\left(\Z+\frac{s_1-p\delta}{2p}\right) \bm{\a_1}\right).
\end{align*}
On the linear span of ${\bm \a_1}$ or ${\bm \a_2}$, the quadratic form $Q_2$ reduces to $Q_1$.
\end{lem}
\begin{proof}
First note that an element of the conjugacy class $A_2 + \bm{\mu_2}(\bm{s})$ can be written uniquely as
\begin{align*}
\bm{n} = \left(m_1+\frac{2s_1+s_2}{3p}\right)\bm{\a_1}+\left(m_2+\frac{2s_2+s_1}{3p}\right)\bm{\a_2},
\quad \mbox{where }  m_1,m_2\in \mathbb{Z} .
\end{align*}
Since $\bm{\b_1}=\frac23\bm{\a_1}+\frac13\bm{\a_2}$, we have $\bm{\a_1}=\frac12 (3\bm{\b_1}-\bm{\a_2})$ and hence
\begin{align*}
\bm{n}= \left(\frac{3m_1}{2}+\frac{2s_1+s_2}{2p}\right) \bm{\b_1}+\left(m_2-\frac{m_1}{2}+\frac{s_2}{2p}\right)\bm{\a_2}.
\end{align*}
Making the change of variables $m_1\mapsto 2m_1+\delta$ with $m_1\in \Z$, $\delta\in\{0,1\}$ and then shifting $m_2\mapsto m_2 + m_1$ gives the first equality claim. The second one is proved in the same way.
\end{proof}
For example, we can use this result to decompose the theta function $\Theta^{[A_2]}_{\bm{\mu}}(\t) :=
\sum_{\bm{n}\in \IZ^2 +\bm{\mu}}  q^{pQ_2(\bm{n})}$ to unary theta functions as
\begin{align*}
\Theta^{[A_2]}_{\bm{\mu_2}(\bm{s})}(\t)
= \sum_{\d\in \{0,1\}} \vartheta_{3p, 2s_1+s_2+3p\d} (\t) \vartheta_{p,s_2-p\d}(\t)
=\sum_{\d\in \{0,1\}} \vartheta_{3p, 2s_2+s_1+3p\d} (\t) \vartheta_{p,s_1-p\d}(\t).
\end{align*}

We are now ready to state the properties of $\psi^{[A_2]}_{p, \bm{s}}$.
\begin{prop}
\label{prop:psiA2_properties}
Let $s_1\in \{0,1,\dots, 3p-1\}$ and $s_2\in \{0,1,\dots, p-1\}$.
\begin{enumerate}[label={\rm(\arabic*)},wide,labelindent=0pt]
\item The function $\wh{\psi}^{[A_2]}_{p,\bm{s}}$ has the modular transformation properties
\begin{align*}
\wh{\psi}^{[A_2]}_{p,\bm{s}}(\t+1,w+1) &= e^{ \frac{2\pi i}{3p}\lp s_1^2 +s_1s_2+s_2^2\rp} \wh{\psi}^{[A_2]}_{p,\bm{s}}(\t,w),\\
\wh{\psi}^{[A_2]}_{p,\bm{s}}\left(-\frac{1}{\t}, -\frac{1}{w}\right)
&=-\frac{i\t}{p \sqrt{3}} \sum_{r_1=0}^{3p-1} \sum_{r_2=0}^{p-1} e^{- \frac{2\pi i}{3p}
	\lp 2r_1s_1+r_2s_1+r_1s_2+2r_2s_2\rp} \wh{\psi}^{[A_2]}_{p,\bm{r}}(\t,w).
\end{align*}

\item We have 
\begin{align*}
&\psi^{[A_2]}_{p, \bm{s}} (\t) = \frac13 \delta_{\bm{s}=\bm 0} + \sum_{\bm{n}\in \IZ^2 +\bm{\mu_2}(\bm{s})}
\sgn(n_1)\sgn(n_2) q^{pQ_2(\bm{n})}.
\end{align*}

\item \label{prop_part:psiA2_Eichler} We have the following representation in terms of Eichler integrals
\begin{align*}
\wh{\psi}^{[A_2]}_{p,\bm{s}} (\t, w) =& \sqrt{6p} 
\int_{\t}^{w} \frac{h_{p,\bm{s}}(\t,w_1)}{\sqrt{i(w_1-\t)}} dw_1
+\frac13 \Theta^{[A_2]}_{\bm{\mu_2}(\bm{s})}(\t),
\end{align*}
\hspace{.75cm}where 
\begin{align*}
h_{p,\bm{s}}(\t,w):=\sum_{\d \in \{0,1\}} \left(\vartheta^{[1]}_{3p, 2s_1+s_2+3p\d} (w) \wh{\psi}^{[A_1]}_{p,s_2-p\d}(\t, w)
+\vartheta^{[1]}_{3p, 2s_2+s_1+3p\d} (w) \wh{\psi}^{[A_1]}_{p,s_1-p\d}(\t, w)\right).
\end{align*}

\end{enumerate}
In summary, $\psi^{[A_2]}_{p, \bm{s}}$ forms a vector-valued false modular form of depth two and weight one.
\end{prop}
\begin{proof}
\mbox{ } \nolisttopbreak
\begin{enumerate}[label={\rm(\arabic*)},wide,labelindent=0pt]
\item The modular transformations follow directly from Theorem \ref{co:modularity}.

\item The claim follows from Lemma \ref{la:generalized_error_asymptotic_form} and Remark \ref{rem:E2_zero_value}, which says that
\begin{equation*}
E_{pQ_{2}, (\bm{\b_1},\bm{\b_2})}\left(\bm{0}\right)=\frac{2}{\pi}\arcsin\left(
\frac{B_2(\bm{\b_1},\bm{\b_2})}{\sqrt{B_2(\bm{\b_1},\bm{\b_1}) \, B_2(\bm{\b_2},\bm{\b_2})}}
\right)=\frac13.
\end{equation*}

\item Specializing Lemma \ref{lem:EQMComplexRecursion} to our case we obtain
\begin{align}\nonumber
&E_{pQ_2, (\bm{\b_1},\bm{\b_2})}\left(-i\sqrt{i(w-\t)}\bm{x}\right)\\
\nonumber
&\hspace{.6cm}= \sqrt{\frac{3p}{2 }} B_2(\bm{\b_1},\bm{x}) \int_{\t}^{w}\frac{e^{\frac{3\pi ip}{2} B_{2}\left(\bm{\b_1},\bm{x}\right)^2 (w_1-\t)}}{\sqrt{i(w_1-\t)}} E_{pQ_{2}, \bm{\a_2}}\left(-i\sqrt{i(w_1-\t)}\bm{x}\right)   dw_1\\
\label{E2Q2}
&\hspace{1.2cm}+\sqrt{\frac{3p}{2}} B_2(\bm{\b_2},\bm{x}) \int_{\t}^{w}\frac{e^{\frac{3\pi ip}{2} B_{2}\left(\bm{\b_2},\bm{x}\right)^2 (w_1-\t)}}{\sqrt{i(w_1-\t)}} E_{pQ_{2}, \bm{\a_1}}\left(-i\sqrt{i(w_1-\t)}\bm{x}\right)   dw_1+\frac13. 
\end{align}
Now by Lemma \ref{lem:A_2_decomposition} one can write 
\begin{align}
&\sum_{\bm{n}\in \IZ^2+\bm{\mu_2}(\bm{s})} q^{pQ_2(\bm{n})} e^{\frac32 \pi i p  B_2(\bm{\b_1},\bm{n})^2 (w_1-\t)} E_{pQ_2,\bm{\a_2}} \left(-i\bm{n}\sqrt{i(w_1-\t)}\right) B_2(\bm{\b_1},\bm{n})
\notag
\\ & \qquad =
2\sum_{\delta\in\{0,1\}}  \sum_{m_1\in \Z + \frac{2s_1+s_2+3p\d}{6p}} m_1
e^{6\pi i p m_1^2 w_1} 
\sum_{m_2\in\Z +\frac{s_2-p\d}{2p}} \erf\left(-i\sqrt{2\pi i p(w_1-\t)} \, m_2\right) q^{pm_2^2}
\notag
\\& \qquad =
2\sum_{\d\in\{0,1\}} \vartheta^{[1]}_{3p, 2s_1+s_2+3p\d} (w_1) \wh{\psi}^{[A_1]}_{p,s_2-p\d}(\t, w_1).
\label{eq:A_2_erf_decomposed_to_A_1}
\end{align}
Using a similar computation for the second term in \eqref{E2Q2} gives the claim. \qedhere
\end{enumerate}
\end{proof}

Similarly, we give the properties of $\varphi^{[A_2]}_{p, \bm{s}}$ in the following proposition.

\begin{prop}\label{prop:PhiA2}
Let $s_1\in \{0,1,\dots, 3p-1\}$ and $s_2\in \{0,1,\dots, p-1\}$.
\begin{enumerate}[label={\rm(\arabic*)},wide,labelindent=0pt]
\item\label{PhiA2} The function $\wh{\varphi}_{p,\bm{s}}^{[A_2]}$ can be explicitly written as
\begin{align*}
\hspace{.65cm}&\wh{\varphi}_{p,\bm{s}}^{[A_2]} (\t, w) 
= 
\sum_{\bm{n}\in \Z^2+\bm{\mu_2}(\bm{s})} 
(n_1+n_2)E_{pQ_{2}, (\bm{\b_1},\bm{\b_2})}\left(-i\sqrt{i(w-\t)} \, \bm{n} \right)
q^{pQ_{2}(\bm{n})}
\\&  +\frac{\sqrt{3}i}{\pi \sqrt{2ip(w-\t)}}
\sum_{\d\in\{0,1\}} \left(\vartheta_{3p, 2s_1+s_2+3p\d} (w) \wh{\psi}^{[A_1]}_{p,s_2-p\d} (\t, w)
+\vartheta_{3p, 2s_2+s_1+3p\d} (w) \wh{\psi}^{[A_1]}_{p,s_1-p\d} (\t, w)\right) .\notag
\end{align*}

\item The function $\wh{\varphi}_{p,\bm{s}}^{[A_2]}$ has the modular transformation properties
\begin{align*}
\wh{\varphi}^{[A_2]}_{p,\bm{s}}(\t+1,w+1) &= e^{ \frac{2\pi i}{3p}\lp s_1^2 +s_1s_2+s_2^2\rp} \wh{\varphi}^{[A_2]}_{p,\bm{s}}(\t,w),
\\
\wh{\varphi}^{[A_2]}_{p,\bm{s}}\left(-\frac{1}{\t}, -\frac{1}{w}\right)
&=-\frac{i\t^2}{p \sqrt{3}} \sum_{r_1=0}^{3p-1} \sum_{r_2=0}^{p-1} e^{- \frac{2\pi i}{3p}
	\lp 2r_1s_1+r_2s_1+r_1s_2+2r_2s_2\rp} \wh{\varphi}^{[A_2]}_{p,r_1,r_2}(\t,w).
\end{align*}

\item We have
\begin{align*}
\varphi_{p,\bm{s}}^{[A_2]} (\t)= \sum_{\bm{n}\in \Z^2+\bm{\mu_2}(\bm{s})} \sgn(n_1)\sgn(n_2) (n_1+n_2) q^{pQ_2(\bm{n})}.
\end{align*}

\item\label{PhiA2_Eichler} We have the following representation in terms of Eichler integrals
\begin{align*}
\wh{\varphi}_{p,\bm{s}}^{[A_2]} (\t, w) =& 
\left(\frac13+\frac{\sqrt{3}}{\pi}\right)\sum_{\bm{n}\in \Z^2+\bm{\mu_2}(\bm{s})} (n_1+n_2) q^{pQ_2(\bm{n})} 
\\ & \qquad
+ \sqrt{6p}\int_{\t}^{w} \lp \frac{f_{p,\bm{s}}(\t,w_1)}{\sqrt{i(w_1-\t)}} 
+ \frac{1}{4 \pi p} \frac{g_{p,\bm{s}}(\t,w_1)}{(i(w_1-\t))^\frac32}
\rp dw_1 ,
\end{align*}
\hspace{.75cm}with 
\begin{align*}
f_{p,\bm{s}}(\t,w)&:= \sum_{\d\in \{0,1\}} \left(\vartheta_{3p, 2s_1+s_2+3p\d}^{[1]}(w)  \wh{\varphi}^{[A_1]}_{p,s_2-p\d} (\t,w)
+\vartheta_{3p, 2s_2+s_1+3p\d}^{[1]}(w)  \wh{\varphi}^{[A_1]}_{p,s_1-p\d} (\t,w)\right),
\\
g_{p,\bm{s}}(\t,w)&:=  \sum_{\d\in \{0,1\}} \left(
\vartheta_{3p, 2s_1+s_2+3p\d}(w)  \wh{\psi}^{[A_1]}_{p,s_2-p\d} (\t,w)
+\vartheta_{3p, 2s_2+s_1+3p\d}(w)  \wh{\psi}^{[A_1]}_{p,s_1-p\d} (\t,w) \right).
\end{align*}
\end{enumerate}
In summary, $\varphi^{[A_2]}_{p, \bm{s}}$ forms a vector-valued false modular form of depth two and weight two.
\end{prop}

\begin{rem}\label{rem:A2_varphi_Eichler_convergence}
Before we move on to the proof, let us remark on convergence of the integral given in part \ref{PhiA2_Eichler}.
 As can be seen from equation \eqref{eq:psi_A_1_explicit_form}
and part \ref{PhiA1} of Proposition \ref{prop:PhiA1} (or alternatively using their expression as Eichler integrals), as $w \to \t$ we have
\begin{align}\label{eq:psiA1_w_tau_limit}
\wh{\psi}^{[A_1]}_{p,s} (\t,w) &= - i \sqrt{8ip(w-\t)} \vartheta_{p,s}^{[1]}(\t) \lp 1 + O(w-\t) \rp,
\\ \label{eq:phiA1_w_tau_limit}
\wh{\varphi}^{[A_1]}_{p,s} (\t,w) &= \frac{i}{\pi } \frac{1}{\sqrt{2 i p (w-\t)}} \vartheta_{p,s} (\t) \lp 1 + O(w-\t) \rp .
\end{align}
The limiting behavior here, together with the fact that\footnote{
This identity follows from Lemma \ref{lem:A_2_decomposition} and writing $\bm{\rho} = \bm{\a_1} + \bm{\a_2}$ for the first expression and $\bm{\rho} = \bm{\b_1} + \bm{\b_2}$ for the second one.
}
\begin{align}
&\sum_{\bm{n}\in \Z^2+\bm{\mu_2}(\bm{s})} B_2(\bm{\rho},\bm{n}) \, q^{pQ_2(\bm{n})} 
= 2\sum_{\d\in\{0,1\}}  \left(  \vartheta_{3p, 2s_1+s_2+3p\d} (\t)
\vartheta^{[1]}_{p,s_2-p\d}(\t)
+\vartheta_{3p, 2s_2+s_1+3p\d} (\t)
\vartheta^{[1]}_{p,s_1-p\d}(\t) \right)
\notag \\
&\qquad \qquad \qquad 
= 2\sum_{\d\in\{0,1\}}  \left(  \vartheta^{[1]}_{3p, 2s_1+s_2+3p\d} (\t)
\vartheta_{p,s_2-p\d}(\t)
+\vartheta^{[1]}_{3p, 2s_2+s_1+3p\d} (\t)
\vartheta_{p,s_1-p\d}(\t) \right)
\label{eq:A_2_linear_theta_identity}
\end{align}
is crucial in rendering the integral in part \ref{PhiA2_Eichler} of Proposition \ref{prop:PhiA2} convergent near $w_1 = \t$ by canceling the part of the integrand behaving as $\frac{1}{w_1-\t}$.
\end{rem}

\begin{proof}
\mbox{ } \nolisttopbreak
\begin{enumerate}[label={\rm(\arabic*)},wide,labelindent=0pt]
\item We use Lemma \ref{lem:generalized_error_derivative} to compute
\begin{align*}
&\frac{1}{2\pi i p}\left[\bm{\rho}^T\bm{\partial_{z}}
\lp E_{pQ_{2}, (\bm{\b_1},\bm{\b_2})}\left(-i\left(\bm{n}+\frac{\bm{\zz}-\bm{z}}{w-\t}\right)\sqrt{i(w-\t)}\right)\rp\right]_{\bm{z}=\bm{\zz}=\bm{0}}
\\ &\hspace{3cm}
=\frac{\sqrt{3}i}{\pi \sqrt{2ip(w-\t)}}
\left( e^{\frac{3\pi i p}{2} B_{2}\left(\bm{\b_1},\bm{n}\right)^2(w-\t)} E_{pQ_{2}, \bm{\a_2}}\left(-i\sqrt{i(w-\t)}\bm{n}\right) \right.
\\& \hspace{7cm} \left.
+e^{\frac{3\pi i p}{2} B_{2}(\bm{\b_2},\bm{n})^2(w-\t)} E_{pQ_{2}, \bm{\a_1}}
\left(-i\sqrt{i(w-\t)}\bm{n}\right)
\right).
\end{align*}
Summing over $\bm{n}$ using Lemma \ref{lem:A_2_decomposition} (as in equation \eqref{eq:A_2_erf_decomposed_to_A_1})  gives the second line of the claimed identity. The contribution on the first line is obtained using $\frac{1}{2\pi i p} [ \bm{\rho}^T\bm{\partial_{z}}( e^{2\pi i pB_{2}(\bm{n},\bm{z})}) ]_{\bm{z} = \bm{0}}=B_{2}(\bm{n},\bm{\rho})$.

\item The modular transformations follow from Theorem \ref{co:modularity} and Lemma \ref{la:derivativeLimit}. 

\item Taking $w\to \t+i\infty$ in part \ref{PhiA2} using Lemma \ref{la:generalized_error_asymptotic_form} gives the claim.

\item  We insert equation \eqref{E2Q2} into the contribution on the first line of part \ref{PhiA2} to rewrite it as (while noting that $\bm{\rho}=\frac32 \bm{\b_1}+\frac12\bm{\a_2}=\frac12 \bm{\a_1}+\frac32\bm{\b_2}$)
\begin{align}
&\sum_{\bm{n}\in \Z^2+\bm{\mu_2}(\bm{s})} 
(n_1+n_2)E_{pQ_{2}, (\bm{\b_1},\bm{\b_2})}\left(-i\sqrt{i(w-\t)} \, \bm{n} \right)
q^{pQ_{2}(\bm{n})}
=
\frac13\sum_{\bm{n}\in \Z^2+\bm{\mu_2}(\bm{s})} B_2(\bm{\rho},\bm{n}) q^{pQ_2(\bm{n})}
\notag \\& \hspace{1.5cm}
+\sqrt{\frac{3p}{2}} \int_{\t}^{w} \sum_{\bm{n}\in \Z^2+\bm{\mu_2}(\bm{s})} e^{\frac{3\pi i p}{2} B_2(\bm{\b_1},\bm{n})^2w_1} E_{pQ_2,\bm{\a_2}}\left(-i\sqrt{i(w_1-\t)}\bm{n}\right) 
\notag \\&\hspace{3cm}
\times e^{\frac{\pi ip}{2} B_2(\bm{\a_2},\bm{n})^2\t} \left(\frac32 B_2(\bm{\b_1},\bm{n})^2+B_2(\bm{\b_1}, \bm{n})B_2\left(\frac{\bm{\a_2}}{2},\bm{n}\right)\right) \frac{dw_1}{\sqrt{i(w_1-\t)}}
\notag \\&\hspace{1.5cm}
+\sqrt{\frac{3p}{2}} \int_{\t}^{w} \sum_{\bm{n}\in \Z^2+\bm{\mu_2}(\bm{s})} e^{\frac{3\pi i p}{2} B_2(\bm{\b_2},\bm{n})^2w_1} E_{pQ_2,\bm{\a_1}}\left(-i\sqrt{i(w_1-\t)}\bm{n}\right) 
\notag \\&\hspace{3cm}
\times
e^{\frac{\pi ip}{2} B_2(\bm{\a_1},\bm{n})^2\t}  
\left(\frac32 B_2(\bm{\b_2},\bm{n})^2+B_2(\bm{\b_2}, \bm{n})B_2\left(\frac{\bm{\a_1}}{2},\bm{n}\right)\right) \frac{dw_1}{\sqrt{i(w_1-\t)}}.
\label{eq:auxiliary_eq_A_2_varphi}
\end{align}
Using Lemma \ref{lem:A_2_decomposition} and Proposition \ref{prop:PhiA1} \ref{PhiA1}, we can rewrite the second term here as 
\begin{align*}
\sqrt{\frac{3p}{2}} \int_{\t}^{w} &
\Bigg(
\frac{1}{\pi i p}\sum_{\d\in\{0,1\}} \frac{\partial}{\partial w_1}\left( \vartheta_{3p, 2s_1+s_2+3p\d}(w_1) \right) \wh{\psi}^{[A_1]}_{p,s_2-p\d} (\t,w_1)
\\
&\qquad +
2\sum_{\d\in\{0,1\}}  \vartheta_{3p, 2s_1+s_2+3p\d}^{[1]}(w_1)  \left(\wh{\varphi}^{[A_1]}_{p,s_2-p\d} (\t,w_1)
-\frac{i}{\pi}\frac{\vartheta_{p,s_2-p\d}(w_1)}{\sqrt{2 i p (w_1-\t)}}\right)  \Bigg)
\frac{dw_1}{\sqrt{i(w_1-\t)}} .
\end{align*}
Then we integrate by parts and use Proposition \ref{prop:psiA1_properties} \ref{Psi1AsIntegral}  and equation \eqref{eq:psiA1_w_tau_limit} to write
\begin{align*}
&\int_{\t}^{w}  \frac{\partial}{\partial w_1}\left( \vartheta_{3p, 2s_1+s_2+3p\d}(w_1) \right) \wh{\psi}^{[A_1]}_{p,s_2-p\d} (\t,w_1)
\frac{dw_1}{\sqrt{i(w_1-\t)}}
\\&\qquad \qquad  =
\frac{\vartheta_{3p, 2s_1+s_2+3p\d}(w)  \wh{\psi}^{[A_1]}_{p,s_2-p\d} (\t,w)}{\sqrt{i(w-\t)}} 
+i \sqrt{8p}\, \vartheta_{3p, 2s_1+s_2+3p\d}(\t)   
\vartheta_{p,s_2-p\d}^{[1]}(\t)
\\&\qquad \qquad \qquad +i
\int_{\t}^{w}   \vartheta_{3p, 2s_1+s_2+3p\d}(w_1) 
\left( \sqrt{2p}
\frac{\vartheta_{p,s_2-p\d}^{[1]}(w_1)}{w_1-\t} 
+\frac{1}2\frac{\wh{\psi}^{[A_1]}_{p,s_2-p\d} (\t,w_1)}{(i(w_1-\t))^\frac32} \right) dw_1.
\end{align*}
Note that the limiting values in equations \eqref{eq:psiA1_w_tau_limit} and \eqref{eq:phiA1_w_tau_limit} are crucial in ensuring that both of the integrals above are convergent near $w_1 = \t$. 
The third term on the right-hand side of equation \eqref{eq:auxiliary_eq_A_2_varphi} can be treated similarly (using the symmetry under the exchange of subscripts $1$ and $2$). Putting everything together while using equation \eqref{eq:A_2_linear_theta_identity} gives the result. \qedhere
\end{enumerate}
\end{proof}

\begin{rem}\label{rem:A2_limiting_value}
Using the Eichler integral representations in Proposition \ref{prop:psiA2_properties} and \ref{prop:PhiA2} together with Remark \ref{rem:A2_varphi_Eichler_convergence}, as $w \to \t$ we find that
\begin{align*}%\label{eq:psi_A2_w_tau_limit}
\wh{\psi}^{[A_2]}_{p,\bm{s}} (\t, w) &=
\frac13 \Theta^{[A_2]}_{\bm{\mu_2}(\bm{s})}(\t) 
+ O(w-\t) ,
\\
\wh{\varphi}_{p,\bm{s}}^{[A_2]} (\t, w) &=
\left(\frac13+\frac{\sqrt{3}}{\pi}\right)\sum_{\bm{n}\in \Z^2+\bm{\mu_2}(\bm{s})} (n_1+n_2) q^{pQ_2(\bm{n})} 
+ O(w-\t)  .
\end{align*}
\end{rem}

\begin{rem}
Recall that irreducible modules of the lattice vertex algebra $V_{\sqrt{p}A_2}$  are 
parametrized by the cosets $\Lambda^*/\Lambda$ and the associated $\SL_2(\mathbb{Z})$-module is determined by  the matrices $S$ and $T$

\begin{align*}
S_{\bm{\mu_2}(\bm{r}),\bm{\mu_2}(\bm{s})}
=\frac{1}{\sqrt{3p}} e^{\frac{2 \pi i}{3p} \left(2 r_1 s_1+r_2 s_1+r_1 s_2+2r_2 s_2\right)}
\andd
T_{\bm{\mu_2}(\bm{r}),\bm{\mu_2}(\bm{r})}=e^{2 \pi i \left(\frac{1}{3p}\left(2 r_1^2+2r_1 r_2+2r_2^2\right) - \frac{1}{12} \right)} .
\end{align*}
By Proposition \ref{prop:psiA2_properties}, we see that the $|\Lambda^*/\Lambda|$-dimensional space spanned by  
the completions $\frac{\wh{\psi}^{[A_2]}_{p,\bm{s}}(\t,w)}{\eta(\tau)^2}$ transforms with the same matrices $S$ and $T$. Similarly, from Proposition \ref{prop:PhiA2}, we get that  $\frac{\wh{\varphi}^{[A_2]}_{p,\bm{s}}(\t,w) }{\eta(\tau)^2}$ transforms under $\SL_2(\mathbb{Z})$ the same way but with weight one. In conclusion, from the perspective of vertex 
algebras, completed objects coming from $W^0(p)_{A_2}$-characters recover the modular data of the lattice vertex algebra of $\sqrt{p}A_2$.
\end{rem}

%%%%%%%%%%%%%%%%%%%%%%%%%%%%%%%%%%%%%%%%%%%%%%%%%%%%%%%%
\addtocontents{toc}{\SkipTocEntry}
\subsection*{$4.4.3.$ Depth three false modular forms over $\sqrt{p} A_3$}
We finally study the case of $A_3$. First, we write the fundamental weights $\{ \bm{\b_1}, \bm{\b_2}, \bm{\b_3} \}$ and the Weyl vector $\bm{\rho}$ in the $\bm{\a}$-basis (which we use to identify $A_3$ with $\IZ^3$)
%\footnote{
%Conversely, the simple roots can be expressed in terms of the fundamental weights as
%\begin{equation*}
%\bm{\a_1} = 2 \bm{\b_1} - \bm{\b_2}, \qquad
%\bm{\a_2} =-\bm{\b_1}+ 2 \bm{\b_2} - \bm{\b_3}, \qquad
%\bm{\a_3} = - \bm{\b_2} + 2 \bm{\b_3}
%\end{equation*}
%}
\begin{align*}
\bm{\b_1} &= \frac{1}{4} \lp 3 \bm{\a_1} +2\bm{\a_2} + \bm{\a_3} \rp, \quad  
\bm{\b_2} = \frac{1}{2} \lp \bm{\a_1} +2\bm{\a_2} + \bm{\a_3} \rp, \quad  
\bm{\b_3} = \frac{1}{4} \lp \bm{\a_1} +2\bm{\a_2} + 3 \bm{\a_3} \rp,\\
\bm{\rho} &= \frac{3}{2} \bm{\a_1} +2 \bm{\a_2} + \frac{3}{2} \bm{\a_3} .
\end{align*}
With respect to the quadratic form $p Q_3(\cdot)$, the conjugacy classes for the lattice $A_3$ are spanned by vectors of the form
\begin{equation*}
\bm{\mu_3} (\bm{s}) := \frac{s_1}{p} \bm{\b_1} + \frac{s_2}{p} \bm{\b_2}+ \frac{s_3}{p} \bm{\b_3}
= \frac{1}{4p} \lp 3s_1+2s_2+s_3, 2s_1+4s_2+2s_3, s_1 + 2s_2+3s_3 \rp,
\end{equation*}
where $\bm{s} \in \IZ^3$ represents the corresponding conjugacy class. Since $2 \bm{\b_1} + \bm{\b_2}, \bm{\b_1} + \bm{\b_3}, 4 \bm{\b_1} \in A_3$, the conjugacy classes $\bm{\mu_3} (\bm{s})$ remain invariant under shifts of $\bm{s}$ as follows
\begin{equation*}
\bm{\mu_3}(\bm{s}) \equiv 
\bm{\mu_3}(s_1+2p,s_2+p,s_3), \ 
\bm{\mu_3}(s_1+p,s_2,s_3+p), \ 
\bm{\mu_3}(s_1+4p,s_2,s_3) 
\pmod{A_3} .
\end{equation*}
So each conjugacy class can be uniquely represented as $\bm{\mu_3}(\bm{s})$ by restricting $\bm{s}$ to
$0 \leq s_1 \leq 4p-1$ and $0 \leq s_2, s_3 \leq p-1.$
Finally, in preparation for Lemma \ref{lem:EQMComplexRecursion}, 
we state how $A_3 + \bm{\mu_3} (\bm{s})$ can be decomposed into components in the linear span of a $\bm{\b_m}$ (for $m \in \{1,2,3\}$) and its orthogonal complement.

\begin{lem}
\label{lem:A_3_decomposition} 
\mbox{ } \nolisttopbreak
\begin{enumerate}[label={\rm(\arabic*)},wide,labelindent=0pt,leftmargin=*]
\item Decomposing $A_3 + \bm{\mu_3} (\bm{s})$ to parts in the linear span of $\bm{\b_1}$ and its orthogonal complement (spanned by $\bm{\a_2}$ and $\bm{\a_3}$) gives 
\begin{align*}
A_3 + \bm{\mu_3} (\bm{s}) =
\bigcup_{\g \in \{0,1,2\} } \Bigg( 
&4 \lp \IZ  + \frac{3s_1+2s_2+s_3+4\g p}{12p} \rp \bm{\b_1}
\\ & \qquad 
+ \lp \IZ + \frac{2(s_2-\g p) + s_3}{3p} \rp \bm{\a_2} + \lp \IZ + \frac{(s_2-\g p) +2 s_3}{3p} \rp \bm{\a_3}
\Bigg) .
\end{align*}
For fixed $\g$, the subset along $\bm{\a_2}$ and $\bm{\a_3}$ gives the shifted lattice $A_2 + \bm{\mu_2} \lp s_2 - \g p, s_3 \rp$.

\item Decomposing $A_3 + \bm{\mu_3} (\bm{s})$ to parts in the linear span of $\bm{\b_2}$ and its orthogonal complement (spanned by $\bm{\a_1}$ and $\bm{\a_3}$) we find
\begin{align*}
A_3 + \bm{\mu_3} (\bm{s}) =
\bigcup_{\d \in \{0,1\} } \Bigg( 
&2 \lp \IZ  + \frac{s_1+2s_2+s_3+2\d p}{4p} \rp \bm{\b_2}
\\ & \qquad 
+ \lp \IZ + \frac{s_1 - \d p}{2p} \rp \bm{\a_1}
+ \lp \IZ + \frac{s_3 - \d p}{2p} \rp \bm{\a_3}
\Bigg) .
\end{align*}
For fixed $\d$, the subset along $\bm{\a_1}$ gives the shifted lattice $A_1 + \bm{\mu_1} \lp s_1 - \d p \rp$ and the subset along $\bm{\a_3}$ gives an orthogonal shifted lattice $A_1 + \bm{\mu_1} \lp s_3 - \d p \rp$.

\item Decomposing $A_3 + \bm{\mu_3} (\bm{s})$ to parts in the linear span of $\bm{\b_3}$ and its orthogonal complement (spanned by $\bm{\a_1}$ and $\bm{\a_2}$) we find
\begin{align*}
A_3 + \bm{\mu_3} (\bm{s}) =
\bigcup_{\g \in \{0,1,2\} } \Bigg( 
&4 \lp \IZ  + \frac{s_1+2s_2+3s_3+4\g p}{12p} \rp \bm{\b_3}
\\ & \qquad 
+ \lp \IZ + \frac{2s_1 + (s_2-\g p)}{3p} \rp \bm{\a_1}
+ \lp \IZ + \frac{s_1 + 2(s_2-\g p)}{3p} \rp \bm{\a_2} 
\Bigg) .
\end{align*}
For fixed $\g$, the subset along $\bm{\a_1}$ and $\bm{\a_2}$ gives the shifted lattice $A_2 + \bm{\mu_2} \lp s_1, s_2 - \g p\rp$.
\end{enumerate}
\end{lem}
\begin{proof}
\mbox{ } \nolisttopbreak
\begin{enumerate}[label={\rm(\arabic*)},wide,labelindent=0pt]
\item Since $\bm{\b_1} = \frac{1}{4} \lp 3 \bm{\a_1} +2\bm{\a_2} + \bm{\a_3} \rp$, we have $\bm{\a_1} = \frac{4}{3} \bm{\b_1} - \frac{2}{3} \bm{\a_2} - \frac{1}{3}  \bm{\a_3}$ and we can write elements of $A_3 + \bm{\mu_3} (\bm{s})$ in the following form where  $\bm{m} \in \IZ^3$
\begin{align*}
& \lp m_1 + \frac{3s_1+2s_2+s_3}{4p} \rp \bm{\a_1}
+ \lp m_2 + \frac{s_1+2s_2+s_3}{2p} \rp \bm{\a_2} + \lp m_3 + \frac{s_1+2s_2+3s_3}{4p} \rp \bm{\a_3}
\\
& \qquad = 
4 \lp \frac{m_1}{3} + \frac{3s_1+2s_2+s_3}{12p} \rp \bm{\b_1}
+ \lp m_2 -\frac{2m_1}{3} + \frac{2s_2+s_3}{3p} \rp \bm{\a_2} 
+ \lp m_3 -\frac{m_1}{3}  + \frac{s_2+2s_3}{3p} \rp \bm{\a_3} .
\end{align*}
Making the change of variables $m_1 \mapsto 3m_1 + \g$ with $m_1 \in \IZ$ and $\g \in \{0,1,2\}$, and then changing $m_2 \mapsto m_2 + 2m_1$ and $m_3 \mapsto m_3+m_1$ we get
\begin{equation*}
4 \lp m_1 + \frac{3s_1+2s_2+s_3+4\g p}{12p} \rp \bm{\b_1}
+ \lp m_2  + \frac{2(s_2- \g p)+s_3}{3p} \rp \bm{\a_2} 
+ \lp m_3 + \frac{(s_2- \g p)+2s_3}{3p} \rp \bm{\a_3} ,
\end{equation*}
where $\bm{m}$ runs over $\IZ^3$ and $\g$ over integers modulo $3$.

\item Since $\bm{\b_2} = \frac{1}{2} \lp \bm{\a_1} +2\bm{\a_2} + \bm{\a_3} \rp$, we have $\bm{\a_2} =  \bm{\b_2} - \frac{1}{2} \bm{\a_1} - \frac{1}{2}  \bm{\a_3}$ and we can write elements of $A_3 + \bm{\mu_3} (\bm{s})$ in the following form where  $\bm{m} \in \IZ^3$
\begin{align*}
& \lp m_1 + \frac{3s_1+2s_2+s_3}{4p} \rp \bm{\a_1}
+ \lp m_2 + \frac{s_1+2s_2+s_3}{2p} \rp \bm{\a_2} + \lp m_3 + \frac{s_1+2s_2+3s_3}{4p} \rp \bm{\a_3}
\\
& \qquad = 
\lp m_1 -\frac{m_2}{2}+ \frac{s_1}{2p} \rp \bm{\a_1}
+ \lp m_2 + \frac{s_1+2s_2+s_3}{2p} \rp \bm{\b_2} 
+ \lp m_3 -\frac{m_2}{2}+ \frac{s_3}{2p}  \rp \bm{\a_3} .
\end{align*}
Making the change of variables $m_2 \mapsto 2m_2 + \d$ with $m_2 \in \IZ$ and $\d \in \{0,1\}$, and then changing $m_1 \mapsto m_1 + m_2$ and $m_3 \mapsto m_3+m_2$ we get
\begin{equation*}
\lp m_1 + \frac{s_1-\d p}{2p} \rp \bm{\a_1}
+ 2 \lp m_2 + \frac{s_1+2s_2+s_3+2\d p}{4p} \rp \bm{\b_2} 
+ \lp m_3 + \frac{s_3-\d p}{2p}  \rp \bm{\a_3} ,
\end{equation*}
where $\bm{m}$ runs over $\IZ^3$ and $\d$ over integers modulo $2$.

\item Exchanging the labels 1 and 3 in part (1) proves (3). \qedhere
\end{enumerate}
\end{proof}

We now start our study of depth three examples with
\begin{equation}\label{eq:false_completion_A_3}
\wh{\psi}^{[A_3]}_{p, \bm{s}} (\t,w) = \sum_{\bm{n} \in \IZ^3 + \bm{\mu_3} (\bm{s})}
E_{pQ_3,(\bm{\b_1},\bm{\b_2}, \bm{\b_3})} \lp -i \sqrt{i(w-\t)} \bm{n} \rp q^{p Q_3 (\bm{n})}.
\end{equation}

\begin{prop}\label{prop:A_3_false_theta}
\mbox{ }  \nolisttopbreak
\begin{enumerate}[label={\rm(\arabic*)},wide,labelindent=0pt,leftmargin=*]
\item The function $\wh{\psi}^{[A_3]}_{p,\bm{s}}$ has the modular transformation properties
\begin{align*}
\wh{\psi}^{[A_3]}_{p, \bm{s}} (\t+1,w+1) &= e^{ \frac{\pi i}{4p} \lp 3s_1^2 + 4s_2^2+3s_3^2+4s_1s_2+4s_1s_3+2s_1s_3 \rp }
\wh{\psi}^{[A_3]}_{p, \bm{s}} (\t,w) , 
\\
\wh{\psi}^{[A_3]}_{p, \bm{s}} \lp -\frac{1}{\t}, - \frac{1}{w} \rp &=
\chi_{\t,w} \frac{(-i \t)^{\frac{3}{2}}}{\sqrt{4p^3}} \sum_{r_1=0}^{4p-1} \sum_{r_2,r_3=0}^{p-1}
e^{-2 \pi i p B_3 ( \bm{\mu_3} (\bm{s}), \bm{\mu_3} (\bm{r}) )}
\wh{\psi}^{[A_3]}_{p, \bm{r}} (\t,w)   ,
\end{align*}
where
\begin{equation*}
p B_3 ( \bm{\mu_3} (\bm{s}), \bm{\mu_3} (\bm{r}) ) = \frac{1}{4p}
(s_1 \,\, s_2 \,\, s_3)
\mat{3 & 2 & 1 \\ 2 & 4 & 2 \\ 1 & 2 & 3}
\mat{r_1 \\ r_2 \\ r_3}.
\end{equation*}

\item \label{prop_part:A3_psi_boundary} We have 
\begin{align*}
\psi^{[A_3]}_{p, \bm{s}} (\t) =& 
\sum_{\bm{n} \in \IZ^3 + \bm{\mu_3} (\bm{s})}
\sgn (n_1) \sgn (n_2) \sgn (n_3) q^{p Q_3 (\bm{n})}
\\ & \qquad
+ \frac{2}{\pi} \arcsin \lp \frac{1}{\sqrt{3}} \rp  \d_{\frac{3s_1+2s_2+s_3}{4p},\frac{s_1+2s_2+s_3}{2p} \in \IZ}
\sum_{n \in \IZ + \frac{s_1+2s_2+3s_3}{4p}}
\sgn (n) q^{pn^2}
\\ & \qquad \quad 
+ \frac{2}{\pi} \arcsin \lp \frac{1}{\sqrt{3}} \rp   \d_{\frac{s_1+2s_2+s_3}{2p}, \frac{s_1+2s_2+3s_3}{4p} \in \IZ}
\sum_{n \in \IZ + \frac{3s_1+2s_2+s_3}{4p}} 
\sgn (n) q^{pn^2}
\\ &  \qquad \quad \quad
+ \frac{2}{\pi} \arcsin \lp \frac{1}{3} \rp  \d_{\frac{3s_1+2s_2+s_3}{4p},\frac{s_1+2s_2+3s_3}{4p} \in \IZ}
\sum_{n \in \IZ + \frac{s_1+2s_2+s_3}{2p}} 
\sgn (n) q^{pn^2}.
\end{align*}

\item We have the following Eichler integral representation
\begin{align*}
\wh{\psi}^{[A_3]}_{p, \bm{s}} (\t,w) &=
2 \sqrt{p} \int_\t^w  \frac{H_{p,\bm{s}} (\t,w_1)}{\sqrt{i (w_1-\t)}} dw_1,
\end{align*}
where
\begin{multline*}
\hspace{.4cm}H_{p,\bm{s}} (\t,w) = 
\sum_{\d \in \{0,1\}}
\vth^{[1]}_{2p,s_1+2s_2+s_3+2\d p} (w) 
 \wh{\psi}^{[A_1]}_{p,s_1-\d p} (\t, w) \wh{\psi}^{[A_1]}_{p,s_3-\d p} (\t, w)
\\ \hspace{.8cm} +
\sqrt{3} \sum_{\g \in \{0,1,2 \}}  \lp 
\vth^{[1]}_{6p,3s_1+2s_2+s_3+4\g p} (w) \wh{\psi}^{[A_2]}_{p,s_2-\g p, s_3} (\t, w)
+ \vth^{[1]}_{6p,s_1+2s_2+3s_3+4\g p} (w) \wh{\psi}^{[A_2]}_{p,s_1,s_2-\g p} (\t, w)  \rp .
\end{multline*}
\end{enumerate}
In summary, $\psi^{[A_3]}_{p, \bm{s}}$ forms a vector-valued false modular form of depth three and weight 
$\frac{3}{2}$.
\end{prop}
\begin{proof}
\mbox{ }  \nolisttopbreak
\begin{enumerate}[label={\rm(\arabic*)},wide, labelindent=0pt]
\item The modular transformations simply follow by specializing the general result in Theorem \ref{co:modularity} and using the relevant inner products of $\bm{\b_j}$'s (which can be read from their coordinates in $\bm{\a}$-basis)
\begin{equation}\label{eq:A_3_weight_inner_products}
\lp B_3(\bm{\b_j},\bm{\b_k}) \rp_{j,k\in\{1,3\}} = \frac 14 \mat{3 & 2 & 1 \\ 2 & 4 & 2 \\ 1 & 2 & 3 } .
\end{equation}

\item Specializing Lemma \ref{la:generalized_error_asymptotic_form} to three dimensions and using Remark \ref{rem:E2_zero_value}, we find (with $\bm{m_{j+3}}:=\bm{m_j}$)
\begin{multline*}
\hspace{.4cm}\lim_{t \to \infty} E_{Q,M} (t \bm{x}) 
= \sgn (B(\bm{m_1},\bm{x})) \sgn (B(\bm{m_2},\bm{x})) \sgn (B(\bm{m_3},\bm{x}))
\\ 
+\frac{2}{\pi} \sum_{j=1}^3 
\d_{B(\bm{m_j},\bm{x}) = B(\bm{m_{j+1}},\bm{x})=0} \, \sgn (B(\bm{m_{j+2}},\bm{x}))
\arcsin (B(\bm{\wh{m}_j},\bm{\wh{m}_{j+1}})).
\end{multline*}
According to equation \eqref{eq:A_3_weight_inner_products}, we have $B(\bm{\wh{m}_1},\bm{\wh{m}_2})  =B(\bm{\wh{m}_2},\bm{\wh{m}_3}) =\frac{1}{\sqrt{3}}$ and $B(\bm{\wh{m}_1},\bm{\wh{m}_3}) =\frac{1}{3}$.
We find the claim by using these values to take the limit $\lim_{w \to \t + i \infty} \wh{\psi}^{[A_3]}_{p, \bm{s}} (\t,w)$ in equation \eqref{eq:false_completion_A_3}.

\item Specializing the recursion relation in Lemma \ref{lem:EQMComplexRecursion} to our case, we find
\begin{align}
\label{eq:triple_error_decomposition}
&\hspace{.7cm}E_{p Q_3,(\bm{\b_1},\bm{\b_2},\bm{\b_3})} \left(-i \sqrt{i (w-\t)} \bm{x}\right) 
\\ & \qquad
=  \frac{2 \sqrt{p}}{\sqrt{3}}  \int_w^\t  
B_3(\bm{\b_1},\bm{x}) e^{\frac{4}{3} \pi i p B_3(\bm{\b_1},\bm{x})^2 (w_1-\t)}
E_{pQ_3,\lp \frac{2 \bm{\a_2} + \bm{\a_3} }{3},\frac{\bm{\a_2} + 2\bm{\a_3} }{3} \rp } \left(-i \sqrt{i (w_1-\t)} \bm{x}\right) 
\frac{d w_1}{\sqrt{i (w_1-\t)}}
\notag \\ & \qquad
+\sqrt{p} \int_w^\t  B_3(\bm{\b_2},\bm{x}) e^{\pi i p B_3(\bm{\b_2},\bm{x})^2 (w_1-\t)}
E_{pQ_3,\lp  \frac{ \bm{\a_1} }{2}, \frac{ \bm{\a_3} }{2} \rp} \left(-i \sqrt{i (w_1-\t)} \bm{x}\right) 
\frac{d w_1}{\sqrt{i (w_1-\t)}}
\notag \\ & \qquad
+ \frac{2 \sqrt{p}}{\sqrt{3}}  \int_w^\t  
B_3(\bm{\b_3},\bm{x}) e^{\frac{4}{3} \pi i p B_3(\bm{\b_3},\bm{x})^2 (w_1-\t)}
E_{pQ_3,\lp \frac{2 \bm{\a_1} + \bm{\a_2} }{3},\frac{\bm{\a_1} + 2\bm{\a_2} }{3} \rp } \left(-i \sqrt{i (w_1-\t)} \bm{x}\right) 
\frac{d w_1}{\sqrt{i (w_1-\t)}} .
\notag
\end{align}
Inserting this expression in \eqref{eq:false_completion_A_3} and summing over $\IZ^3 + \bm{\mu_3} (\bm{s})$ using parts (1), (2), and (3) of Lemma \ref{lem:A_3_decomposition} on the second, third, and fourth lines of \eqref{eq:triple_error_decomposition}, respectively, gives the stated result. \qedhere
\end{enumerate}
\end{proof}

Next, we study the properties of $\varphi^{[A_3]}_{p, \bm{s}}$. Define 
\begin{align*}
	&F_{p,\bm{s}} (\t,w) := 
	\ \sum_{\mathclap{\d \in \{0,1\}}} \vth^{[1]}_{2p,s_1+2s_2+s_3+2\d p} (w) 
	\lp  \wh{\psi}^{[A_1]}_{p,s_1-\d p} (\t, w) \wh{\varphi}^{[A_1]}_{p,s_3-\d p} (\t, w)
	+  \wh{\varphi}^{[A_1]}_{p,s_1-\d p} (\t, w) \wh{\psi}^{[A_1]}_{p,s_3-\d p} (\t, w) \rp
	\\ &  \quad +
	\sqrt{3} \sum_{\g \in \{0,1,2 \}}  \lp 
	\vth^{[1]}_{6p,3s_1+2s_2+s_3+4\g p} (w) \wh{\varphi}^{[A_2]}_{p,s_2-\g p, s_3} (\t, w)
	+ \vth^{[1]}_{6p,s_1+2s_2+3s_3+4\g p} (w) \wh{\varphi}^{[A_2]}_{p,s_1,s_2-\g p} (\t, w)  \rp,\\
	&G_{p,\bm{s}} (\t,w) := 
	2 \sum_{\d \in \{0,1\}} \vth_{2p,s_1+2s_2+s_3+2\d p} (w) 
	\wh{\psi}^{[A_1]}_{p,s_1-\d p} (\t, w) \wh{\psi}^{[A_1]}_{p,s_3-\d p} (\t, w)
	\\ &  \quad +
	\sqrt{3} \sum_{\g \in \{0,1,2 \}}  \lp 
	\vth_{6p,3s_1+2s_2+s_3+4\g p} (w) \wh{\psi}^{[A_2]}_{p,s_2-\g p, s_3} (\t, w)
	+ \vth_{6p,s_1+2s_2+3s_3+4\g p} (w) \wh{\psi}^{[A_2]}_{p,s_1,s_2-\g p} (\t, w)  \rp  .
\end{align*}
\begin{prop}\label{prop:PhiA3}
\mbox{ }  \nolisttopbreak
\begin{enumerate}[label={\rm(\arabic*)},wide,labelindent=0pt,leftmargin=*]
\item\label{PhiA3} The function $\wh{\varphi}_{p,\bm{s}}^{[A_3]}$ can be explicitly expressed as
\begin{align*}%\label{eq:A_3_varphi_series}
\hspace{.75cm}\wh{\varphi}^{[A_3]}_{p, \bm{s}} (\t,w) &= \sum_{\bm{n} \in \IZ^3 + \bm{\mu_3} (\bm{s})}
E_{pQ_3,(\bm{\b_1},\bm{\b_2}, \bm{\b_3})} \lp -i \sqrt{i(w-\t)} \bm{n} \rp (n_1+n_2+n_3) q^{p Q_3 (\bm{n})}\\
& \qquad +\frac{i}{\pi} \frac{1}{\sqrt{ip(w-\t)}} 
\left(\sqrt{3} \sum_{\g \in \{0,1,2\} }\vth_{6p, 3s_1+2s_2+s_3+4p\g} (w) \wh{\psi}^{[A_2]}_{p,s_2-p\g,s_3} (\t, w) \right.\\
&\hspace{4cm}
\left.+2 \sum_{\d \in \{0,1\} }\vth_{2p, s_1+2s_2+s_3+2p\d} (w) \wh{\psi}^{[A_1]}_{p,s_1-p\d} (\t, w)
\wh{\psi}^{[A_1]}_{p,s_3-p\d} (\t, w) \right.\\[-2.5ex]
&\hspace{4cm}
\left.+\sqrt{3} \sum_{\g \in \{0,1,2\} }\vth_{6p, s_1+2s_2+3s_3+4p\g} (w) \wh{\psi}^{[A_2]}_{p,s_1,s_2-p\g} (\t, w) \right).
\end{align*}

\item The function $\wh{\varphi}^{[A_3]}_{p,\bm{s}}$ has the modular transformation properties
\begin{align*}
\wh{\varphi}^{[A_3]}_{p, \bm{s}} (\t+1,w+1) &= e^{2\pi i p Q_3 (\bm{\mu_3} (\bm{s})) }
\wh{\varphi}^{[A_3]}_{p, \bm{s}} (\t,w),  
\\
\wh{\varphi}^{[A_3]}_{p, \bm{s}} \lp -\frac{1}{\t}, - \frac{1}{w} \rp &=
\chi_{\t,w} \frac{e^{-\frac{3\pi i}{4}} \t^{\frac{5}{2}}}{\sqrt{4p^3}} \sum_{r_1=0}^{4p-1} \sum_{r_2,r_3=0}^{p-1}
e^{-2 \pi i p B_3 ( \bm{\mu_3} (\bm{s}), \bm{\mu_3} (\bm{r}) )}
\wh{\varphi}^{[A_3]}_{p, \bm{r}} (\t,w)   .
\end{align*}

\item\label{prop_part:A3_phi_boundary} We have 
\begin{align*}
\varphi^{[A_3]}_{p, \bm{s}} (\t) =& 
\sum_{\bm{n} \in \IZ^3 + \bm{\mu_3} (\bm{s})} 
\sgn (n_1) \sgn (n_2) \sgn (n_3) (n_1+n_2+n_3) q^{p Q_3 (\bm{n})}
\\ & \qquad
+ \frac{2}{\pi} \arcsin \lp \frac{1}{\sqrt{3}} \rp  \d_{\frac{3s_1+2s_2+s_3}{4p},\frac{s_1+2s_2+s_3}{2p} \in \IZ}
\sum_{n \in \IZ + \frac{s_1+2s_2+3s_3}{4p}} 
\sgn (n) n q^{pn^2}
\\ & \qquad \quad 
+ \frac{2}{\pi} \arcsin \lp \frac{1}{\sqrt{3}} \rp   \d_{\frac{s_1+2s_2+s_3}{2p}, \frac{s_1+2s_2+3s_3}{4p} \in \IZ}
\sum_{n \in \IZ + \frac{3s_1+2s_2+s_3}{4p}} 
\sgn (n) n q^{pn^2}
\\ &  \qquad \quad \quad
+ \frac{2}{\pi} \arcsin \lp \frac{1}{3} \rp  \d_{\frac{3s_1+2s_2+s_3}{4p},\frac{s_1+2s_2+3s_3}{4p} \in \IZ}
\sum_{n \in \IZ + \frac{s_1+2s_2+s_3}{2p}} 
\sgn (n) n q^{pn^2}.
\end{align*}

\item \label{PhiA3_Eichler} We have the following representation in terms of Eichler integrals:
\begin{align*}
\wh{\varphi}_{p,\bm{s}}^{[A_3]} (\t, w) =& 
 2 \sqrt{p}\int_{\t}^{w} \frac{F_{p,\bm{s}}(\t,w_1)}{\sqrt{i(w_1-\t)}} dw_1
 +\frac{1}{2\pi \sqrt{p}} \rint{\t}{w} \frac{G_{p,\bm{s}}(\t,w_1)}{(i(w_1-\t))^\frac32} dw_1,
\end{align*}
where we use the regularized integral defined in equation \eqref{eq:regularized_integral}.
\end{enumerate}
In summary, $\varphi^{[A_3]}_{p, \bm{s}}$ forms a vector-valued false modular form of depth three and weight 
$\frac{5}{2}$.
\end{prop}
\begin{proof}
\mbox{ }  \nolisttopbreak

\begin{enumerate}[label={\rm(\arabic*)},wide,labelindent=0pt]
\item We first use Lemma \ref{lem:generalized_error_derivative} with $\bm{\rho} = \frac{3}{2} \bm{\a_1} + 2 \bm{\a_2} + \frac{3}{2} \bm{\a_3}$ to get
\begin{align*}
&\frac{1}{2\pi i p}\left[\bm{\rho}^T\bm{\partial_{z}} \lp E_{pQ_{3}, (\bm{\b_1},\bm{\b_2},\bm{\b_3})} \left(-i\left(\bm{n}+\frac{\bm{\zz}-\bm{z}}{w-\t}\right)\sqrt{i(w-\t)}\right)\rp \right]_{\bm{z}=\bm{\zz}=\bm{0}}
\\ & \qquad
= \frac{i}{\pi} \frac{1}{\sqrt{ip(w-\t)}} 
\Bigg( 
\sqrt{3} e^{\frac{4\pi i p}{3}   B_{3}(\bm{\b_1},\bm{n})^2 (w-\t)}
E_{pQ_{3}, \lp  \frac{2 \bm{\a_2} + \bm{\a_3} }{3},  \frac{\bm{\a_2} +2 \bm{\a_3} }{3} \rp} \left(-i \bm{n} \sqrt{i(w-\t)}  \right) 
\\ & \qquad \qquad  \qquad \qquad \qquad
+2 e^{\pi i p  B_{3}(\bm{\b_2},\bm{n})^2 (w-\t)}
E_{pQ_{3}, \lp   \frac{\bm{\a_1}}{2}, \frac{\bm{\a_3}}{2}  \rp} \left(-i \bm{n} \sqrt{i(w-\t)}  \right) 
\\ & \qquad \qquad  \qquad  \qquad \qquad \qquad
+\sqrt{3} e^{\frac{4\pi i p}{3}   B_{3}(\bm{\b_3},\bm{n})^2 (w-\t)}
E_{pQ_{3}, \lp  \frac{2 \bm{\a_1} + \bm{\a_2} }{3},  \frac{\bm{\a_1} +2 \bm{\a_2} }{3} \rp} \left(-i \bm{n} \sqrt{i(w-\t)}  \right) 
\Bigg).
\end{align*}
Multiplying with $q^{p Q_3 (\bm{n})}$ and summing over $\bm{n}$ using Lemma \ref{lem:A_3_decomposition} gives the second, third, and fourth lines of the claimed identity. The contribution on the first line is obtained using the relation $\frac{1}{2\pi i p} [ \bm{\rho}^T\bm{\partial_{z}}\lp e^{2\pi i pB_{3}(\bm{n},\bm{z})}\rp]_{\bm{z} = \bm{0}}=B_{3}(\bm{n},\bm{\rho})$.

\item The modular transformations follow from Theorem \ref{co:modularity} and Lemma \ref{la:derivativeLimit}. 

\item Taking $w\to \t+i\infty$ in part \ref{PhiA2} using Lemma \ref{la:generalized_error_asymptotic_form} 
(as in Proposition \ref{prop:A_3_false_theta} \ref{prop_part:A3_psi_boundary})
gives the claim.

\item We insert equation \eqref{eq:triple_error_decomposition} into the contribution on the first line of part \ref{PhiA3}.
We start with the contribution of the first term on the right-hand side of  \eqref{eq:triple_error_decomposition}
\begin{align*}
I_{\bm{s}} &:= 
\frac{2 \sqrt{p}}{\sqrt{3}}  \int_w^\t   \sum_{\bm{n} \in \IZ^3 + \bm{\mu_3} (\bm{s})} B_3(\bm{\b_1},\bm{n})  B_3(\bm{\rho},\bm{n}) 
e^{\frac{4}{3} \pi i p B_3(\bm{\b_1},\bm{n})^2 (w_1-\t)}
\\ & \qquad \qquad \qquad 
\times E_{pQ_3,\lp \frac{2 \bm{\a_2} + \bm{\a_3} }{3},\frac{\bm{\a_2} + 2\bm{\a_3} }{3} \rp } \left(-i \bm{n} \sqrt{i (w_1-\t)} \right)  q^{p Q_3 (\bm{n})}
\frac{d w_1}{\sqrt{i (w_1-\t)}}.
\end{align*}
Noting that $\bm{\rho} = 2 \bm{\b_1} + \bm{\a_2}+\bm{\a_3}$ and taking the sum over $\IZ^3 + \bm{\mu_3} (\bm{s})$ using the first part of Lemma \ref{lem:A_3_decomposition} and Proposition \ref{prop:PhiA2} (1) we get
\begin{align*}
I_{\bm{s}} &= \frac{\sqrt{3}}{\pi i \sqrt{p}} \sum_{\g \in \{0,1,2\} } \int_w^\t \frac{1}{\sqrt{i (w_1-\t)}}
 \vth'_{6p,3s_1+2s_2+s_3+4\g p} (w_1)  \wh{\psi}^{[A_2]}_{p,s_2 - \g p, s_3} (\t, w_1) d w_1
\\ &
+ 2 \sqrt{3p}\sum_{\g \in \{0,1,2\} } \int_w^\t \frac{1}{\sqrt{i (w_1-\t)}} \vth^{[1]}_{6p,3s_1+2s_2+s_3+4\g p} (w_1) \bigg(
\wh{\varphi}^{[A_2]}_{p,s_2 - \g p, s_3} (\t, w_1) - \frac{i}{\pi} \frac{\sqrt{3}}{\sqrt{2ip(w_1-\t)}}
\\ & 
\times \  \sum_{\mathclap{\d \in \{0,1\} } } \lp
\vth_{3p,s_2+2s_3-\g p+3\d p} (w_1) \wh{\psi}^{[A_1]}_{p,s_2-\g p-\d p} (\t, w_1)
+ \vth_{3p,2s_2+s_3-2\g p+3\d p} (w_1) \wh{\psi}^{[A_1]}_{p,s_3-\d p} (\t, w_1) 
\rp \bigg) d w_1 .
\end{align*}
We use Proposition \ref{prop:psiA2_properties}
to integrate by parts and get\footnote{Here we also note Remarks \ref{rem:A2_varphi_Eichler_convergence} and \ref{rem:A2_limiting_value} in justifying the convergence of these integrals.}
\begin{align*}
&\int_\t^w \frac{1}{\sqrt{i (w_1-\t)}}
\frac{\del}{\del w_1} \lp \vth_{6p,3s_1+2s_2+s_3+4\g p} (w_1) \rp \wh{\psi}^{[A_2]}_{p,s_2 - \g p, s_3} (\t, w_1) d w_1
\\
&\quad =
\frac{\vth_{6p,3s_1+2s_2+s_3+4\g p} (w) \wh{\psi}^{[A_2]}_{p,s_2 - \g p, s_3} (\t, w)}{\sqrt{i (w-\t)}}
+
\frac{i}{2} \rint{\t}{w} \frac{ \vth_{6p,3s_1+2s_2+s_3+4\g p} (w_1) \wh{\psi}^{[A_2]}_{p,s_2 - \g p, s_3} (\t, w_1)}{(i(w_1-\t))^{\frac{3}{2}}} dw_1 
\\
& \qquad \quad 
+i \sqrt{6p}  \int_\t^w \frac{\vth_{6p,3s_1+2s_2+s_3+4\g p} (w_1) h_{p,s_2-\g p,s_3} (\t,w_1)}{ w_1-\t } d w_1 .
\end{align*}
Plugging this expression back (and noting the definition of $h_{p,\bm{s}}$ from Proposition \ref{prop:psiA2_properties})  we find $I_{\bm{s}} = I^{[1]}_{\bm{s}} + I^{[2]}_{\bm{s}} + I^{[3]}_{\bm{s}} + I^{[4]}_{\bm{s}}$ where
\begin{align*}
&I^{[1]}_{\bm{s}} := 
\frac{\sqrt{3}}{2\pi \sqrt{p}} \sum_{\g \in \{0,1,2\} } \rint{\t}{w}  \frac{ \vth_{6p,3s_1+2s_2+s_3+4\g p} (w_1) \wh{\psi}^{[A_2]}_{p,s_2 - \g p, s_3} (\t, w_1)}{\left(i(w_1-\t)\right)^{\frac{3}{2}}} dw_1 
\\ & \qquad \qquad \qquad \qquad \qquad \qquad
+
2 \sqrt{3p}\sum_{\g \in \{0,1,2\} } \int_\t^w \frac{\vth^{[1]}_{6p,3s_1+2s_2+s_3+4\g p} (w_1)  \wh{\varphi}^{[A_2]}_{p,s_2 - \g p, s_3} (\t, w_1) }{\sqrt{i (w_1-\t)}}  d w_1 ,\\
&I^{[2]}_{\bm{s}} := - \frac{i \sqrt{3}}{\pi \sqrt{p}} \sum_{\g \in \{0,1,2\} }
\frac{\vth_{6p,3s_1+2s_2+s_3+4\g p} (w) \wh{\psi}^{[A_2]}_{p,s_2 - \g p, s_3} (\t, w)}{\sqrt{i (w-\t)}},\\
&I^{[3]}_{\bm{s}} :=  \frac{3 \sqrt{2}}{\pi} \sum_{\substack{ \g \in \{0,1,2\} \\ \d \in \{0,1\}  }}
\int_\t^w
\frac{ \wh{\psi}^{[A_1]}_{p,s_3-\d p}(\t, w_1) }{w_1 - \t} 
\Big(\vth_{6p,3s_1+2s_2+s_3+4\g p} (w_1) \vth^{[1]}_{3p,2s_2+s_3-2\g p+3\d p} (w_1)
\\
&\hspace{7.4cm}
- \vth^{[1]}_{6p,3s_1+2s_2+s_3+4\g p} (w_1) \vth_{3p,2s_2+s_3-2\g p+3\d p} (w_1) \Big) d w_1,
\\
&I^{[4]}_{\bm{s}} :=  \frac{3 \sqrt{2}}{\pi} \sum_{\substack{ \g \in \{0,1,2\} \\ \d \in \{0,1\}  }}
\int_\t^w
\frac{ \wh{\psi}^{[A_1]}_{p,s_2-\g p-\d p}(\t, w_1) }{w_1 - \t} \Big(
\vth_{6p,3s_1+2s_2+s_3+4\g p} (w_1) \vth^{[1]}_{3p,s_2+2s_3-\g p+3\d p} (w_1)
\\
&\hspace{7.4cm}
- \vth^{[1]}_{6p,3s_1+2s_2+s_3+4\g p} (w_1) \vth_{3p,s_2+2s_3-\g p+3\d p} (w_1)  \Big) d w_1 .
\end{align*}
Next we study the contribution of the second term on the right-hand side of \eqref{eq:triple_error_decomposition}
\begin{align*}
J_{\bm{s}} &:= 
\sqrt{p}  \int_\t^w   \sum_{\bm{n} \in \IZ^3 + \bm{\mu_3} (\bm{s})}
B_3(\bm{\b_2},\bm{n})  B_3(\bm{\rho},\bm{n}) 
e^{\pi i p B_3(\bm{\b_2},\bm{n})^2 (w_1-\t)}
\\ & \qquad \qquad \qquad   \qquad \qquad 
\times E_{pQ_3,\lp \frac{\bm{\a_1}}{2} , \frac{\bm{\a_3}}{2} \rp } \left(-i \sqrt{i (w_1-\t)} \bm{n}\right)  q^{p Q_3 (\bm{n})}
\frac{d w_1}{\sqrt{i (w_1-\t)}}.
\end{align*}
Noting that $\bm{\rho} = 2 \bm{\b_2} + \frac{\bm{\a_1}}{2}+\frac{\bm{\a_3}}{2}$ and taking the sum over $\IZ^3 + \bm{\mu_3} (\bm{s})$ using Lemma \ref{lem:A_3_decomposition} (2) and Proposition \ref{prop:PhiA1} (1) we get
\begin{align*}
J_{\bm{s}} &=\frac{2}{\pi i \sqrt{p}} \sum_{\d \in \{0,1\} } \int_\t^w
 \vth'_{2p,s_1+2s_2+s_3+2\d p} (w_1)  
\wh{\psi}^{[A_1]}_{p,s_1 - \d p} (\t, w_1) \wh{\psi}^{[A_1]}_{p,s_3 - \d p} (\t, w_1)
\frac{d w_1}{\sqrt{i (w_1-\t)}}
\\ &\ 
+ 2 \sqrt{p}\sum_{\d \in \{0,1\} } \int_\t^w 
\vth^{[1]}_{2p,s_1+2s_2+s_3+2\d p} (w_1) \Bigg(
\wh{\psi}^{[A_1]}_{p,s_1 - \d p} (\t, w_1)
\lp \wh{\varphi}^{[A_1]}_{p,s_3 - \d p}(\t,w_1) -\frac{i}{\pi \sqrt{2p}} \frac{\vth_{p,s_3-\d p} (w_1)}{\sqrt{i (w_1-\t)}}  \rp
\\ &  \qquad \qquad  \qquad \qquad  \qquad 
+\wh{\psi}^{[A_1]}_{p,s_3 - \d p} (\t, w_1)
\lp \wh{\varphi}^{[A_1]}_{p,s_1 - \d p}(\t,w_1) -\frac{i}{\pi \sqrt{2p}} \frac{\vth_{p,s_1-\d p} (w_1)}{\sqrt{i (w_1-\t)}}  \rp
\Bigg)
\frac{d w_1}{\sqrt{i (w_1-\t)}}.
\end{align*}
Next we integrate by parts using Proposition \ref{prop:psiA1_properties} \ref{Psi1AsIntegral} to find
\begin{align*}
&\int_\t^w  \vth'_{2p,s_1+2s_2+s_3+2\d p} (w_1)  
\wh{\psi}^{[A_1]}_{p,s_1 - \d p} (\t, w_1) \wh{\psi}^{[A_1]}_{p,s_3 - \d p} (\t, w_1)
\frac{d w_1}{\sqrt{i (w_1-\t)}}
\\
& =
\frac{ \vth_{2p,s_1+2s_2+s_3+2\d p} (w) 
\wh{\psi}^{[A_1]}_{p,s_1 - \d p} (\t, w) \wh{\psi}^{[A_1]}_{p,s_3 - \d p} (\t, w)  }{\sqrt{i (w-\t)}}
\\
& \ \  
+
\frac{i}{2} \int_{\t}^{w} \frac{ \vth_{2p,s_1+2s_2+s_3+2\d p} (w_1) 
\wh{\psi}^{[A_1]}_{p,s_1 - \d p} (\t, w_1) \wh{\psi}^{[A_1]}_{p,s_3 - \d p} (\t, w_1)  }{\left(i(w_1-\t)\right)^{\frac{3}{2}}} dw_1 
\\
& \ \ 
+i \sqrt{2p} \int_\t^w \frac{\vth_{2p,s_1+2s_2+s_3+2\d p} (w_1) }{ w_1-\t }
\lp
\vth^{[1]}_{p,s_1 - \d p} (w_1) \wh{\psi}^{[A_1]}_{p,s_3 - \d p} (\t, w_1) 
+  \wh{\psi}^{[A_1]}_{p,s_1 - \d p} (\t, w_1) \vth^{[1]}_{p,s_3 - \d p} (w_1)
\rp  d w_1 .
\end{align*}
Plugging this expression back we find $J_{\bm{s}} = J^{[1]}_{\bm{s}} + J^{[2]}_{\bm{s}} + J^{[3]}_{\bm{s}} + J^{[3]}_{s_3,s_2,s_1}$ where
\begin{align*}
J^{[1]}_{\bm{s}} &:=  
\frac{1}{\pi \sqrt{p}} \sum_{\d \in \{0,1\} }
\int_{\t}^{w} \frac{ \vth_{2p,s_1+2s_2+s_3+2\d p} (w_1) 
\wh{\psi}^{[A_1]}_{p,s_1 - \d p} (\t, w_1) \wh{\psi}^{[A_1]}_{p,s_3 - \d p} (\t, w_1)  }{\left(i(w_1-\t)\right)^{\frac{3}{2}}} dw_1
\\
& \quad+2 \sqrt{p}\sum_{\d \in \{0,1\} } \int_\t^w 
\frac{\vth^{[1]}_{2p,s_1+2s_2+s_3+2\d p} (w_1)}{\sqrt{i (w_1-\t)}} \\
&\hspace{3cm}\times
\lp
\wh{\psi}^{[A_1]}_{p,s_1 - \d p} (\t, w_1)
\wh{\varphi}^{[A_1]}_{p,s_3 - \d p} (\t,w_1)
+\wh{\psi}^{[A_1]}_{p,s_3 - \d p} (\t, w_1)
\wh{\varphi}^{[A_1]}_{p,s_1 - \d p} (\t,w_1)
\rp d w_1,
\\
J^{[2]}_{\bm{s}} &:= - \frac{2i}{\pi \sqrt{p}} \sum_{\d \in \{0,1\} }
\frac{ \vth_{2p,s_1+2s_2+s_3+2\d p} (w) 
\wh{\psi}^{[A_1]}_{p,s_1 - \d p} (\t, w) \wh{\psi}^{[A_1]}_{p,s_3 - \d p} (\t, w)  }{\sqrt{i (w-\t)}},
\\
J^{[3]}_{\bm{s}} &:=  \frac{ \sqrt{2}}{\pi} \sum_{\d \in \{0,1\}} \int_\t^w
\frac{ \wh{\psi}^{[A_1]}_{p,s_3-\d p}(\t, w_1) }{w_1 - \t} 
\\
&\hspace{1.7cm}\times\left(2 \vth_{2p,s_1+2s_2+s_3+2\d p} (w_1)  \vth^{[1]}_{p,s_1-\d p} (w_1)- \vth^{[1]}_{2p,s_1+2s_2+s_3+2\d p} (w_1)  \vth_{p,s_1-\d p} (w_1)  \right) d w_1.
\end{align*}
Finally exchanging the labels 1 and 3, we find that the contribution of the third term on the right-hand side of  \eqref{eq:triple_error_decomposition} is simply $I_{s_3,s_2,s_1}$.

In summary, our discussion above shows that the contribution on the first line of part \ref{PhiA3} can be replaced with 
$I_{\bm{s}} + J_{\bm{s}} + I_{s_3,s_2,s_1}$. 
The sum $I^{[2]}_{\bm{s}} + J^{[2]}_{\bm{s}} + I^{[2]}_{s_3,s_2,s_1}$ simply cancels the contributions on the second, third, and fourth lines of (1), whereas $I^{[1]}_{\bm{s}} + J^{[1]}_{\bm{s}} + I^{[1]}_{s_3,s_2,s_1}$ gives the expression in the claim. 
Therefore, the proof is complete if we demonstrate that the remaining contributions vanish. For this purpose, we next show that $I^{[4]}_{\bm{s}} + I^{[4]}_{s_3,s_2,s_1} = 0$ and $I^{[3]}_{\bm{s}}  + J^{[3]}_{\bm{s}}  = 0$ (with the latter also implying $I^{[3]}_{s_3,s_2,s_1}  + J^{[3]}_{s_3,s_2,s_1}  = 0$ by exchanging the labels $1$ and $3$).

We start with $I^{[4]}_{\bm{s}} + I^{[4]}_{s_3,s_2,s_1}$. First we rewrite $I^{[4]}_{\bm{s}}$, shifting $\d \mapsto \d +\g$,
\begin{align*}
I^{[4]}_{\bm{s}} =  \frac{3 \sqrt{2}}{\pi} \sum_{\d \pmod{2} }
\int_\t^w
\frac{ \wh{\psi}^{[A_1]}_{p,s_2-\d p}(\t, w_1) }{w_1 - \t} 
p_{\bm{s},\d} (w_1) dw_1,
\end{align*}
where
\begin{align*}
p_{\bm{s},\d} (\t) := 
\sum_{\g \pmod{3} }
\bigg(  
&\vth_{6p,3s_1+2s_2+s_3+4\g p} (\t) \vth^{[1]}_{3p,s_2+2s_3+2\g p+3\d p} (\t)
\\ & \qquad \qquad
- \vth^{[1]}_{6p,3s_1+2s_2+s_3+4\g p} (\t) \vth_{3p,s_2+2s_3+2\g p+3\d p} (\t)  \bigg).
\end{align*}
Our goal is then to show that $p_{\bm{s},\d} (\t) + p_{s_3,s_2,s_1,\d} (\t)$ vanishes for all integers $s_1,s_2,s_3,\d$.
Shifting $\g \mapsto \g +\d$ for $p_{s_3,s_2,s_1,\d} (\t)$, we rewrite $p_{\bm{s},\d} (\t) + p_{s_3,s_2,s_1,\d} (\t)$ as
\begin{align*}
&\sum_{\g \pmod{3} }
\bigg(  
\vth_{6p,3s_1+2s_2+s_3+4\g p} (\t) \vth^{[1]}_{3p,s_2+2s_3+2\g p+3\d p} (\t)
- \vth^{[1]}_{6p,3s_1+2s_2+s_3+4\g p} (\t) \vth_{3p,s_2+2s_3+2\g p+3\d p} (\t)  
\\
& \ 
+\vth_{6p,s_1+2s_2+3s_3+4\g p+4 \d p} (\t) \vth^{[1]}_{3p,2s_1+s_2+2\g p-\d p} (\t)
- \vth^{[1]}_{6p,s_1+2s_2+3s_3+4\g p+4 \d p} (\t) \vth_{3p,2s_1+s_2+2\g p-\d p} (\t)   \bigg).
\end{align*}
Now, let $r_1 := 3s_1+2s_2+s_3$ and $r_2 := s_2+2s_3+3\d p$ so that $r_1+r_2 \equiv 0 \pmod{3}$ and
\begin{equation*}
\frac{r_1+4r_2}{3} = s_1+2s_2+3s_3+4\d p
\qquad \mbox{and} \qquad
\frac{2r_1-r_2}{3} = 2s_1+s_2- \d p .
\end{equation*}
Then $p_{\bm{s},\d} (\t) + p_{s_3,s_2,s_1,\d} (\t)=P_{r_1,r_2} (\t) + P_{\frac{r_1+4r_2}{3},\frac{2r_1-r_2}{3}} (\t)$, where for any two integers $r_1,r_2$ with $r_1+r_2 \equiv 0 \pmod{3}$ we define
\begin{align*}
P_{r_1,r_2} (\t) &:= 
\sum_{\g \pmod{3} }
\lp
\vth_{6p,r_1+4\g p} (\t) \vth^{[1]}_{3p,r_2+2\g p} (\t) 
- \vth^{[1]}_{6p,r_1+4\g p} (\t) \vth_{3p,r_2+2\g p} (\t)  \rp 
\\
&=
\sum_{\g \pmod{3} } \sum_{\bm{n} \in \Z^2 + \lp \frac{r_1+4\g p}{12p},\frac{r_2+2 \g p}{6p} \rp}
(n_2 - n_1) q^{3p\left(2n_1^2+n_2^2\right) } .
\end{align*}
Shifting $n_1\mapsto n_1+n_2$ and collapsing the sum on $\gamma$ gives 
\begin{align*}
P_{r_1,r_2} (\t) =
- \sum_{\bm{n} \in \Z^2 + \lp \frac{r_1-2r_2}{12p},\frac{r_2}{2p} \rp} n_1  q^{p\left(6n_1^2+4n_1n_2+n_2^2\right) } .
\end{align*}
We now apply the involution $\bm{n} \mapsto (-n_1,n_2+4n_1)$ to obtain $P_{r_1,r_2} (\t)=
- P_{\frac{r_1+4r_2}{3},\frac{2r_1-r2}{3}} (\t)$, which implies $I^{[4]}_{\bm{s}} + I^{[4]}_{s_3,s_2,s_1}=0$.

Finally, we consider the contribution
\begin{equation*}
I^{[3]}_{\bm{s}}  + J^{[3]}_{\bm{s}} = 
\frac{ \sqrt{2}}{\pi} \sum_{\d \in \{0,1\}} \int_\t^w \frac{ \wh{\psi}^{[A_1]}_{p,s_3-\d p}(\t, w_1) }{w_1 - \t} 
\lp t_{1,\bm{s},\d} (w_1) + t_{2,\bm{s},\d} (w_1) \rp  d w_1,
\end{equation*}
where
\begin{align*}
t_{1,\bm{s},\d} (\t) &:=  3 \sum_{\g \in \{0,1,2\}}
\bigg(  \vth_{6p,3s_1+2s_2+s_3+4\g p} (\t) \vth^{[1]}_{3p,2s_2+s_3-2\g p+3\d p} (\t)
\\ & \qquad \qquad \qquad \qquad 
- \vth^{[1]}_{6p,3s_1+2s_2+s_3+4\g p} (\t) \vth_{3p,2s_2+s_3-2\g p+3\d p} (\t)  \bigg),\\
t_{2,\bm{s},\d} (\t) &:=
2 \vth_{2p,s_1+2s_2+s_3+2\d p} (\t)  \vth^{[1]}_{p,s_1-\d p} (\t)
- \vth^{[1]}_{2p,s_1+2s_2+s_3+2\d p} (\t)  \vth_{p,s_1-\d p} (\t)  .
\end{align*}
We start with 
\begin{align*}
t_{1,\bm{s},\d} (\t) &=  
3 \sum_{\g \in \{0,1,2\}}
\sum_{\bm{n} \in \Z^2 + \lp \frac{3s_1+2s_2+s_3+4\g p}{12p}, \frac{2s_2+s_3-2\g p+3\d p}{6p}  \rp }
(n_2-n_1) q^{3p\left(2n_1^2+n_2^2\right)},
\end{align*}
then shift $n_1\mapsto n_1-n_2$ and collapse the sum over $\g$ to obtain
\begin{align*}
t_{1,\bm{s},\d} (\t) &= 
\sum_{\bm{n} \in \Z^2 + \lp \frac{s_1+2s_2+s_3+2\d p}{4p}, \frac{2s_2+s_3+3\d p}{2p}   \rp }
(2n_2-3n_1) q^{p\left(6n_1^2-4n_1n_2+n_2^2\right)} .
\end{align*}
Then finally changing $n_2\mapsto 2n_1-n_2$ gives $t_{1,\bm{s},\d} (\t) = - t_{2,\bm{s},\d} (\t)$, thereby proving $I^{[3]}_{\bm{s}}  + J^{[3]}_{\bm{s}} = 0$.
\qedhere
\end{enumerate}
\end{proof}

\section{Higher Depth False Theta Functions from $\hat{Z}$-Invariants}\label{sec:Zhat_invariant}

As our final example, we consider higher rank generalizations of $\hat{Z}$-invariants \cite{GPPV} of plumbed $3$-manifolds that are recently proposed by Chung \cite{Chung} at least if the gauge group is  $G=\SU(n)$ with $n \geq 3$. Here we restrict our attention to the gauge group $\SU(3)$ and the case where the $3$-manifold in question is a Seifert fibration. Moreover, for simplicity of presentation, we only consider Brieskorn spheres so that there is only one homological block. Such a $3$-manifold is determined by a triple on integers $(P_1,P_2,P_3)$ where $P_j \geq 2$ are pairwise coprime. For instance, the ``smallest" triple $(2,3,5)$, which corresponds to a plumbing graph of type $E_8$, gives the celebrated Poincar\'e homology sphere.

We closely follow Chung's article \cite{Chung} where such invariants were explicitly computed. 
Let $M_3$ denote the Brieskorn sphere corresponding to the triple $(P_1,P_2,P_3)$ and let also $P=P_1 P_2 P_3$. We also assume
\begin{equation*} 
\frac{1}{P_1}+\frac{1}{P_2}+\frac{1}{P_3} <1.
\end{equation*}
According to  equation (3.15) in \cite{Chung}, the $\hat{Z}$-invariant of $M_3$ is given by the $q$-series
\begin{equation}\label{ZSU3}
\hat{Z}_{\SU(3)}(q)=
\sum_{\substack{n_{j,k} \geq 0 \\ 1 \leq j < k \leq 3}} \left( \prod_{1 \leq j < k \leq 3} \chi_{2P}(n_{j,k}) \right) 
q^{\frac{1}{8P} \left(2n_{1,2}^2+2n_{2,3}^2+2 n_{1,3}^2 +2n_{1,3} n_{2,3}+2n_{1,2} n_{2,3} -2n_{1,2} n_{1,3}\right)} .
\end{equation}
Here $\chi_{2P}(n)$ is defined as (see equation (2.17) of \cite{Chung})
\begin{equation*}
\chi_{2P}(n):=
\psi_{2P}^{\ell_1}(n) + \psi_{2P}^{\ell_2}(n) + \psi_{2P}^{\ell_3}(n) + \psi_{2P}^{\ell_4}(n),
\end{equation*}
where
\begin{equation*}
\psi_{2P}^{\ell}(n):=
\begin{cases}
\pm 1 \quad &  \mathrm{if}  \ n \equiv \pm \ell \pmod{2P},  \\
0 \quad & \mathrm{otherwise},
\end{cases}
\end{equation*}
and
\begin{align*} 
\ell_1:=P \left(1-\left(\frac{1}{P_1}+\frac{1}{P_2}+\frac{1}{P_3}\right) \right),   
& \qquad \ell_2:=P\left(1-\left(-\frac{1}{P_1}-\frac{1}{P_2}+\frac{1}{P_3}\right)\right), 
\\
\ell_3:=P \left(1-\left(\frac{1}{P_1}-\frac{1}{P_2}-\frac{1}{P_3}\right) \right) ,  
&  \qquad \ell_4:=P\left(1-\left(-\frac{1}{P_1}+\frac{1}{P_2}-\frac{1}{P_3}\right)\right) .
\end{align*}
Since $0 < \ell_j < 2P$ for all $j \in \{1,2,3,4\}$, we can rewrite \eqref{ZSU3} as 
\begin{equation*}
\hat{Z}_{\SU(3)}(q)
=
\sum_{\bm{a} \in \CA^3}  \sum_{\bm{s} \in \{-1,+1\}^3}  s_1 s_2 s_3
\sum_{\bm{n} \in \mathbb{N}^3 +  \lp \frac{1}{2} + s_1 a_1, \frac{1}{2} + s_2 a_2, \frac{1}{2} + s_3 a_3 \rp }
q^{\frac{P}{2} \lp  (n_1+n_2)^2 + (n_1 +n_3)^2 + (n_2-n_3)^2  \rp },
\end{equation*}
where
\begin{equation*}
\CA := \left\{ r_1,r_2,r_3,r_4 \right\} 
\quad \mbox{with} \  
r_j :=  \frac{\ell_j -P}{2P} .
\end{equation*}
The quadratic form appearing in this function is degenerate, so we can take the sum over one of the summation variables to get
\begin{align*}
\hat{Z}_{\SU(3)}(q) 
&=
\sum_{\bm{a} \in \CA^3}  \sum_{\bm{s} \in \{-1,1\}^3}  s_1 s_2 s_3
\sum_{m_1,m_2 \geq 0}  \min ( m_1,m_2) \,   q^{P Q_2 \lp  m_1+s_1a_1+s_2a_2, m_2+s_1a_1+s_3a_3 \rp } ,
\end{align*}
where $Q_2 (\bm{n}) = n_1^2 + n_2^2 - n_1 n_2$ is the quadratic form for the $A_2$-lattice as in Section \ref{sec:w_algebra_characters}. 
Now we note that $-1 < s_1a_1+s_2a_2, \, s_1a_1+s_3a_3 < 1$ and extend the sum over $\bm{m}$ to the entire $\IZ^2$ by inserting
\begin{align*}
\frac{1}{4} (1 + \sgn ( m_1+s_1a_1+s_2a_2) )  (1 + \sgn ( m_2+s_1a_1+s_3a_3 ) ) .
\end{align*}
Next we note that for $\bm{m} \in \IZ^2$ and any $-1 < \e < 1$ we have
\begin{equation*}
2 \min ( m_1,m_2) =  (m_1 +m_2) - \sgn (m_1-m_2 + \e) (m_1-m_2)
\end{equation*}
and use this identity with $\e := s_2 a_2 - s_3 a_3$ to get 
\begin{multline*}
\hat{Z}_{\SU(3)}(q) 
=
\frac{1}{8}
\sum_{\bm{a} \in \CA^3}  \sum_{\bm{s} \in \{\pm 1\}^3}  s_1 s_2 s_3
\sum_{\bm{m} \in \IZ^2 + \pmat{s_1a_1+s_2a_2 \\ s_1a_1+s_3a_3} } 
(1 + \sgn (m_1)  (1 + \sgn (m_2) )\\
\times ((m_1+m_2 - 2s_1 a_1 - s_2a_2-s_3 a_3)  - \sgn(m_1-m_2)(m_1 - m_2  - s_2a_2 + s_3 a_3) ) q^{P Q_2 (\bm{m} )} .
\end{multline*}
Changing variables as $\bm{m} \mapsto - \bm{m}$ and $\bm{s} \mapsto - \bm{s}$ in one half of the sum, we find that
\begin{align*}
\hat{Z}_{\SU(3)}(q) 
=
\frac{1}{8}
\sum_{\bm{a} \in \CA^3}  &\sum_{\bm{s} \in \{\pm 1\}^3}  s_1 s_2 s_3
\sum_{\bm{m} \in \IZ^2 + \pmat{s_1a_1+s_2a_2 \\ s_1a_1+s_3a_3} }  
q^{P Q_2 (\bm{m} )} 
\\
& \times
( (1 + \sgn (m_1)  \sgn (m_2) ) (m_1 +m_2 - 2s_1 a_1 - s_2a_2-s_3 a_3)
\\
& \qquad \qquad
- \sgn (m_1 - m_2)   (\sgn (m_1) + \sgn (m_2) ) (m_1 - m_2  - s_2a_2 + s_3 a_3)) .
\end{align*}
On the summand involving the factor $\sgn(m_1-m_2) \sgn(m_1)$ we apply the transformation
\begin{equation*}
g: \bm{m} \mapsto (m_1, m_1-m_2), \quad
\bm{s} \mapsto (s_2,s_1,-s_3), 
\quad \bm{a} \mapsto (a_2,a_1,a_3),
\end{equation*}
on the summand involving the factor $\sgn(m_1-m_2) \sgn(m_2)$ we apply 
\begin{equation*}
g': \bm{m} \mapsto (m_2-m_1, m_2), \quad
\bm{s} \mapsto (s_3,-s_2,s_1), 
\quad \bm{a} \mapsto (a_3,a_2,a_1),
\end{equation*}
and we average the summand with no sign function factors over the transformations $(\mathrm{Id}, g, g')$. This yields
\begin{align*}
\hat{Z}_{\SU(3)}(q) 
= \frac{1}{4}
\sum_{\bm{a} \in \CA^3}  \sum_{\bm{s} \in \{\pm 1\}^3}  s_1 s_2 s_3 \, G_{(s_1a_1+s_2a_2, s_1a_1+s_3a_3)} (\t),
\end{align*}
where (adopting notation from Section \ref{sec:w_algebra_characters})
\begin{align*}
G_{\bm{\mu}} (\t) := 
\sum_{\bm{m} \in \IZ^2 + \bm{\mu}}  
\sgn (m_1) \sgn (m_2) B_2 (\bm{\rho}, \bm{m} - \bm{\mu} ) q^{P Q_2 (\bm{m}) }.
\end{align*}
Noting that $G_{\bm{\mu}} (\t) = - G_{-\bm{\mu}} (\t)$ (as can be seen by changing $\bm{m} \mapsto - \bm{m}$) and $G_{\bm{\mu}} (\t) = G_{(\mu_2,\mu_1)} (\t)$ (by changing $\bm{m} \mapsto (m_2,m_1)$), we can take the sum over $\bm{s}$ to get
\begin{align*}
\hat{Z}_{\SU(3)}(q)  = \frac{1}{2}  
\sum_{\bm{a} \in \CA^3} \left(G_{(a_1+a_2,a_1+a_3)} (\t) - 2 G_{(a_1+a_2,a_1-a_3)} (\t)  + G_{(a_1-a_2,a_1-a_3)} (\t)\right).
\end{align*}
Now looking at Propositions \ref{prop:psiA2_properties} and \ref{prop:PhiA2}, we find (for the relevant values of $\bm{\mu}$)
\begin{equation*}
G_{\bm{\mu}} (\t) = \varphi_{P,\bm{S} (\bm{\mu})}^{[A_2]} (\t)
- B_2 (\bm{\rho}, \bm{\mu} ) \,
\psi_{P,\bm{S} (\bm{\mu})}^{[A_2]} (\t)
\where
\bm{S} (\bm{\mu}) := 
P(B_2(\bm{\mu}, \bm{\a_1}),B_2(\bm{\mu}, \bm{\a_2})) .
\end{equation*}
In summary, we explicitly see that these $\SU(3)$ $\hat{Z}$-invariants are linear combinations of depth two false modular forms $\psi^{[A_2]}_{p, \bm{s}} (\t)$ and $\varphi_{p,\bm{s}}^{[A_2]} (\t)$ studied in Section \ref{sec:w_algebra_characters}.

\begin{rem}
Recall that these two false modular forms are the same ones which appeared in the discussion of the characters of $W^0(P)_{A_2}$ vertex algebras. In fact, we can further clarify the relation between the $\SU(3)$ $\hat{Z}$-invariants here and $W^0(P)_{A_2}$ characters as follows: As noted in the discussion of Subsections \ref{sec:generalities_W_algebra} and \ref{sec:vac_character_A1_A2}, these characters look like (when $\bm{\wh{\l}} = \bm{\g} = \bm{0}$ and up to possible corrections involving sums over lower rank lattices)
\begin{equation*}
\frac{1}{4 \eta (\t)^2}  \mathbb{G}_{\frac{\bm{\bar{\l}}}{\sqrt{p}} -  \frac{\bm{\rho}}{p}} (\t)
\where \ 
\mathbb{G}_{\bm{\mu}} (\t) := 
\sum_{w \in W} (-1)^{\ell (w)} G_{w ( \bm{\mu} )} (\t) .
\end{equation*}
The $\SU(3)$ $\hat{Z}$-invariants considered above can be similarly written as 
\begin{align*}
\hat{Z}_{\SU(3)}(q)  &= 
\frac{1}{4}  \sum_{\bm{a} \in \{r_1,r_2,r_3,r_4\}^3}  \mathbb{G}_{(a_1+a_2,a_1+a_3)} (\t) 
\\ & \qquad
+\frac{1}{2} \lp \mathbb{G}_{(r_1-r_2,r_1-r_3)} (\t) + \mathbb{G}_{(r_1-r_2,r_1-r_4)} (\t)
+ \mathbb{G}_{(r_1-r_3,r_1-r_4)} (\t)  + \mathbb{G}_{(r_2-r_3,r_2-r_4)} (\t)  \rp .
\end{align*}
\end{rem}

\appendix
\section{A Lemma on the Kostant Partition Function of $A_3$}\label{sec:Kostant_A3_lemma}

In this section we give a proof of Lemma \ref{la:KostkaRewrite}.

\begin{proof}[Proof of Lemma \ref{la:KostkaRewrite}]
We start with the right-hand side and note that it can be explicitly expanded using the values of $(-1)^{\ell (w)} \mathbb{P}_3 ( w (\bm{\a}) )$ given in the following table as $w$ runs over the Weyl group $S_4$ of $A_3$ (where we denote $w$ as a permutation).

\begin{table}[h!]
\centering
\small
\begin{tabular}{c | c}
\toprule
$w \in W = S_4$   & $(-1)^{\ell (w)} \mathbb{P}_3 ( w (\bm{\a}) )$  \\ 
\hline 
$e, (14)(23)$ & $\hspace{-2em}\! \sgn ( n_1) \sgn (n_2) \sgn (n_3) \, n_2 (4 n_1 n_3 - 1)$ \\
$(23), (14)$ & $\hspace{-2em}\!\! -\sgn (n_1) \sgn (n_1-n_2+n_3)  \sgn (n_3) (n_1-n_2+n_3) (4 n_1 n_3 - 1)$ \\
$(234), (143)$ & $\!\!\!\!\! \sgn (n_1) \sgn (n_1-n_3) \sgn (n_2-n_3) (n_1-n_3) (4 n_1 (n_2-n_3) - 1)$ \\
$(1234), (13)$ & $\!\!\!\!\! \sgn (n_3) \sgn (n_1-n_3) \sgn (n_2-n_3) (n_1-n_3) (4 n_3 (n_3-n_2) - 1)$ \\
$(243), (142)$ & $\!\!\!\!\!  \sgn (n_1) \sgn (n_1-n_2+n_3) \sgn (n_1-n_2) (n_1-n_2+n_3) (4 n_1 (n_1-n_2) - 1)$ \\
$(12), (1324)$ & $\hspace{-1em}\!\!\! \sgn (n_1-n_2) \sgn (n_2) \sgn (n_3) \, n_2 (4 n_3 (n_2-n_1) - 1)$ \\
$(12)(34), (13)(24)$ & $\!\!\!\! \sgn (n_1-n_2) \sgn (n_2) \sgn (n_2-n_3) \, n_2 (4 (n_1-n_2) (n_2-n_3) + 1)$ \\
$(1243), (1342)$ & $\sgn (n_2-n_3) \sgn (n_1-n_2+n_3) \sgn (n_1-n_2) (n_1-n_2+n_3) (4 (n_2-n_1) (n_2-n_3) - 1)$ \\
$(123), (134)$ & $\!\!\!\!\! \sgn (n_2-n_3) \sgn (n_1-n_2+n_3) \sgn (n_3) (n_1-n_2+n_3) (4 (n_2-n_3) n_3 + 1)$ \\
$(34), (1423)$ & $\hspace{-1em} \!\!\! \sgn (n_1) \sgn (n_2) \sgn (n_2-n_3) \, n_2 (4 (n_3-n_2) n_1 + 1)$ \\
$(124), (132)$ & $\!\!\!\!\! \sgn (n_3) \sgn (n_1-n_3) \sgn (n_1-n_2) (n_1-n_3) (4 n_3 (n_1-n_2) + 1)$ \\
$(24), (1432)$ & $\!\!\!\!\! \sgn (n_1) \sgn (n_1-n_3) \sgn (n_1-n_2) (n_1-n_3) (4 n_1 (n_2-n_1) + 1)$ \\
\bottomrule
\end{tabular}
\end{table}

Since both sides of the claimed identity are odd with respect to the action of the Weyl reflections, it is enough to prove the result when $\bm{\a}$ is restricted to the fundamental Weyl chamber associated with the simple roots $\bm{\a_1}$, $\bm{\a_2}$, $\bm{\a_3}$. Moreover, for any boundary component of a Weyl chamber, there is a Weyl reflection leaving it invariant. Therefore, both sides of our expression vanish on the boundaries and we can restrict our attention to the interior of the fundamental Weyl chamber.
Recall that given any $\bm{\a} = n_1 \bm{\a_1} + n_2 \bm{\a_2} + n_3 \bm{\a_3}$, the interior of the fundamental Weyl chamber is the region where $B_3 (\bm{\a} , \bm{\a_j} ) > 0$ for all $j \in \{1,2,3\}$, i.e., it is where we have
\begin{equation*}
2n_1 - n_2 > 0, \quad
-n_1 + 2n_2 -n_3 > 0, \quad
-n_2 + 2n_3 > 0.
\end{equation*} 
Noting that $n_1, n_2, n_3, n_1-n_2+n_3 > 0$ in this region, the right-hand side can be simplified as
\begin{align*}
&\frac{1}{16} \sum_{w \in W} (-1)^{\ell (w)} \mathbb{P}_3 ( w (\bm{\a}) )
= P_1 (\bm{n}) + \sgn(n_1-n_3) \sgn(n_2-n_3) P_2 (\bm{n})
+ \sgn(n_1-n_2) P_3 (\bm{n})
\\
& \qquad
+ \sgn (n_1-n_2) \sgn (n_2-n_3) P_4 (\bm{n}) + \sgn (n_2-n_3) P_5 (\bm{n})
+ \sgn (n_1-n_2) \sgn (n_1-n_3) P_6 (\bm{n}),
\end{align*}
where
\begin{align*}
P_1 (\bm{n}) &:= \frac{1}{8} (2n_2-n_1-n_3)(4 n_1 n_3 - 1), \quad &
P_2 (\bm{n}) &:= \frac{1}{8} (n_1-n_3) (4 (n_1-n_3) (n_2-n_3) -2), \\
P_3 (\bm{n}) &:= \frac{1}{8} (n_1+n_3)(4(n_1-n_2)^2-1), \quad &
P_4 (\bm{n}) &:= \frac{1}{8} (2n_2-n_1-n_3)(4 (n_1-n_2) (n_2-n_3) + 1), \\
P_5 (\bm{n}) &:= - \frac{1}{8} (n_1+n_3)( 4(n_2-n_3)^2-1 ), \quad &
P_6 (\bm{n}) &:= - \frac{1}{8} (n_1-n_3) (4(n_1-n_2)(n_1-n_3)-2).
\end{align*}
Then considering all the possibilities for the ordering of $n_1,n_2,n_3$, we find that the right-hand side becomes the following piecewise polynomial in the interior of the fundamental Weyl chamber
\begin{align*}
\frac{1}{16} \sum_{w \in W} (-1)^{\ell (w)} \mathbb{P}_3 ( w (\bm{\a}) )  \! = \!
\begin{cases}
P_1 (\bm{n}) + P_2(\bm{n}) + P_3 (\bm{n}) + P_4 (\bm{n}) + P_5 (\bm{n}) + P_6 (\bm{n})
&  \!\!\!\mbox{if } n_1 > n_2 , \\
P_1 (\bm{n}) + P_2(\bm{n}) - P_3 (\bm{n}) + P_4 (\bm{n}) - P_5 (\bm{n}) + P_6 (\bm{n})
& \!\!\! \mbox{if } n_3> n_2 , \\
P_1 (\bm{n}) + P_2(\bm{n}) - P_3 (\bm{n}) - P_4 (\bm{n}) + P_5 (\bm{n}) - P_6 (\bm{n})
&  \!\!\! \mbox{if } n_2 > n_1 \geq n_3, \\
P_1 (\bm{n}) - P_2(\bm{n}) - P_3 (\bm{n}) - P_4 (\bm{n}) + P_5 (\bm{n}) + P_6 (\bm{n})
&  \!\!\! \mbox{if } n_2 > n_3 \geq n_1 .\\
\end{cases}
\end{align*}

Turning to the left-hand side of the claimed identity, Theorem 6.1 of \cite{KP} shows that it is equal to the following expression in the interior of the fundamental Weyl chamber
\begin{align*}
\sum_{w \in W} (-1)^{\ell (w)} K_3 \lp w ( \bm{\a} ) -  \bm{\rho} \rp  \! = \!
\begin{cases}
\frac{1}{2} (2n_2-n_1-n_3)(2n_3-n_2)(n_2+n_3-n_1)
&  \!\!\!\mbox{if } n_1 > n_2 , \\
\frac{1}{2} (2n_1-n_2)(2n_2-n_1-n_3)(n_1+n_2-n_3)
& \!\!\! \mbox{if } n_3> n_2 , \\
\frac{1}{2} (n_2 - 2n_3) \lp 3n_1^2 + 2n_2^2+n_3^2 - 5n_1n_2-n_2n_3-1 \rp 
&  \!\!\! \mbox{if } n_2 > n_1 \geq n_3, \\
\frac{1}{2} (n_2 - 2n_1) \lp n_1^2 + 2n_2^2+3n_3^2 - n_1n_2-5n_2n_3-1 \rp 
&  \!\!\! \mbox{if } n_2 > n_3 \geq n_1 .\\
\end{cases}
\end{align*}
Finally, the equality of these two expressions immediately follow from the equality of the stated polynomials in the relevant regions.
\end{proof}

A small variation of Lemma \ref{la:KostkaRewrite} replaces the linear factor in $\mathbb{P}_3 ( w (\bm{\a}) )$ as $n_2\mapsto\frac12(n_1+n_2+n_3)$ with the conclusion of the lemma still holding. This is thanks to the following result.

\begin{lem}\label{lem:konstant_A3_average2}
For $\mathbb{T}_3 (n_1 \bm{\a_1} + n_2 \bm{\a_2} + n_3 \bm{\a_3} ) :=   
 \sgn ( n_1 ) \, \sgn ( n_2 ) \, \sgn ( n_3 ) (n_1 - n_2 + n_3)$ and any $\bm{\a} \in A_3 + \bm{\rho}$ we have
\begin{align*}
\sum_{w \in W} (-1)^{\ell (w)} \mathbb{T}_3 ( w (\bm{\a}) ) = 0 .
\end{align*}
\end{lem}

\begin{proof}
Following the proof of Lemma \ref{lem:konstant_A3_average2}, it is enough to prove the statement in the interior of the fundamental Weyl chamber. Here we have
\begin{align*}
\frac{1}{2} \sum_{w \in W} (-1)^{\ell (w)} \mathbb{T}_3 ( w (\bm{\a}) )  \! = \!
\begin{cases}
T_1 (\bm{n}) + T_2(\bm{n}) + T_3 (\bm{n}) + T_4 (\bm{n}) + T_5 (\bm{n}) + T_6 (\bm{n})
&  \!\!\!\mbox{if } n_1 > n_2 , \\
T_1 (\bm{n}) + T_2(\bm{n}) - T_3 (\bm{n}) + T_4 (\bm{n}) - T_5 (\bm{n}) + T_6 (\bm{n})
& \!\!\! \mbox{if } n_3> n_2 , \\
T_1 (\bm{n}) + T_2(\bm{n}) - T_3 (\bm{n}) - T_4 (\bm{n}) + T_5 (\bm{n}) - T_6 (\bm{n})
&  \!\!\! \mbox{if } n_2 > n_1 \geq n_3, \\
T_1 (\bm{n}) - T_2(\bm{n}) - T_3 (\bm{n}) - T_4 (\bm{n}) + T_5 (\bm{n}) + T_6 (\bm{n})
&  \!\!\! \mbox{if } n_2 > n_3 \geq n_1 ,\\
\end{cases}
\end{align*}
where
\begin{align*}
T_3 (\bm{n}) := T_5 (\bm{n}) := 0
\quad \mbox{ and } \quad
T_1 (\bm{n}) := - T_2 (\bm{n}) := T_4 (\bm{n}) := -T_6 (\bm{n}) := n_1 - 2n_2 + n_3 .
\end{align*}
The lemma statement immediately follows from these values.
\end{proof}

\addtocontents{toc}{\SkipTocEntry}

\end{document}